\documentclass{tac}
\newif\iftac\tactrue
\newif\ifcref\creftrue

\usepackage{amssymb,amsmath,stmaryrd,mathrsfs}

\def\definetac{\newif\iftac}    % Can't define a \newif inside another \if!
\ifx\tactrue\undefined
  \definetac
  \ifx\state\undefined\tacfalse\else\tactrue\fi
\fi

\def\definebeamer{\newif\ifbeamer}
\ifx\beamertrue\undefined
  \definebeamer
  \ifx\uncover\undefined\beamerfalse\else\beamertrue\fi
\fi

\def\definecref{\newif\ifcref}
\ifx\creftrue\undefined
  \definecref
  \creffalse
\fi

\iftac\else\usepackage{amsthm}\fi
\usepackage[all]{xy}
\usepackage{tikz,tikz-cd}
\usetikzlibrary{arrows}
\ifbeamer\else
  \usepackage{enumitem}
  \usepackage{xcolor}
  \definecolor{darkgreen}{rgb}{0,0.45,0} 
  \iftac\else\usepackage[pagebackref,colorlinks,citecolor=darkgreen,linkcolor=darkgreen]{hyperref}
  \renewcommand*{\backref}[1]{}
  \renewcommand*{\backrefalt}[4]{({%
      \ifcase #1 Not cited.%
            \or On p.~#2%
            \else On pp.~#2%
      \fi%
    })}\fi
\fi
\usepackage{mathtools}          % for all sorts of things
\usepackage{graphics}           % for \scalebox, used in \widecheck
\usepackage{ifmtarg}            % used in \jd
\usepackage{microtype}
\usepackage{wasysym}
\usepackage{braket}             % for \Set, etc.

\usepackage{url}                % for citations to web sites
\usepackage{xspace}             % put spaces after a \command in text
\ifcref\usepackage{cleveref,aliascnt}\fi
\usepackage[status=draft,author=]{fixme}
\fxusetheme{color}
\usepackage{mathpartir}
\usepackage{cmll,tensor}

\makeatletter
\let\ea\expandafter

\def\mdef#1#2{\ea\ea\ea\gdef\ea\ea\noexpand#1\ea{\ea\ensuremath\ea{#2}\xspace}}
\def\alwaysmath#1{\ea\ea\ea\global\ea\ea\ea\let\ea\ea\csname your@#1\endcsname\csname #1\endcsname
  \ea\def\csname #1\endcsname{\ensuremath{\csname your@#1\endcsname}\xspace}}

\DeclareRobustCommand\widecheck[1]{{\mathpalette\@widecheck{#1}}}
\def\@widecheck#1#2{%
    \setbox\z@\hbox{\m@th$#1#2$}%
    \setbox\tw@\hbox{\m@th$#1%
       \widehat{%
          \vrule\@width\z@\@height\ht\z@
          \vrule\@height\z@\@width\wd\z@}$}%
    \dp\tw@-\ht\z@
    \@tempdima\ht\z@ \advance\@tempdima2\ht\tw@ \divide\@tempdima\thr@@
    \setbox\tw@\hbox{%
       \raise\@tempdima\hbox{\scalebox{1}[-1]{\lower\@tempdima\box
\tw@}}}%
    {\ooalign{\box\tw@ \cr \box\z@}}}

\newcount\foreachcount

\def\foreachletter#1#2#3{\foreachcount=#1
  \ea\loop\ea\ea\ea#3\@alph\foreachcount
  \advance\foreachcount by 1
  \ifnum\foreachcount<#2\repeat}

\def\foreachLetter#1#2#3{\foreachcount=#1
  \ea\loop\ea\ea\ea#3\@Alph\foreachcount
  \advance\foreachcount by 1
  \ifnum\foreachcount<#2\repeat}

\def\definescr#1{\ea\gdef\csname s#1\endcsname{\ensuremath{\mathscr{#1}}\xspace}}
\foreachLetter{1}{27}{\definescr}
\def\definecal#1{\ea\gdef\csname c#1\endcsname{\ensuremath{\mathcal{#1}}\xspace}}
\foreachLetter{1}{27}{\definecal}
\def\definebold#1{\ea\gdef\csname b#1\endcsname{\ensuremath{\mathbf{#1}}\xspace}}
\foreachLetter{1}{27}{\definebold}
\def\definebb#1{\ea\gdef\csname d#1\endcsname{\ensuremath{\mathbb{#1}}\xspace}}
\foreachLetter{1}{27}{\definebb}
\def\definefrak#1{\ea\gdef\csname f#1\endcsname{\ensuremath{\mathfrak{#1}}\xspace}}
\foreachletter{1}{9}{\definefrak}
\foreachletter{10}{27}{\definefrak}
\foreachLetter{1}{27}{\definefrak}
\def\definesf#1{\ea\gdef\csname i#1\endcsname{\ensuremath{\mathsf{#1}}\xspace}}
\foreachletter{1}{6}{\definesf}
\foreachletter{7}{14}{\definesf}
\foreachletter{15}{20}{\definesf}
\foreachletter{21}{27}{\definesf}
\foreachLetter{1}{27}{\definesf}
\def\definebar#1{\ea\gdef\csname #1bar\endcsname{\ensuremath{\overline{#1}}\xspace}}
\foreachLetter{1}{27}{\definebar}
\foreachletter{1}{8}{\definebar} % \hbar is something else!
\foreachletter{9}{15}{\definebar} % \obar is something else!
\foreachletter{16}{27}{\definebar}
\def\definetil#1{\ea\gdef\csname #1til\endcsname{\ensuremath{\widetilde{#1}}\xspace}}
\foreachLetter{1}{27}{\definetil}
\foreachletter{1}{27}{\definetil}
\def\definehat#1{\ea\gdef\csname #1hat\endcsname{\ensuremath{\widehat{#1}}\xspace}}
\foreachLetter{1}{27}{\definehat}
\foreachletter{1}{27}{\definehat}
\def\definechk#1{\ea\gdef\csname #1chk\endcsname{\ensuremath{\widecheck{#1}}\xspace}}
\foreachLetter{1}{27}{\definechk}
\foreachletter{1}{27}{\definechk}
\def\defineul#1{\ea\gdef\csname u#1\endcsname{\ensuremath{\underline{#1}}\xspace}}
\foreachLetter{1}{27}{\defineul}
\foreachletter{1}{27}{\defineul}

\def\autofmt@n#1\autofmt@end{\mathrm{#1}}
\def\autofmt@b#1\autofmt@end{\mathbf{#1}}
\def\autofmt@d#1#2\autofmt@end{\mathbb{#1}\mathsf{#2}}
\def\autofmt@c#1#2\autofmt@end{\mathcal{#1}\mathit{#2}}
\def\autofmt@s#1#2\autofmt@end{\mathscr{#1}\mathit{#2}}
\def\autofmt@f#1\autofmt@end{\mathsf{#1}}
\def\autofmt@k#1\autofmt@end{\mathfrak{#1}}
\def\autofmt@u#1\autofmt@end{\underline{\smash{\mathsf{#1}}}}
\def\autofmt@U#1\autofmt@end{\underline{\underline{\smash{\mathsf{#1}}}}}
\def\autofmt@h#1\autofmt@end{\widehat{#1}}
\def\autofmt@r#1\autofmt@end{\overline{#1}}
\def\autofmt@t#1\autofmt@end{\widetilde{#1}}
\def\autofmt@k#1\autofmt@end{\check{#1}}

\def\auto@drop#1{}
\def\autodef#1{\ea\ea\ea\@autodef\ea\ea\ea#1\ea\auto@drop\string#1\autodef@end}
\def\@autodef#1#2#3\autodef@end{%
  \ea\def\ea#1\ea{\ea\ensuremath\ea{\csname autofmt@#2\endcsname#3\autofmt@end}\xspace}}
\def\autodefs@end{blarg!}
\def\autodefs#1{\@autodefs#1\autodefs@end}
\def\@autodefs#1{\ifx#1\autodefs@end%
  \def\autodefs@next{}%
  \else%
  \def\autodefs@next{\autodef#1\@autodefs}%
  \fi\autodefs@next}

\DeclareSymbolFont{bbold}{U}{bbold}{m}{n}
\DeclareSymbolFontAlphabet{\mathbbb}{bbold}

\newcommand{\dtwo}{\ensuremath{\mathbbb{2}}\xspace}

\mdef\delbar{\overline{\partial}}

\newcommand{\dual}{^{\vee}}
\mdef\hf{\textstyle\frac12 }
\mdef\thrd{\textstyle\frac13 }
\mdef\qtr{\textstyle\frac14 }

\newcommand{\op}{^{\mathrm{op}}}
\newcommand{\co}{^{\mathrm{co}}}

\mdef\Id{\mathrm{Id}}
\mdef\id{\mathrm{id}}
\alwaysmath{ell}
\alwaysmath{infty}

\alwaysmath{odot}
\def\frc#1/#2.{\frac{#1}{#2}}   % \frc x^2+1 / x^2-1 .
\mdef\ten{\mathrel{\otimes}}

\mdef\sqten{\mathrel{\boxtimes}}
\DeclareFontFamily{U}{min}{}
\DeclareFontShape{U}{min}{m}{n}{<-> udmj30}{}
\newcommand{\yon}{\!\text{\usefont{U}{min}{m}{n}\symbol{'210}}\!}

\DeclareFontFamily{U}{mathb}{\hyphenchar\font45}
\DeclareFontShape{U}{mathb}{m}{n}{
      <5> <6> <7> <8> <9> <10> gen * mathb
      <10.95> mathb10 <12> <14.4> <17.28> <20.74> <24.88> mathb12
      }{}
\DeclareSymbolFont{mathb}{U}{mathb}{m}{n}
\DeclareFontSubstitution{U}{mathb}{m}{n}
\DeclareMathSymbol{\dotplus}       {2}{mathb}{"00}% name to be checked
\DeclareMathSymbol{\dotdiv}        {2}{mathb}{"01}% name to be checked
\DeclareMathSymbol{\dottimes}      {2}{mathb}{"02}% name to be checked
\DeclareMathSymbol{\divdot}        {2}{mathb}{"03}% name to be checked
\DeclareMathSymbol{\udot}          {2}{mathb}{"04}% name to be checked
\DeclareMathSymbol{\square}        {2}{mathb}{"05}% name to be checked
\DeclareMathSymbol{\Asterisk}      {2}{mathb}{"06}
\DeclareMathSymbol{\bigast}        {1}{mathb}{"06}
\DeclareMathSymbol{\coAsterisk}    {2}{mathb}{"07}
\DeclareMathSymbol{\bigcoast}      {1}{mathb}{"07}
\DeclareMathSymbol{\circplus}      {2}{mathb}{"08}% name to be checked
\DeclareMathSymbol{\pluscirc}      {2}{mathb}{"09}% name to be checked
\DeclareMathSymbol{\convolution}   {2}{mathb}{"0A}% name to be checked
\DeclareMathSymbol{\divideontimes} {2}{mathb}{"0B}% name to be checked
\DeclareMathSymbol{\blackdiamond}  {2}{mathb}{"0C}% name to be checked
\DeclareMathSymbol{\sqbullet}      {2}{mathb}{"0D}% name to be checked
\DeclareMathSymbol{\bigstar}       {2}{mathb}{"0E}
\DeclareMathSymbol{\bigvarstar}    {2}{mathb}{"0F}
\DeclareMathSymbol{\corresponds}   {3}{mathb}{"1D}% name to be checked
\DeclareMathSymbol{\updownarrows}          {3}{mathb}{"D6}
\DeclareMathSymbol{\downuparrows}          {3}{mathb}{"D7}
\DeclareMathSymbol{\Lsh}                   {3}{mathb}{"E8}
\DeclareMathSymbol{\Rsh}                   {3}{mathb}{"E9}
\DeclareMathSymbol{\dlsh}                  {3}{mathb}{"EA}
\DeclareMathSymbol{\drsh}                  {3}{mathb}{"EB}
\DeclareMathSymbol{\looparrowdownleft}     {3}{mathb}{"EE}
\DeclareMathSymbol{\looparrowdownright}    {3}{mathb}{"EF}
\DeclareMathSymbol{\curvearrowleftright}   {3}{mathb}{"F2}
\DeclareMathSymbol{\curvearrowbotleft}     {3}{mathb}{"F3}
\DeclareMathSymbol{\curvearrowbotright}    {3}{mathb}{"F4}
\DeclareMathSymbol{\curvearrowbotleftright}{3}{mathb}{"F5}
\DeclareMathSymbol{\leftsquigarrow}        {3}{mathb}{"F8}
\DeclareMathSymbol{\rightsquigarrow}       {3}{mathb}{"F9}
\DeclareMathSymbol{\leftrightsquigarrow}   {3}{mathb}{"FA}
\DeclareMathSymbol{\lefttorightarrow}      {3}{mathb}{"FC}
\DeclareMathSymbol{\righttoleftarrow}      {3}{mathb}{"FD}
\DeclareMathSymbol{\uptodownarrow}         {3}{mathb}{"FE}
\DeclareMathSymbol{\downtouparrow}         {3}{mathb}{"FF}
\DeclareMathSymbol{\varhash}       {0}{mathb}{"23}

\newcommand{\too}[1][]{\ensuremath{\overset{#1}{\longrightarrow}}}

\let\toot\rightleftarrows
\let\otto\leftrightarrows

\let\into\hookrightarrow

\mdef\we{\overset{\sim}{\longrightarrow}}
\mdef\leftwe{\overset{\sim}{\longleftarrow}}

\let\xto\xrightarrow

\def\rightarrowtailfill@{\arrowfill@{\Yright\joinrel\relbar}\relbar\rightarrow}
\newcommand\xrightarrowtail[2][]{\ext@arrow 0055{\rightarrowtailfill@}{#1}{#2}}

\def\twoheadrightarrowfill@{\arrowfill@{\relbar\joinrel\relbar}\relbar\twoheadrightarrow}
\newcommand\xtwoheadrightarrow[2][]{\ext@arrow 0055{\twoheadrightarrowfill@}{#1}{#2}}

\def\slashedarrowfill@#1#2#3#4#5{%
  $\m@th\thickmuskip0mu\medmuskip\thickmuskip\thinmuskip\thickmuskip
   \relax#5#1\mkern-7mu%
   \cleaders\hbox{$#5\mkern-2mu#2\mkern-2mu$}\hfill
   \mathclap{#3}\mathclap{#2}%
   \cleaders\hbox{$#5\mkern-2mu#2\mkern-2mu$}\hfill
   \mkern-7mu#4$%
}
\def\rightslashedarrowfill@{%
  \slashedarrowfill@\relbar\relbar\mapstochar\rightarrow}
\newcommand\xslashedrightarrow[2][]{%
  \ext@arrow 0055{\rightslashedarrowfill@}{#1}{#2}}
\mdef\hto{\xslashedrightarrow{}}
\mdef\htoo{\xslashedrightarrow{\quad}}

\def\toiso{\xto{\smash{\raisebox{-.5mm}{$\scriptstyle\sim$}}}}

\def\jd#1{\@jd#1\ej}
\def\@jd#1|-#2\ej{\@@jd#1,,\;\vdash\;\left(#2\right)}
\def\@@jd#1,{\@ifmtarg{#1}{\let\next=\relax}{\left(#1\right)\let\next=\@@@jd}\next}
\def\@@@jd#1,{\@ifmtarg{#1}{\let\next=\relax}{,\,\left(#1\right)\let\next=\@@@jd}\next}
\def\jdm#1{\@jdm#1\ej}
\def\@jdm#1|-#2\ej{\@@jd#1,,\\\vdash\;\left(#2\right)}
\long\def\my@drawfill#1#2;{%
\@skipfalse
\fill[#1,draw=none] #2;
\@skiptrue
\draw[#1,fill=none] #2;
}
\newif\if@skip
\newcommand{\skipit}[1]{\if@skip\else#1\fi}
\newcommand{\drawfill}[1][]{\my@drawfill{#1}}

\newcounter{nodemaker}
\newif\ifhyperref
\@ifpackageloaded{hyperref}{\hyperreftrue}{\hyperreffalse}
\iftac
  \let\your@state\state
  \def\state#1{\my@state#1}
  \def\my@state#1.{\gdef\currthmtype{#1}\your@state{#1.}}
  \let\your@staterm\staterm
  \def\staterm#1{\my@staterm#1}
  \def\my@staterm#1.{\gdef\currthmtype{#1}\your@staterm{#1.}}
  \let\@defthm\newtheorem
  \def\switchtotheoremrm{\let\@defthm\newtheoremrm}
  \def\defthm#1#2#3{\@defthm{#1}{#2}} % Ignore the third argument (for cleveref only)
  \let\your@section\section
  \def\section{\gdef\currthmtype{section}\your@section}
  \def\currthmtype{}
  \ifhyperref
    \def\autoref#1{\ref*{label@name@#1}~\ref{#1}}
  \else
    \def\autoref#1{\ref{label@name@#1}~\ref{#1}}
  \fi
  \AtBeginDocument{%
    \let\old@label\label%
    \def\label#1{%
      {\let\your@currentlabel\@currentlabel%
        \edef\@currentlabel{\currthmtype}%
        \old@label{label@name@#1}}%
      \old@label{#1}}
  }
  \let\cref\autoref
\else\ifcref
  \def\defthm#1#2#3{%
    \newaliascnt{#1}{thm}
    \newtheorem{#1}[#1]{#2}
    \aliascntresetthe{#1}
    \crefname{#1}{#2}{#3}% following brace must be on separate line to support poorman cleveref sed file
  }
\else
  \ifhyperref
    \def\defthm#1#2#3{% Ignore the third argument (for cleveref only)
      \newtheorem{#1}{#2}[section]%
      \expandafter\def\csname #1autorefname\endcsname{#2}%
      \expandafter\let\csname c@#1\endcsname\c@thm}
  \else
    \def\defthm#1#2#3{\newtheorem{#1}[thm]{#2}} % Ignore the third argument (for cleveref only)
    \ifx\SK@label\undefined\let\SK@label\label\fi
    \let\old@label\label
    \let\your@thm\@thm
    \def\@thm#1#2#3{\gdef\currthmtype{#3}\your@thm{#1}{#2}{#3}}
    \def\currthmtype{}
    \def\label#1{{\let\your@currentlabel\@currentlabel\def\@currentlabel%
        {\currthmtype~\your@currentlabel}%
        \SK@label{#1@}}\old@label{#1}}
    \def\autoref#1{\ref{#1@}}
  \fi
  \let\cref\autoref
\fi\fi

\newtheorem{thm}{Theorem}[section]
\ifcref
  \crefname{thm}{Theorem}{Theorems}
\else
  
\fi
\defthm{cor}{Corollary}{Corollaries}
\defthm{prop}{Proposition}{Propositions}
\defthm{lem}{Lemma}{Lemmas}
\defthm{sch}{Scholium}{Scholia}
\defthm{assume}{Assumption}{Assumptions}
\defthm{claim}{Claim}{Claims}
\defthm{conj}{Conjecture}{Conjectures}
\defthm{hyp}{Hypothesis}{Hypotheses}
\iftac\switchtotheoremrm\else\theoremstyle{definition}\fi
\defthm{defn}{Definition}{Definitions}
\defthm{notn}{Notation}{Notations}
\iftac\switchtotheoremrm\else\theoremstyle{remark}\fi
\defthm{rmk}{Remark}{Remarks}
\defthm{eg}{Example}{Examples}
\defthm{egs}{Examples}{Examples}
\defthm{ex}{Exercise}{Exercises}
\defthm{ceg}{Counterexample}{Counterexamples}

\ifcref
  \crefformat{section}{\S#2#1#3}
  \Crefformat{section}{Section~#2#1#3}
  \crefrangeformat{section}{\S\S#3#1#4--#5#2#6}
  \Crefrangeformat{section}{Sections~#3#1#4--#5#2#6}
  \crefmultiformat{section}{\S\S#2#1#3}{ and~#2#1#3}{, #2#1#3}{ and~#2#1#3}
  \Crefmultiformat{section}{Sections~#2#1#3}{ and~#2#1#3}{, #2#1#3}{ and~#2#1#3}
  \crefrangemultiformat{section}{\S\S#3#1#4--#5#2#6}{ and~#3#1#4--#5#2#6}{, #3#1#4--#5#2#6}{ and~#3#1#4--#5#2#6}
  \Crefrangemultiformat{section}{Sections~#3#1#4--#5#2#6}{ and~#3#1#4--#5#2#6}{, #3#1#4--#5#2#6}{ and~#3#1#4--#5#2#6}
  \crefformat{appendix}{Appendix~#2#1#3}
  \Crefformat{appendix}{Appendix~#2#1#3}
  \crefrangeformat{appendix}{Appendices~#3#1#4--#5#2#6}
  \Crefrangeformat{appendix}{Appendices~#3#1#4--#5#2#6}
  \crefmultiformat{appendix}{Appendices~#2#1#3}{ and~#2#1#3}{, #2#1#3}{ and~#2#1#3}
  \Crefmultiformat{appendix}{Appendices~#2#1#3}{ and~#2#1#3}{, #2#1#3}{ and~#2#1#3}
  \crefrangemultiformat{appendix}{Appendices~#3#1#4--#5#2#6}{ and~#3#1#4--#5#2#6}{, #3#1#4--#5#2#6}{ and~#3#1#4--#5#2#6}
  \Crefrangemultiformat{appendix}{Appendices~#3#1#4--#5#2#6}{ and~#3#1#4--#5#2#6}{, #3#1#4--#5#2#6}{ and~#3#1#4--#5#2#6}
  \crefformat{subappendix}{\S#2#1#3}
  \Crefformat{subappendix}{Section~#2#1#3}
  \crefrangeformat{subappendix}{\S\S#3#1#4--#5#2#6}
  \Crefrangeformat{subappendix}{Sections~#3#1#4--#5#2#6}
  \crefmultiformat{subappendix}{\S\S#2#1#3}{ and~#2#1#3}{, #2#1#3}{ and~#2#1#3}
  \Crefmultiformat{subappendix}{Sections~#2#1#3}{ and~#2#1#3}{, #2#1#3}{ and~#2#1#3}
  \crefrangemultiformat{subappendix}{\S\S#3#1#4--#5#2#6}{ and~#3#1#4--#5#2#6}{, #3#1#4--#5#2#6}{ and~#3#1#4--#5#2#6}
  \Crefrangemultiformat{subappendix}{Sections~#3#1#4--#5#2#6}{ and~#3#1#4--#5#2#6}{, #3#1#4--#5#2#6}{ and~#3#1#4--#5#2#6}
  \crefname{part}{Part}{Parts}
  \crefname{figure}{Figure}{Figures}
\fi

\iftac
  \let\qed\endproof
  \let\your@endproof\endproof
  \def\my@endproof{\your@endproof}
  \def\endproof{\my@endproof\gdef\my@endproof{\your@endproof}}
  \def\qedhere{\tag*{\endproofbox}\gdef\my@endproof{\relax}}
\fi

\iftac
  \def\pr@@f[#1]{\subsubsection*{\sc #1.}}
\fi

\def\thmqedhere{\expandafter\csname\csname @currenvir\endcsname @qed\endcsname}

\ifbeamer\else

\fi

\ifbeamer\else
  \setitemize[1]{leftmargin=2em}
  \setenumerate[1]{leftmargin=*}
\fi

\iftac
  \let\c@equation\c@subsection
\else
  \let\c@equation\c@thm
\fi
\numberwithin{equation}{section}

\ifcref\else
  \@ifpackageloaded{mathtools}{\mathtoolsset{showonlyrefs,showmanualtags}}{}
\fi

\alwaysmath{alpha}
\alwaysmath{beta}
\alwaysmath{gamma}
\alwaysmath{Gamma}
\alwaysmath{delta}
\alwaysmath{Delta}
\alwaysmath{epsilon}
\mdef\ep{\varepsilon}
\alwaysmath{zeta}
\alwaysmath{eta}
\alwaysmath{theta}
\alwaysmath{Theta}
\alwaysmath{iota}
\alwaysmath{kappa}
\alwaysmath{lambda}
\alwaysmath{Lambda}
\alwaysmath{mu}
\alwaysmath{nu}
\alwaysmath{xi}
\alwaysmath{pi}
\alwaysmath{rho}
\alwaysmath{sigma}
\alwaysmath{Sigma}
\alwaysmath{tau}
\alwaysmath{upsilon}
\alwaysmath{Upsilon}
\alwaysmath{phi}
\alwaysmath{Pi}
\alwaysmath{Phi}
\mdef\ph{\varphi}
\alwaysmath{chi}
\alwaysmath{psi}
\alwaysmath{Psi}
\alwaysmath{omega}
\alwaysmath{Omega}
\let\al\alpha
\let\be\beta

\let\si\sigma

\let\ze\zeta

\makeatother

\def\p{^+}
\def\m{^-}

\let\types\vdash

\def\isprop{\;\mathsf{prop}}

\def\ucd#1{\underline{\mathsf{CD}}_{#1}}

\def\sfd{_{\mathsf{fd}}}
\def\sfpd{_{\mathsf{fpd}}}
\def\sub{\mathsf{Sub}}
\def\usub(#1){\ul{\mathsf{Sub}(#1)}}
\def\cdt{\mathsf{CD}_{\cT}}

\def\ul#1{\underline{\smash{#1}}}
\def\uphi{\ul{\phi}}
\def\upsi{\ul{\psi}}
\def\uxi{\ul{\xi}}
\def\uze{\ul{\ze}}
\def\madj{\mathbb{MA}\mathsf{dj}}
\def\tmadj{\mathcal{MA}\mathit{dj}}
\def\madjs{\mathbb{MA}\mathsf{dj}_{\mathrm{S}}}
\def\madjcgr{\mathbb{MA}\mathsf{dj}_{\mathrm{CGR}}}
\def\polmadj{\mathbb{P}\mathsf{ol}\mathbb{MA}\mathsf{dj}}
\def\dchu{\mathbb{C}\mathsf{hu}}
\def\dadj{\mathbb{A}\mathsf{dj}}
\def\cAdj{\mathcal{A}\mathit{dj}}
\def\iSeq{\mathsf{Seq}}
\def\psh{\mathrm{Psh}}
\def\adjc{\mathsf{Adj}_{\cC}}
\def\adjcom{\mathsf{Adj}_{\cC}(\Omega)}

\def\stpoly#1{\ast\mathsf{Poly}_{#1}}

\def\zzz{(\mathord{;})}
\def\zz{()}

\let\tens\boxtimes
\def\cotens{\mathbin{\raisebox{-1pt}{\rotatebox{45}{$\scriptstyle\boxtimes$}}}}
\def\dual#1{{#1}^\perp}

\def\cat{\ensuremath{\mathcal{C}\mathit{at}}\xspace}

\def\tcat{\ensuremath{2\text{-}\mathcal{C}\mathit{at}}\xspace}

\def\poly{\ensuremath{\mathcal{P}\mathit{oly}}\xspace}
\def\chu{\mathsf{Chu}}
\def\chuz{\mathsf{Chu}_0}
\def\tchu{\mathcal{C}\mathit{hu}}
\def\Set{\ensuremath{\mathsf{Set}}\xspace}
\def\dial{\mathsf{Dial}}

\DeclarePairedDelimiter\floor{\lfloor}{\rfloor}
\DeclarePairedDelimiter\rep{[}{]}
\def\repl#1{\rep{#1}_{\mathrm{L}}}
\def\repr#1{\rep{#1}_{\mathrm{R}}}
\def\ev{\mathsf{ev}}
\def\floorlex#1{\floor{#1}_{\mathrm{lex}}}
\def\Lex{\mathrm{Lex}}

\let\P\cP
\def\abs#1{|#1|}

\def\d{^{\bullet}}
\let\ihom\multimap

\title[The 2-Chu-Dialectica construction]{The 2-Chu-Dialectica construction and the polycategory of multivariable adjunctions}
\author{Michael Shulman}
\thanks{This material is based on research
    sponsored by The United States Air Force Research Laboratory under
    agreement number FA9550-15-1-0053 and FA9550-16-1-0292.  The
    U.S. Government is authorized to reproduce and distribute reprints
    for Governmental purposes notwithstanding any copyright notation
    thereon.  The views and conclusions contained herein are those of
    the author and should not be interpreted as necessarily
    representing the official policies or endorsements, either expressed
    or implied, of the United States Air Force Research Laboratory, the
    U.S. Government, or Carnegie Mellon University.}

\copyrightyear{2020}
\eaddress{shulman@sandiego.edu}

\begin{document}
\maketitle

\begin{abstract}
  Cheng, Gurski, and Riehl constructed a cyclic double multicategory
  of multivariable adjunctions.  We show that the same information is
  carried by a poly double category, in which opposite categories are
  polycategorical duals.  Moreover, this poly double category is a
  full substructure of a double Chu construction, whose objects are a
  sort of polarized category, and which is a natural home for
  2-categorical dualities.

  We obtain the double Chu construction using a general
  ``Chu-Dialectica'' construction on polycategories, which includes
  both the Chu construction and the categorical Dialectica
  construction of de Paiva.  The Chu and Dialectica constructions each
  impose additional hypotheses making the resulting polycategory
  representable (hence $\ast$-autonomous), but for different reasons;
  this leads to their apparent differences.
\end{abstract}

\setcounter{tocdepth}{1}
\tableofcontents

\section{Introductions}
\label{sec:introduction}

I have written two introductions to this paper, each of which can be
read independently.  If you are interested in Dialectica and Chu
constructions, please continue with \cref{sec:intro-dc}; but if you
are more interested in multivariable adjunctions, I suggest skipping
ahead to read \cref{sec:intro-mvar} first.

\subsection{First Introduction: Unifying the Dialectica and Chu constructions}
\label{sec:intro-dc}

The categorical Dialectica constructions were introduced
by~\cite{depaiva:dialectica,depaiva:dialectica-like} as an abstraction
of G\"{o}del's ``Dialectica interpretation''~\cite{godel:dialectica}.
Although G\"{o}del's interpretation modeled intuitionistic logic, de
Paiva's categorical analysis revealed that it factors naturally
through Girard's \emph{classical linear} logic~\cite{girard:ll}, which
categorically means a $\ast$-autonomous
category~\cite{barr:staraut-ll,cs:wkdistrib}.

On the other hand, the Chu
construction~\cite{chu:construction,chu:constr-app,barr:chu-history}
was introduced specifically as a way to produce $\ast$-autonomous
categories.  Anyone familiar with both constructions can tell that
they have a very similar feel, and one formal functorial comparison
was given in~\cite{paiva:dialectica-chu}.  In this paper we compare
them in a new way, by giving such a general construction that includes
both Dialectica and Chu constructions as special cases.

One reason it is hard to compare the Dialectica and Chu constructions
is that while their underlying categories are defined very similarly,
their monoidal structures are defined rather differently.  This
suggests that a fruitful way to compare them would be to perform them
both in a more general context where these monoidal structures need
not exist, but can be characterized up to isomorphism by universal
properties.  In other words, instead of monoidal categories we will
use \emph{multicategories}, and instead of $\ast$-autonomous
categories we will use
\emph{polycategories}~\cite{szabo:polycats}.\footnote{A
  polycategorical viewpoint on the Chu construction, in the case of
  non-symmetric ``cyclic'' poly-bicategories, can also be found
  in~\cite[Example 1.8(2)]{cks:polybicats}.  In this paper we will
  consider only the symmetric case.}

To define a multi- or polycategorical version of the Dialectica or Chu
constructions, we need to start by asking what universal property is
possessed by their tensor products, i.e.\ what functor they represent,
in the way that the tensor product of abelian groups represents
bilinear maps.  In other words, if $\tens$ denotes these tensor
products, then what does a morphism $A\tens B \to C$ look like if we
``beta-reduce away'' the definition of $\tens$?

First, let us consider the Chu construction, which in its basic form
applies to a closed symmetric monoidal category \bC with pullbacks,
equipped with an arbitrary object $\Omega$.  In the resulting category
$\chu(\bC,\Omega)$,
\begin{itemize}
\item The objects are triples $A = (A\p,A\m,\uA)$ where $A\p,A\m$ are objects of \bC and $\uA : A\p \otimes A\m \to \Omega$ is a morphism in \bC.
\item The morphisms $f:(A\p,A\m,\uA)\to (B\p,B\m,\uB)$ are pairs $(f\p,f\m)$ where $f\p : A\p \to B\p$ and $f\m : B\m \to A\m$ are morphisms in \bC such that
  \[\uA \circ (1 \otimes f\m) = \uB \circ (f\p \otimes 1).\]
\end{itemize}
The tensor product of two objects $A,B\in\chu(\bC,\Omega)$ is defined by
\begin{align*}
  (A\tens B)\p &= A\p \otimes B\p \\
  (A\tens B)\m &= [A\p,B\m] \times_{[A\p\otimes B\p,\Omega]} [B\p,A\m]\\
  \ul{A\tens B} &= \scriptstyle \Big(\big([A\p,B\m] \times_{[A\p\otimes B\p, \Omega]} [B\p,A\m]\big)\otimes A\p \otimes B\p
    \to [A\p\otimes B\p, \Omega] \otimes A\p \otimes B\p
    \to \Omega\Big)
\end{align*}
By definition of the morphisms in $\chu(\bC,\Omega)$, a morphism
$A\tens B \to C$ consists of \bC-morphisms
$f\p : A\p\otimes B\p \to C\p$ and
$f\m : C\m \to [A\p,B\m] \times_{[A\p\otimes B\p,\Omega]} [B\p,A\m]$
such that some diagram commutes.  But by the universal property of
pullbacks and internal-homs, to give $f\m$ is equivalent to giving
$f\m_1:A\p \otimes C\m \to B\m$ and $f\m_2 : B\p \otimes C\m \to A\m$
such that
$\uB \circ (1_{B\p} \otimes f\m_1) = \uA \circ (1_{A\p}\otimes f\m_2)$
(modulo a symmetry in the domain).  And the commutative diagram then
simply asserts that this joint composite is also equal (again, modulo
symmetry) to $\uC \circ (f\p \otimes 1_{C\m})$.

Thus, in total a morphism $A\tens B \to C$ in $\chu(\bC,\Omega)$ consists of morphisms
\begin{align*}
  f\p &: A\p\otimes B\p \to C\p\\
  f\m_1 &:A\p \otimes C\m \to B\m\\
  f\m_2 &: B\p \otimes C\m \to A\m
\end{align*}
such that
\[ \uB \circ (1_{B\p} \otimes f\m_1) = \uA \circ (1_{A\p}\otimes f\m_2) = \uC \circ (f\p \otimes 1_{C\m}). \]
There are several things to note about this:
\begin{itemize}
\item It is certainly a ``two-variable'' generalization of the definition of ordinary morphisms $A\to B$ in $\chu(\bC,\Omega)$.
\item It makes sense even if $\bC$ is only a multicategory, with $f\p : (A\p,B\p) \to C\p$ and so on.
\item With a little thought, one can guess the correct $n$-variable version, 
  % and verify that it corresponds to maps out of an $n$-ary tensor product.
  % One can also
  dualize to describe maps $A\to B\cotens C$, where $B\cotens C = (B^*\tens C^*)^*$ is the dual ``cotensor product'' (the ``par'' of linear logic), and then generalize to maps from an $n$-ary tensor to an $m$-ary cotensor.
  This leads to our polycategorical definition.
\item If we write the above equalities in the internal type theory of \bC, using formal variables $a:A\p$, $b:B\p$, and $c:C\m$, they become
\[ \uC(f\p(a,b),c) = \uB(b,f\m_1(a,c)) = \uA(a,f\m_2(b,c)), \]
which is highly reminiscent of the hom-set isomorphisms in a \emph{two-variable adjunction}.
We will pick up this thread in \cref{sec:intro-mvar}.
%  relating three functors $f:A\times B\to C$ and $g:A\op \times C\to B$, and $h:B\op\times C\to A$:
% \[ C(f(a,b),c) \cong B(b,g(a,c)) \cong A(a,h(b,c)), \]
% such as the tensor product and internal-homs of a closed monoidal category, or the tensor/hom/cotensor situation of an enriched category.
\end{itemize}

Moving on to the Dialectica construction, we will describe the version
from~\cite{paiva:dialectica-chu}, which looks the most like the Chu
construction.  This Dialectica construction applies to a closed
symmetric monoidal category \bC with finite products, equipped with an
object $\Omega$ that internally has the structure of a closed monoidal
poset.  In the resulting category $\dial(\bC,\Omega)$,
\begin{itemize}
\item The objects are the same as those of $\chu(\bC,\Omega)$: triples $A = (A\p,A\m,\uA)$ where $A\p,A\m$ are objects of \bC and $\uA : A\p \otimes A\m \to \Omega$ is a morphism in \bC.
\item The morphisms $f:(A\p,A\m,\uA)\to (B\p,B\m,\uB)$ are pairs $(f\p,f\m)$ where $f\p : A\p \to B\p$ and $f\m : B\m \to A\m$ are morphisms in \bC such that
  \[\uA \circ (1 \otimes f\m) \le \uB \circ (f\p \otimes 1)\]
  in the internal order of $\Omega$ (applied pointwise to morphisms $A\p\otimes B\m\to \Omega$).
\end{itemize}
The tensor product of two objects $A,B\in\dial(\bC,\Omega)$ is defined by
\begin{align*}
  (A\tens B)\p &= A\p \otimes B\p \\
  (A\tens B)\m &= [A\p,B\m] \times [B\p,A\m]
\end{align*}
with $\ul{A\tens B}$ being the tensor product (in the internal monoidal structure of $\Omega$) of the two morphisms
\begin{gather*}
  A\p \otimes B\p \otimes ([A\p,B\m] \times [B\p,A\m])
  \to A\p \otimes B\p \otimes [A\p,B\m]
  \to B\p \otimes B\m
  \xto{\uB} \Omega\\
  A\p \otimes B\p \otimes ([A\p,B\m] \times [B\p,A\m])
  \to A\p \otimes B\p \otimes [B\p,A\m]
  \to A\p \otimes A\m
  \xto{\uA} \Omega
\end{gather*}
Now by definition of the morphisms in $\dial(\bC,\Omega)$, a morphism
$A\tens B \to C$ consists of \bC-morphisms
$f\p : A\p\otimes B\p \to C\p$ and
$f\m : C\m \to [A\p,B\m] \times [B\p,A\m]$ such that some inequality
holds.  But by the universal property of products and internal-homs,
to give $f\m$ is equivalent to giving $f\m_1:A\p \otimes C\m \to B\m$
and $f\m_2 : B\p \otimes C\m \to A\m$, and the inequality then asserts
that
\begin{equation}
  (\uA \circ (1_{A\p}\otimes f\m_2)) \tens (\uB \circ (1_{B\p} \otimes f\m_1)) \le (\uC \circ (f\p \otimes 1_{C\m}))\label{eq:diale21}
\end{equation}
in the internal order of $\Omega$ (applied pointwise to morphisms $A\p\otimes B\p\otimes C\m \to \Omega$).
We now note similarly that:
\begin{itemize}
\item This is also certainly a ``two-variable'' generalization of the definition of ordinary morphisms $A\to B$ in $\dial(\bC,\Omega)$.
\item It also makes sense if $\bC$ is only a multicategory, with $f\p : (A\p,B\p) \to C\p$ etc.
\item In fact, it makes sense even if $\Omega$ is only a \emph{multi-poset} (a multicategory having at most one morphism with any given domain and codomain, just as a poset is a category with this property), with~\eqref{eq:diale21} replaced by
\begin{equation}
  \big(\uA \circ (1_{A\p}\otimes f\m_2) \,,\, \uB \circ (1_{B\p} \otimes f\m_1)\big) \le (\uC \circ (f\p \otimes 1_{C\m}))\label{eq:diale21m}
\end{equation}
\item One can again guess the correct $n$-to-$m$-variable version and write down a polycategorical definition, with $\Omega$ replaced by a \emph{poly-poset} (a {polycategory} having at most one morphism in each hom-set).
\end{itemize}

Furthermore, the descriptions of morphisms $A\tens B \to C$ in
$\chu(\bC,\Omega)$ and $\dial(\bC,\Omega)$ are very similar, indeed
they are related in essentially the same way as the descriptions of
ordinary morphisms $A\to B$.  Specifically, the Chu construction asks
for an \emph{equality}, while the Dialectica construction asks for an
\emph{inequality} --- where an ``inequality'' between more than two
elements is interpreted with respect to a multi-poset or poly-poset
structure.

This leads to our common generalization: just as equalities
$\phi=\psi$ are inequalities in a discrete poset (where $\phi\le \psi$
is defined to mean $\phi=\psi$), ``multi-variable equalities''
$\phi=\psi=\xi$ can be regarded as ``multi-variable inequalities'' in
a ``discrete poly-poset'', where an inequality $(\phi,\psi)\le (\xi)$
is \emph{defined to mean} $\phi=\psi=\xi$.  Thus, \emph{the
  polycategorical Dialectica construction includes the polycategorical
  Chu construction}.  The reason the original constructions look
different is that they make different additional assumptions, each of
which implies that the polycategorical result is ``representable'' and
hence defines a $\ast$-autonomous category --- but this
representability happens in \emph{different ways} for the original
Dialectica and Chu constructions.

In fact, we will generalize further in a few ways:
\begin{itemize}
\item We allow $\Omega$ to be a poly\emph{category} rather than a poly\emph{poset}, i.e.\ our construction will be ``proof-relevant'' in the strongest sense.
\item We will replace the object $\Omega$ by a not-necessarily-representable presheaf with the same structure.
  This allows us to include the original Dialectica constructions~\cite{depaiva:dialectica,depaiva:dialectica-like}, where instead of morphisms into $\Omega$ we use subobjects, without supposing $\bC$ to have a subobject classifier.
\item We will generalize the output of the construction to be a $\bC$-indexed family of polycategories rather than a single one, as in~\cite{biering:dialectica,hofstra:dialectica}.
  This amounts to building a model of first-order rather than merely propositional linear logic.
\end{itemize}
Taken together, these generalizations imply that the output of our
``Chu-Dialectica construction'' is the same kind of thing as its
input: a multicategory equipped with a presheaf of polycategories,
which we call a \emph{virtual linear hyperdoctrine}.  I do not know
whether this endomorphism of the category of virtual linear
hyperdoctrines has a universal property
(see~\cite{pavlovic:chu-i,hofstra:dialectica} for universal properties
of the Chu and Dialectica constructions respectively).

From the perspective of higher category theory, we can regard our
construction as a categorification.  In the original Chu construction,
$\Omega$ is a discrete object, i.e.\ a 0-category.  In the original
Dialectica construction, $\Omega$ is a posetal object, a.k.a.\ a
$(0,1)$-category (where a set or 0-category is more verbosely called a
$(0,0)$-category).  Our construction (as well as other categorified
Dialectica constructions,
e.g.~\cite{biering:dialectica,hofstra:dialectica}) allows $\Omega$ to
be a categorical object, i.e.\ a $(1,1)$-category.

This suggests that our construction should also specialize to a
version involving $(1,0)$-categories, i.e.\ groupoids.  It seems
appropriate to call this a \emph{2-Chu construction}, since it
replaces the equalities in the ordinary Chu constructions by
\emph{isomorphisms}.  The ``prototypical'' 2-Chu construction
$\tchu(\cat,\Set)$ (which directly categorifies the prototypical Chu
construction $\chu(\Set,2)$) is particularly interesting as its
morphisms are a ``polarized'' sort of \emph{multivariable adjunction}.

The second introduction to the paper, which follows, reverses the flow
of motivation by \emph{starting} with multivariable adjunctions.

\subsection{Second Introduction: The polycategory of multivariable adjunctions}
\label{sec:intro-mvar}

In view of the well-known importance of adjunctions in category
theory, it is perhaps surprising that it has taken so long for
\emph{multivariable adjunctions} to be systematically studied.  To be
sure, \emph{two}-variable adjunctions have a long history, and include
some of the earliest examples of adjunctions.  For instance, in a
biclosed monoidal category each functor $(A\otimes -)$ has a right
adjoint $[A,-]^l$, and each functor $(-\otimes B)$ has a right adjoint
$[B,-]^r$; but this is more symmetrically expressed by saying that the
two-variable functor $\otimes$ has $[-,-]^l$ and $[-,-]^r$ as
two-variable right adjoints.  To be precise, in this case we have
three functors
\begin{mathpar}
  \otimes : \cA \times \cA \to \cA \and
  [-,-]^l : \cA\op \times \cA \to \cA \and
  [-,-]^r : \cA\op \times \cA \to \cA \and
\end{mathpar}
with natural isomorphisms
\[ \cA(A\otimes B,C) \cong \cA(B,[A,C]^l) \cong \cA(A,[B,C]^r). \]
In general, a two-variable adjunction $(\cA,\cB)\to\cC$ consists of functors
\begin{mathpar}
  f:\cA\times \cB\to \cC\and
  g:\cA\op\times \cC\to \cB\and
  h:\cB\op\times \cC\to \cA
\end{mathpar}
and natural isomorphisms
\[ \cC(f(a,b),c) \cong \cB(b,g(a,c)) \cong \cA(a,h(b,c)).\] In
addition to biclosed monoidal structures, another well-known example
is the ``tensor-hom-cotensor'' (or ``copower-hom-power'') situation of
an enriched category, which inspired the terminology
\emph{THC-situation} for the general case
in~\cite{gray:thc-cc-laxlim}.  The name \emph{adjunction of two
  variables} from~\cite{hovey:modelcats} was shortened
in~\cite{riehl:mon-ams,cgr:cyclic} to \emph{two-variable adjunction};
in~\cite{guitart:trijunctions} the term used is \emph{trijunction}
(though see below).

Of course when we have one-variable and two-variable versions of
something, it is natural to expect an $n$-variable version.  If we go
back to the fact that the functors $g$ and $h$ in a two-variable
adjunction are determined up to unique isomorphism by $f$, we can
define an \emph{$n$-variable adjunction} $(\cA_1,\dots,\cA_n) \to \cB$
to be a functor $\cA_1\times\cdots \times \cA_n \to \cB$ such that if
we fix its value on all but one (say $\cA_i$) of the input categories,
the resulting functor $\cA_i \to \cB$ has a right adjoint.  Each such
right adjoint then automatically becomes contravariantly functorial on
the categories $\cA_j$ for $j\neq i$.
% For instance, if we unwind this when $n=3$ we see that a three-variable adjunction $(\cA,\cB,\cD) \to \cE$ consists of functors
% \begin{mathpar}
%   f: \cA\times\cB\times\cD\to\cE\\
%   g: \cA\op \times \cB\op\times \cE\to\cD\and
%   h: \cA\op\times \cD\op\times\cE\to\cB\and
%   k : \cB\op\times\cD\op\times \cE\to\cA
% \end{mathpar}
% and natural isomorphisms
% \[ \cE(f(a,b,d),e) \cong \cD(d,g(a,b,e)) \cong \cB(b,h(a,d,e)) \cong \cA(a,k(b,d,e)). \]

Unsurprisingly, multivariable adjunctions of this sort can be
assembled into a multicategory: we can compose a two-variable
adjunction $(\cA, \cB)\to\cC$ with another one $(\cC,\cD)\to\cE$ to
obtain a three-variable adjunction $(\cA,\cB,\cD)\to\cE$, and so on.
However, one-variable adjunctions are the morphisms not only of a
category but of a 2-category $\cAdj$, whose 2-cells are mate-pairs of
natural transformations.  More generally, one-variable adjunctions
form the horizontal morphisms in a \emph{double} category $\dadj$,
whose vertical morphisms are single functors, and whose 2-cells are a
more general kind of mate-pairs.  (Recall that if $f\dashv g$ and
$h\dashv k$ are adjunctions, then the ``mate correspondence'' is a
bijection between natural transformations $f u \to v h$ and
$u k \to g v$ obtained by pasting with the adjunction unit and counit.
The functoriality of this bijection is conveniently expressed in terms
of the double category $\dadj$; see~\cite{ks:r2cats}.)

The first step towards a similar calculus for multivariable
adjunctions was taken by~\cite{cgr:cyclic}, who exhibited them as the
horizontal\footnote{Actually, their double categories are transposed
  from ours, so for them the multivariable adjunctions are the
  \emph{vertical} morphisms.} morphisms in a \emph{cyclic multi double
  category}\footnote{They called it a \emph{cyclic double
    multicategory}, but the phrase ``double multicategory'' may
  suggest an internal multicategory in multicategories rather than the
  intended meaning of an internal category in multicategories, so we
  have chosen to order the modifiers differently.} $\madj$ (i.e.\ an
internal category in the category of cyclic multicategories).  The
multicategory structure of $\madj$ is unsurprising.  Its vertical
arrows of $\madj$ are functors and its 2-cells are natural
transformations, while its cyclic structure encodes a calculus of
\emph{multivariable mates}, describing the behavior of multivariable
adjunctions with respect to passage to opposite categories.  In
general, a cyclic structure on a multicategory consists of an
involution $(-)\d$ on objects together with a cyclic action on
morphism sets
\[ \cM(A_1,\dots,A_n ; B) \to \cM(A_2,\dots,A_n,B\d; A_1\d) \]
satisfying appropriate axioms.  In $\madj$ we define $\cA\d = \cA\op$,
and the cyclic action generalizes the observation that a two-variable
adjunction $(\cA, \cB)\to\cC$ is essentially the same as a
two-variable adjunction $(\cC\op,\cA)\to \cB\op$ or
$(\cB,\cC\op)\to \cA\op$.  The extension of this cyclic action to
2-cells then encodes the mate correspondence.

In practice, three- and higher-variable adjunctions seem to arise
mainly as composites of two-variable adjunctions.  But the whole
multicategory structure is nevertheless useful, because it gives an
abstract context in which to express conditions and axioms regarding
such composites.  For instance, the associativity of the tensor
product in a closed monoidal category has an equivalent form involving
the internal-hom~\cite{ek:closed-cats}; they are 2-cells in $\madj$
related by the mate correspondence.  Put differently, just as a
monoidal category can be defined as a pseudomonoid in the 2-category
$\cat$, a \emph{closed} monoidal category can be defined as a
pseudomonoid in $\tmadj$, the horizontal 2-multicategory of $\madj$.
Similarly, a module over a pseudomonoid $A$ in $\tmadj$ is an
$A$-enriched category with powers and copowers, and so on.

In this paper I will propose a different viewpoint on $\madj$: rather
than a cyclic multicategory, we can regard it as a
\emph{polycategory}.  A polycategory is like a multicategory, but it
allows the \emph{codomain} of a morphism to contain multiple objects,
as well as the domain; thus we have morphisms like
$f: (A,B) \to (C,D)$.  Such morphisms can be composed only ``along
single objects'', with the ``leftover'' objects in the codomain of $f$
and the domain of $g$ surviving into the codomain and domain of
$g\circ f$.  For instance, given $f: (A,B) \to (C,D)$ and
$g:(E,C) \to (F,G)$ we get $g\circ_C f : (E,A,B) \to (F,G,D)$.
% The usual examples of polycategories are are rather removed from the experience of most mathematicians, but the definition is exactly what we need for multivariable adjunctions.

What is a multivariable adjunction
$(\cA_1,\dots,\cA_m) \to(\cB_1,\dots,\cB_n)$?  There are several ways
to figure out the answer.  One is to inspect the definition of a
multivariable adjunction $(\cA_1,\dots,\cA_m) \to \cB_1$ and rephrase
it in a way that doesn't depend on the assumption $n=1$.  The functors
involved in such an adjunction are
\begin{gather*}
  f_1 : \cA_1\times\cdots\times\cA_m \to \cB_1\\
  g_i : \cA_1\op\times\cdots \widehat{\cA_i\op} \cdots\times \cA_m\op
  \times \cB_1 \to \cA_i \mathrlap{\qquad (1\le i\le m)}
\end{gather*}
where $\widehat{\cA_i\op}$ indicates that ${\cA_i\op}$ is omitted.
This can be described as ``for each category $\cA_i$ or $\cB_j$, a
functor with that codomain, whose domain is the product of all the
other categories, with opposites applied to those denoted by the same
letter as the codomain''.  That is, the functor $g_i$ with codomain
$\cA_i$ depends contravariantly on all the other $\cA$'s and
covariantly on the (single) $\cB$, while the functor $f$ with codomain
$\cB_1$ depends contravariantly on the (zero) other $\cB$'s and
covariantly on all the $\cA$'s.  If we apply this description in the
case $n>1$ as well, we see that a multivariable adjunction
$(\cA_1,\dots,\cA_m) \to(\cB_1,\dots,\cB_n)$ should involve functors
\begin{gather*}
  f_j : \cA_1\times\cdots\times\cA_m \times \cB_1\op \times\cdots \widehat{\cB_j\op} \cdots\times \cB_n\op \to \cB_j \mathrlap{\qquad (1\le j\le n)}\\
  g_i : \cA_1\op\times\cdots \widehat{\cA_i\op} \cdots\times \cA_m\op \times \cB_1 \times \cdots\times \cB_n \to \cA_i \mathrlap{\qquad (1\le i\le m)} 
\end{gather*}
with an appropriate family of natural isomorphisms.
For instance, a multivariable adjunction $(\cA,\cB) \to (\cC,\cD)$ consists of four functors
\begin{mathpar}
  f:\cC\op\times \cA\times \cB \to \cD \and
  g:\cA \times \cB \times \cD\op \to \cC\\
  h: \cA\op\times \cC\times \cD\to \cB \and
  k: \cC\times \cD \times \cB\op\to \cA
\end{mathpar}
and natural isomorphisms 
\[ \cD(f(c,a,b),d) \cong \cC(g(a,b,d),c) \cong \cB(b,h(a,c,d)) \cong \cA(a,k(c,d,b)). \]

I find this definition quite illuminating already.  One of the odd
things about a two-variable adjunction, as usually defined, is the
asymmetric placement of opposites.  The polycategorical perspective
reveals that this arises simply from the asymmetry of having a 2-ary
domain but a 1-ary codomain: a ``$(2,2)$-variable adjunction'' as
above looks much more symmetrical.

With this definition of $(m,n)$-variable adjunctions in hand, it is a
nice exercise to write down a composition law making them into a
polycategory.  For instance, suppose in addition to
$(f,g,h,k) : (\cA,\cB) \to (\cC,\cD)$ as above, we have a two-variable
adjunction $(\ell,m,n) : (\cD,\cE)\to \cZ$ with
$\cZ(\ell(d,e),z) \cong \cD(d,m(e,z)) \cong \cE(e,n(d,z))$.  Then we
have a composite multivariable adjunction
$(\cA,\cB,\cE) \to (\cC,\cZ)$ defined by
\[ \cC(g(a,b,m(e,z)),c) \cong \cZ(\ell(f(c,a,b),e),z) \cong \cA(a,k(c,m(e,z),b)) \cong \cdots. \]

Of course, a $(1,1)$-variable adjunction is an ordinary adjunction,
while a $(2,1)$-variable adjunction is a two-variable adjunction as
above.  A $(2,0)$-variable adjunction $(\cA,\cB)\to ()$ consists of
functors $f:\cA\op\to \cB$ and $g:\cB\op\to \cA$ and a natural
isomorphism $\cB(b,f(a)) \cong \cA(a,g(b))$.  This is sometimes called
a \textbf{mutual right adjunction} or \textbf{dual adjunction}, and
arises frequently in examples, such as Galois connections betwen
posets or the self-adjunction of the contravariant powerset functor.
Similarly, a $(0,2)$-variable adjunction $() \to (\cA,\cB)$ is a
mutual left adjunction $\cB(f(a),b) \cong \cA(g(b),a)$.  Of course a
mutual right or left adjunction can also be described as an ordinary
adjunction between $\cA\op$ and $\cB$, or between $\cA$ and $\cB\op$,
but the choice of which category to oppositize is arbitrary; the
polycategorical approach respects mutual right and left adjunctions as
independent objects.\footnote{At this point I encourage the reader to
  stop and think for a while about what a $(0,0)$-variable adjunction
  should be.  The answer will be given in \cref{rmk:00}.}

More generally, an $(n,0)$-variable adjunction
$(\cA_1,\dots,\cA_n) \to ()$ is a ``mutual right multivariable
adjunction'' between $n$ contravariant functors
\[ f_i : \cA_{i+1}\times \cdots \times \cA_n \times \cA_1 \times
  \cdots \times \cA_{i-1}\to \cA_i\op.\] In fact, the ``mutual right''
version is the formal definition of $n$-variable adjunction given
in~\cite{cgr:cyclic} (and, in the case $n=3$, of ``trijunction''
in~\cite{guitart:trijunctions}).  This makes the cyclic structure more
apparent, but the enforced contravariance makes for a mismatch with
many standard examples.

A further advantage of the polycategorical framework is the way that
opposite categories enter the picture: rather than imposed by the
\emph{structure} of a cyclic action, they are characterized by a
universal \emph{property}.  Specifically, they are duals in the
polycategorical sense: we have multivariable adjunctions
$\eta : () \to (\cA,\cA\op)$ and $\ep : (\cA\op,\cA)\to ()$ satisfying
analogues of the triangle identities.  Opposite categories are also
dual objects in the monoidal bicategory of profunctors, but the
polycategory of multivariable adjunctions provides a new perspective,
which in particular characterizes them up to equivalence (not just
Morita equivalence).

In fact, the characterization of $\cA\op$ as a polycategorical dual of
$\cA$ encodes almost exactly the same information as the cyclic action
of~\cite{cgr:cyclic}.  Any polycategory $\cP$ with strict duals
(a.k.a.\ a ``$\ast$-polycategory''~\cite{hyland:pfthy-abs}) has an
underlying cyclic symmetric multicategory, in which the cyclic action
\[ \cP(A_1,\dots,A_n ; B) \to \cP(A_2,\dots,A_n,B\d; A_1\d) \] is
obtained by composing with $\ep_{B}$ and $\eta_{A_1}$.
% Just like with duals in a monoidal category, by composing with $\eta_\cA$ and $\ep_\cA$ we can translate back and forth between multivariable adjunctions $(\cA,\cB)\to\cC$ and $\cB \to (\cA\op,\cC)$, and then by further composing with $\eta_\cC$ and $\ep_\cC$ we can connect the latter to multivariable adjunctions $(\cB,\cC\op)\to\cA\op$.
Conversely, any cyclic symmetric multicategory $\cM$ can be extended to a polycategory by defining
\[ \cM(A_1,\dots, A_m ; B_1,\dots,B_n) =
  \cM(A_1,\dots,A_m,B_1\d,\dots,\widehat{B_j\d},\dots,B_{n}\d; B_j).\]
The cyclic structure ensures that this is independent, up to
isomorphism, of $j$.  The polycategorical composition can then be
induced from the multicategorical one and the cyclic action, and the
cyclic ``duals'' $A\d$ indeed turn out to be abstract polycategorical
duals.

Thus symmetric polycategories with duals are almost\footnote{See
  \cref{rmk:00} for why the ``almost''.} equivalent to cyclic
symmetric multicategories, and our polycategorical $\madj$ corresponds
under this almost-equivalence to the cyclic $\madj$
of~\cite{cgr:cyclic}.  This provides another \textit{a posteriori}
explanation of the definition of $(m,n)$-variable adjunctions: they
are exactly the morphisms in the polycategory we obtain by passing the
cyclic multicategory $\madj$ across this equivalence.  For instance,
the reader may check that a $(2,2)$-variable adjunction
$(\cA,\cB) \to (\cC,\cD)$ could equivalently be defined to be simply a
three-variable adjunction $(\cA,\cB,\cC\op) \to \cD$ (or,
equivalently, $(\cA,\cB,\cD\op)\to \cC$).

Finally, like the cyclic multicategory $\madj$ of~\cite{cgr:cyclic},
the polycategory $\madj$ is in fact a poly \emph{double} category
(meaning an internal category in the category of polycategories),
whose vertical arrows are functors and whose 2-cells are an
appropriate sort of multivariable mate tuple.  Thus, it is equally
appropriate for studying the multivariable mate correspondence.  It
also suggests new applications: for instance, in a 2-polycategory we
can define \emph{pseudo-comonoids} and \emph{Frobenius pseudomonoids},
and in a future paper~\cite{shulman:frobmvar} (building
on~\cite{ds:quantum,street:frob-psmon,egger:frob-lindist}) I will show
that Frobenius pseudomonoids in $\tmadj$ are $\ast$-autonomous
categories.

However, there is still something unsatisfying about the picture.  The
double category $\dadj$ of ordinary adjunctions can actually be
constructed out of internal adjunctions in \emph{any} 2-category $\sK$
instead of $\cat$; but it is unclear exactly what the analogous
statement should be for multivariable adjunctions.  In particular, the
definition of multivariable adjunction involves the notion of
\emph{opposite category}, which despite its apparent simplicity is
actually one of the more mysterious and difficult-to-abstract
properties of $\cat$.  At the ``one-variable'' level it is simply a
2-contravariant involution
$\cat\co\cong \cat$~\cite{shulman:contravariance}, but its
multivariable nature is still not fully understood (despite important
progress such as~\cite{ds:monbi-hopfagbd,weber:2toposes}).

However, it turns out that we can avoid this question entirely if we
are willing to settle for constructing something rather \emph{larger}
than $\madj$.  Upon inspection, the definition of multivariable
adjunction uses very little information about the relation of a
category to its opposite: basically nothing other than the existence
of the hom-functors $\cA\op \times \cA \to \Set$, and nothing at all
about the structure of their codomain $\Set$.  Thus, instead of trying
to \emph{characterize} the opposite of a category, we can simply
consider ``categories equipped with a formal opposite''.

Let $\sK$ be a symmetric monoidal 2-category with a specified object
$\Omega$.  We define an \emph{$\Omega$-polarized object} to be a
triple $(A\p,A\m,\uA)$ where $A\p,A\m$ are objects of \sK and
$\uA : A\p \otimes A\m \to \Omega$.  If $\sK=\cat$ and $\Omega=\Set$,
every category $\cA$ induces a \emph{representable} $\Set$-polarized
object $\rep \cA$ with $\rep \cA\p = \cA\op$, $\rep \cA\m=\cA$, and
$\underline{\rep \cA} = \hom_\cA$ (but see \cref{rmk:duality}).

A \emph{polarized adjunction} $f:A\to B$ between polarized objects
consists of morphisms $f\p:A\p\to B\p$ and $f\m:B\m\to A\m$ and an
isomorphism $\uA \circ (1\otimes f\m) \cong \uB \circ (f\p\otimes 1)$.
Similarly, a \emph{polarized two-variable adjunction} $(A,B) \to C$
consists of morphisms
\begin{mathpar}
  f:A\p\otimes B\p \to C\p\and
  g:A\p\otimes C\m \to B\m\and
  h:B\p\otimes C\m\to A\m
\end{mathpar}
and isomorphisms (modulo appropriate symmetric actions)
\[ \uA \circ (1\otimes h) \cong \uB \circ (1\otimes g) \cong \uC \circ
  (f\otimes 1). \] We can similarly define polarized $(n,m)$-variable
adjunctions and assemble them into a polycategory.  More generally, we
can take them to be the horizontal morphisms in a poly \emph{double}
category $\polmadj(\sK,\Omega)$; its vertical morphisms are
\emph{polarized functors} $h:A\to B$ consisting of morphisms
$f\p:A\p\to B\p$ and $f\m:A\m\to B\m$ (note that both go in the same
direction) and a 2-cell $\uA \Rightarrow \uB \circ (f\p\otimes f\m)$,
and its 2-cells are families of 2-cells in $\sK$ satisfying a
``polarized mate'' relationship.

In the case $\sK=\cat$, $\Omega=\Set$, a polarized adjunction between
representable polarized categories $\rep \cA \to \rep \cB$ reduces to
an ordinary adjunction, and likewise a polarized two-variable
adjunction $(\rep \cA, \rep \cB) \to\rep \cC$ reduces to an ordinary
two-variable adjunction.  More generally, we can say that
$\polmadj(\cat,\Set)$ contains our original $\madj$ as a
``horizontally full'' subcategory (but see \cref{rmk:duality}).  So
there is a general 2-categorical construction that at least comes
close to reproducing $\madj$.

On the other hand, $\polmadj(\cat,\Set)$ is also interesting in its
own right!  Its objects and vertical arrows are (modulo replacement of
$A\p$ by its opposite) the ``polarized categories'' and functors
of~\cite{cs:polarized}, which were studied as semantics for polarized
logic and games.  It also provides a formal context for \emph{relative
  adjunctions}, in which one or both adjoints are only defined on a
subcategory of their domain.  Furthermore, at least if $\sK$ is closed
monoidal with pseudo-pullbacks (like $\cat$), the polycategory
$\polmadj(\sK,\Omega)$ has (bicategorical) tensor and cotensor
products (the appropriate sort of ``representability'' condition for a
polycategory).

For instance, for polarized objects $A,B$ there is a polarized object
$A\tens B$ such that polarized two-variable adjunctions $(A,B)\to C$
are naturally equivalent to polarized \emph{one}-variable adjunctions
$A\tens B \to C$.  This universal property, like most others, tells us
how to construct $A\tens B$, as follows.  A polarized adjunction
$A\tens B \to C$ consists of morphisms $(A\tens B)\p \to C\p$ and
$C\m \to (A\tens B)\m$ together with a certain isomorphism; whereas in
a polarized two-variable adjunction $(A,B)\to C$ as above we can apply
the internal-hom isomorphism to obtain
\begin{mathpar}
  f:A\p\otimes B\p \to C\p\and
  \gtil :C\m \to [A\p,B\m]\and
  \htil :C\m\to [B\p,A\m].
\end{mathpar}
Comparing the two suggests $(A\tens B)\p = A\p \otimes B\p$ and
$(A\tens B)\m = [A\p,B\m] \times [B\p,A\m]$.  The first is correct,
but the second is not quite right: to incorporate the \emph{two}
isomorphisms of a two-variable adjunction, we have to let
$(A\tens B)\m$ be the pseudo-pullback
$[A\p,B\m] \times^{\mathrm{ps}}_{[A\p\otimes B\p,\Omega]} [B\p,A\m]$.
The third datum is the composite
\begin{align*}
  % (A\tens B)\p &= A\p \otimes B\p \\
  % (A\tens B)\m &= [A\p,B\m] \times_{[A\p\otimes B\p,\Omega]} [B\p,A\m]\\
 \ul{A\tens B} = \scriptstyle \Big(\big([A\p,B\m] \times^{\mathrm{ps}}_{[A\p\otimes B\p, \Omega]} [B\p,A\m]\big)\otimes A\p \otimes B\p
    \to [A\p\otimes B\p, \Omega] \otimes A\p \otimes B\p
    \to \Omega\Big).
\end{align*}
%and we can show that the remaining isomorphisms correspond.
There is a similar ``cotensor product'' $\cotens$ such that polarized
$(1,2)$-variable adjunctions $A\to (B,C)$ are equivalent to polarized
adjunctions $A \to B\cotens C$.  We also have duals defined by
$(A\p,A\m,\uA)\d = (A\m,A\p,\uA\sigma)$, where $\sigma$ is
transposition of inputs; note that $\rep{\cA\op}= \rep{\cA}\d$.  Thus,
the horizontal 2-category of $\polmadj(\sK,\Omega)$ is actually a
\emph{$\ast$-autonomous 2-category}\footnote{The tensor product is
  only bicategorically associative and unital.  Fortunately, we can
  avoid specifying all the coherence axioms involved in an explicit
  up-to-isomorphism $\ast$-autonomous structure on a monoidal
  bicategory by simply noting that we have a 2-polycategory with
  tensor and cotensor products that satisfy an up-to-equivalence
  universal property.  As usual, structure that is characterized by a
  universal property is automatically ``fully coherent''.

  Our $\ast$-autonomous 2-categories are unrelated to the ``linear
  bicategories'' of~\cite{cks:linbicat}, which are instead a
  ``horizontal'' or ``many-objects''
  categorification.}~\cite{barr:staraut}.

It turns out that this structure is a categorification of a
well-studied one.  If $\sK$ is a closed symmetric monoidal
\emph{1-category} with pullbacks, then all the isomorphisms degenerate
to equalities, and the $\ast$-autonomous category of
``$\Omega$-polarized objects'' is precisely the \emph{Chu
  construction}~\cite{chu:construction,chu:constr-app,barr:chu-history}
$\chu(\sK,\Omega)$.  Thus, the horizontal 2-category of
$\polmadj(\sK,\Omega)$ is a \emph{2-Chu construction}
$\tchu(\sK,\Omega)$, while the whole double category
$\polmadj(\sK,\Omega)$ can be called a \emph{double Chu construction};
we denote it by $\dchu(\sK,\Omega)$.  Thus in particular we have
$\dchu(\cat,\Set) = \polmadj(\cat,\Set)$.

This connection also suggests other applications of
$\tchu(\cat,\Set)$.  As a categorification of the prototypical 1-Chu
construction $\chu(\Set,2)$, which is an abstract home for many
concrete dualities, we may expect $\tchu(\cat,\Set)$ to be an abstract
home for concrete 2-categorical dualities.  For instance,
Gabriel-Ulmer duality~\cite{gu:locpres} between finitely complete
categories and locally finitely presentable categories sits inside
$\tchu(\cat,\Set)$ just as Stone duality between Boolean algebras and
Stone spaces sits inside $\chu(\Set,2)$~\cite{pbb:cat-duality}.  There
are other applications as well; see \cref{sec:2-chu}.

There remains, however, the problem of constructing
$\dchu(\sK,\Omega)=\polmadj(\sK,\Omega)$ in general: we need a
systematic way to deal with all the isomorphisms.  For instance, in
defining a $(2,2)$-variable adjunction we wrote
\[ \cD(f(c,a,b),d) \cong \cC(g(a,b,d),c) \cong \cB(b,h(a,c,d)) \cong
  \cA(a,k(c,d,b)) \] but there is no justifiable reason for
privileging these three isomorphisms over all the ${4\choose 2} = 6$
possible pairwise isomorphisms; what we really mean is that these four
profunctors are ``all coherently isomorphic to each other''.  There
are many ways to deal with this, but a particularly elegant approach
is to first formulate a ``lax'' version of the structure in which the
isomorphisms are replaced by directed transformations.  This clarifies
exactly how the isomorphisms ought to be composed, since the
directedness imposes a discipline that allows only certain composites.

In our case, we choose to regard the above family of coherent
isomorphisms as a ``morphism'' relating the four profunctors, and the
natural way to separate the four into domain and codomain is by
copying the analogous division for the multivariable adjunction
itself, with $\cA,\cB$ in the domain and $\cC,\cD$ in the codomain:
\[ \big(\cA(a,k(c,d,b)), \, \cB(b,h(a,c,d))\big) \to
  \big(\cC(g(a,b,d),c) ,\,\cD(f(c,a,b),d)\big). \] Thus, these
morphisms must themselves live in some polycategory.  This suggests
that the ``lax 2-Chu construction'' should apply to a 2-category \sK
containing an object $\Omega$ that is an internal polycategory, with
the ordinary 2-Chu construction recovered by giving $\Omega$ a sort of
``discrete'' polycategory structure in which a morphism
$(\phi,\psi) \to (\xi,\ze)$ consists of a coherent family of
isomorphisms between $\phi,\psi,\xi,\ze$.

This is indeed what we will do.  (We will also generalize in a couple
of other ways, replacing $\Omega$ by a not-necessarily-representable
presheaf, and enhancing the output to an indexed family of
polycategories rather than a single one.)  Intriguingly, it turns out
that while the 2-Chu construction yields a polycategory that is
representable under certain assumptions on $\sK$, the \emph{lax} 2-Chu
construction yields a polycategory that can naturally be shown to be
representable under \emph{different} assumptions on $\sK$.  Moreover,
it is also well-known under a different name: it is one of the
categorical \emph{Dialectica
  constructions}~\cite{depaiva:dialectica,depaiva:dialectica-like,paiva:dialectica-chu}.

From a higher-categorical perspective, our lax 2-Chu construction has
categorified the ordinary Chu construction in two ways.  The latter
involves equalities, a 0-categorical structure.  We first replaced
these by isomorphisms, a groupoidal or ``$(1,0)$-categorical''
structure.  Then we made them directed, yielding a 1-categorical or
$(1,1)$-categorical structure.  By contrast, the Dialectica
construction is usually formulated at the other missing vertex
involving \emph{posets}, a.k.a.\ $(0,1)$-categories (though
1-categorical versions do appear in the literature,
e.g.~\cite{biering:dialectica,hofstra:dialectica}).

Because the representability conditions on the lax and pseudo 2-Chu
constructions are different, the Dialectica and Chu constructions,
though obviously bearing a family
resemblance~\cite{paiva:dialectica-chu}, have not previously been
placed in the same abstract context.  The polycategorical perspective
allows us to exhibit them as both instances of one ``2-Chu-Dialectica
construction'', which moreover includes the polycategory of
(polarized) multivariable adjunctions at the other vertex.  The first
introduction to this paper in \cref{sec:intro-dc}, which you can go
back and read now if you skipped it the first time, reverses the flow
of motivation by \emph{starting} with the question of how to compare
the Chu and Dialectica constructions.

\begin{rmk}\label{rmk:00}
  There is one small fly in the ointment.  The ``lax
  2-Chu-Dialectica'' construction that we will describe is
  \emph{strict}: it expects its input to involve strict
  2-multicategories and 2-polycategories and produces a similarly
  strict output.  This is convenient not just because it is easier,
  but because we can obtain the double-polycategorical version by
  appling it directly to internal categories.  However, there is one
  place where it is not fully satisfactory, involving the question of
  what a ``$(0,0)$-variable adjunction'' should be.

  This question is not answered by~\cite{cgr:cyclic}: the $(0,0)$-ary
  morphisms are the one place where a polycategory with duals contains
  more information than a cyclic multicategory.  Duals allow
  representing any $(n,m)$-ary morphism as an $(n+m-1,1)$-ary
  morphism, but only if $n+m>0$.  Thus, the underlying cyclic
  multicategory of a polycategory only remembers the $(n,m)$-ary
  morphisms for $n+m>0$.

  I claim that a $(0,0)$-variable adjunction should be simply \emph{a
    set}.  There are many ways to argue for this, including the
  following:
  \begin{itemize}
  \item The only way to produce a $(0,0)$-ary morphism in a polycategory is to compose a $(0,1)$-ary morphism with a $(1,0)$-ary one.
    Now a $(0,1)$-variable adjunction $a_1:()\to \cA$ and a $(1,0)$-variable adjunction $a_2:\cA\to ()$ are both just objects of $\cA$, one ``regarded covariantly'' and the other ``regarded contravariantly''.
    What can we get naturally from two such objects?
    Obviously, the hom-set $\cA(a_2,a_1)$.
  \item The unit object of the $\ast$-autonomous 2-category $\tchu(\cat,\Set)$ is $(1,\Set,\id_\Set)$, and its counit is $(\Set,1,\id_\Set)$.
    This can be seen by analogy to the 1-Chu construction, or by checking their universal property with respect to $(n,m)$-ary morphisms for $n+m>0$.
    But if these universal properties extend to $(0,0)$-ary morphisms, then a $(0,0)$-ary morphism must be the same as a polarized adjunction $(1,\Set,\id_\Set) \to (\Set,1,\id_\Set)$, which is (up to equivalence) a set.
  \item A multivariable adjunction $(A_1,\dots,A_n)\to(B_1,\dots,B_m)$ can equivalently be defined as a profunctor $A_1\times\cdots\times A_n \hto B_1 \times\cdots \times B_m$ that is representable in each variable, and a profunctor $1\hto 1$ is just a set.
  \end{itemize}

  However, I do not know of any way to define a \emph{strict}
  2-polycategory of multivariable adjunctions in which the $(0,0)$-ary
  morphisms are sets.  The problem can be seen as follows: suppose we
  have a $(0,1)$-variable adjunction $a:() \to \cA$ (i.e.\ an object
  $a\in \cA$), a $(1,1)$-variable adjunction $f:\cA \to \cB$ (notated
  $f\p\dashv f\m$), and a $(1,0)$-variable adjunction $b:\cB \to ()$
  (i.e.\ an object $b\in \cB$).  The composite $f\circ a : () \to \cB$
  can seemingly only be the object $f\p(a) \in \cB$, and hence
  $b\circ (f\circ a)$ must be the hom-set $\cB(f\p(a),b)$.  But the
  composite $b\circ f : \cA \to ()$ can seemingly only be the object
  $f\m(b)\in \cA$, and hence $(b\circ f) \circ a$ must be the hom-set
  $\cA(a,f\m(b))$, which is only \emph{isomorphic} to $\cB(f\p(a),b)$
  rather than equal to it.\footnote{One of the referees pointed out
    that at the level of the underlying 1-category $\chuz(\cat,\Set)$
    we could define the set of $(0,0)$-ary morphisms to be the set of
    \emph{isomorphism classes} of sets.  But this wouldn't work for
    the full 2-category $\tchu(\cat,\Set)$, since the composition
    functors yielding $(0,0)$-ary output must also act on
    non-invertible 2-cells.}

  In principle, it should be possible to give a ``pseudo'' version of
  the 2-Chu-Dialectica construction.  However, for now we simply
  ignore this question by defining the $(0,0)$-ary hom-category
  ``incorrectly'' to be the terminal category rather than $\Set$.
  Since $(0,0)$-ary morphisms in a polycategory cannot be composed
  with anything else (they have no objects to compose along), it is
  always possible to brutalize a polycategory by declaring there to be
  exactly one $(0,0)$-ary morphism without changing anything else
  (which, as we will see, can also be described as following a
  round-trip pair of adjoint functors through cyclic multicategories).
  For the same reasons, I do not know of any real use for $(0,0)$-ary
  morphisms; so however unsatisfying this cop-out is philosophically,
  it has little practical import.
\end{rmk}

\begin{rmk}\label{rmk:duality}
  I have been rather cavalier about variance in this informal
  introduction.  In fact there are \emph{two} natural ways to define a
  ``representable'' polarized category corresponding to an ordinary
  category \cA:
  \begin{mathpar}
    \repl{\cA} = (\cA\op,\cA,\hom_\cA) \and \repr{\cA} = (\cA,\cA\op,\hom_\cA).
  \end{mathpar}
  (Of course, the two functors denoted $\hom_\cA$ above take their
  arguments in opposite orders.)  The difference is that a polarized
  adjunction $f:\repl{\cA} \to \repl{\cB}$ is an adjunction
  $f\p : \cA \toot \cB : f\m$ in which $f\p:\cA\to\cB$ is the
  \emph{left} adjoint, while a polarized adjunction
  $g:\repr{\cA} \to \repr{\cB}$ is an adjunction
  $g\p : \cA \otto \cB : g\m$ in which $g\p : \cA\to\cB$ is the
  \emph{right} adjoint.  However, in \emph{both} cases a 2-cell
  between polarized adjunctions is a mate-pair of natural
  transformations considered as pointing in the direction of the
  transformation between the \emph{right} adjoints: $f\m \to (f')\m$
  or $g\p \to (g')\p$.

  Similarly, a polarized two-variable adjunction
  $(\repl{\cA},\repl{\cB}) \to \repl{\cC}$ is an ordinary two-variable
  adjunction $(\cA,\cB)\to\cC$ as described above, with a functor
  $f\p:\cA\times \cB\to \cC$ equipped with a pair of two-variable
  right adjoints; but the 2-cells between these go in the direction of
  the induced mates between the right adjoints.  A polarized
  two-variable adjunction $(\repr{\cA},\repr{\cB}) \to \repr{\cC}$, by
  contrast, is a functor $g\p:\cA\times \cB\to \cC$ equipped with a
  pair of two-variable \emph{left} adjoints, with the 2-cells pointing
  in the direction of the mates between the ``forwards'' functors,
  $g\p \to (g')\p$.

  Of course, the two conventions carry the same information, and are
  interchanged by duality: $\repl{\cA\op} = \repr{\cA}$.  In the above
  introduction I wrote $\rep{\cA}$ for $\repl{\cA}$, since the most
  familiar examples of multivariable adjunctions (e.g.\ closed
  monoidal structures) are generally considered to point in the
  direction of their left adjoints.  However, the fact that this
  choice flips the 2-cells makes it seem less natural from an abstract
  point of view, so \emph{in the rest of the paper I will change
    notation and write $\rep{\cA}$ for $\repr{\cA}$}; this also has
  the advantage of coinciding with the orientation of multivariable
  adjunctions chosen by~\cite{cgr:cyclic}.  In particular, this means
  that a $(2,0)$-variable adjunction $(\cA,\cB) \to ()$ is now a
  mutual \emph{left} adjunction, and dually.

  Another way to ``fix'' the orientation of 2-cells would be to use
  $\dchu(\cat,\Set\op)$ instead of $\dchu(\cat,\Set)$.  Then we could
  define $\rep{\cA}$ to be $(\cA,\cA\op,\hom_\cA{}\op)$ and have
  adjunctions point in the direction of left adjoints and 2-cells in
  the direction of transformations between these left adjoints.
  However, this would have the unaesthetic consequence that the
  ``correct'' category of $(0,0)$-ary morphisms, as in \cref{rmk:00},
  would be $\Set\op$ rather than $\Set$.  There seems to be no perfect
  solution.
\end{rmk}

\subsection{Outline}
\label{sec:outline}

We begin in \cref{sec:input} by defining the abstract input (and also
the output!)\ of our 2-Chu-Dialectica construction.  In \cref{sec:2cd}
we give the construction itself (the general $(1,1)$-categorical
case).  Then in \cref{sec:rep} we show how it specializes to one of
the Dialectica constructions (the $(0,1)$-categorical case), while in
\cref{sec:rep2-chu} we show how it specializes to the Chu construction
(the $(0,0)$-categorical case).  Finally, in \cref{sec:2-chu} we
specialize to the 2-Chu construction (the $(1,0)$-categorical case)
and enhance the result to a poly double category of polarized
multivariable adjunctions, and in \cref{sec:cyclic} we connect this
construction to the cyclic multi double category of~\cite{cgr:cyclic}.

\subsection{Acknowledgments}
\label{sec:acknowledgments}

I would like to thank Andrej Bauer and Valeria de Paiva for helping me
to understand the Dialectica construction; Sam Staton for pointing out
that relative adjunctions appear as morphisms in $\tchu(\cat,\Set)$;
and Emily Riehl, Philip Hackney, and Ren\'{e} Guitart for useful
conversations and feedback.

\section{Presheaves of polycategories}
\label{sec:input}

We first recall Szabo's~\cite{szabo:polycats} definition of
polycategory.  On the logical side, polycategories are a categorical
abstraction of the structural rules of classical linear
logic~\cite{girard:ll} (identity, cut, and exchange), while on the
categorical side they are related to $\ast$-autonomous
categories~\cite{barr:staraut} (and more generally linearly
distributive categories~\cite{cs:wkdistrib}) roughly in the same way
that multicategories are related to monoidal categories.

All of our polycategories and multicategories will be symmetric, so we
often omit the adjective.  If $\Gamma$ and $\Gamma'$ are finite lists
of the same length, by an \emph{isomorphism} $\si:\Gamma\toiso\Gamma'$
we mean a permutation of $\abs{\Gamma}$ that maps the objects in
$\Gamma$ to those in $\Gamma'$, i.e.\ if $\Gamma'= (A_1,\dots,A_n)$
then $\Gamma = (A_{\si 1},\dots,A_{\si n})$.

\begin{defn}
  A \textbf{symmetric polycategory} $\P$ consists of
  \begin{itemize}
  \item A set of objects.
  \item For each pair $(\Gamma, \Delta)$ of finite lists of objects, a set $\P(\Gamma; \Delta)$ of ``polyarrows'', which we may also write
    $f:\Gamma\to\Delta$.
  \item For each $\Gamma,\Gamma',\Delta,\Delta'$, and isomorphisms $\rho : \Gamma\toiso\Gamma'$ and $\tau:\Delta\toiso\Delta'$, an action
    \[\P(\Gamma; \Delta) \to \P(\Gamma'; \Delta')\]
    written $f\mapsto \tau f \rho$, that is functorial on composition of permutations.
  \item For each object $A$, an identity polyarrow $1_A \in \P(A; A)$.
  \item For finite lists of objects $\Gamma, \Delta_1, \Delta_2, \Lambda_1, \Lambda_2, \Sigma$, and object $A$, composition maps
    \[\P(\Lambda_1, A, \Lambda_2; \Sigma) \times \P(\Gamma; \Delta_1, A, \Delta_2) \to \P(\Lambda_1, \Gamma, \Lambda_2; \Delta_1, \Sigma, \Delta_2).\]
    We write this operation infix as $\circ_A$, if there is no risk of confusion. % (e.g. $A$ does not appear anywhere else), or $\comp j i$ if $\abs{\Delta_1}=i$ and $\abs{\Lambda_1}=j$.
  \item Axioms of associativity and equivariance:
  \begin{align}
    1_A \circ_A f &= f \label{eq:polyidl}\\
    f \circ_A 1_A &= f \label{eq:polyidr}\\
    (h \circ_B g) \circ_A f &= h \circ_B (g \circ_A f) \label{eq:polyassoc1}\\
    (h \circ_B g) \circ_A f &= \sigma((h \circ_A f) \circ_B g) \label{eq:polyassoc2}\\
    h \circ_B (g \circ_A f) &= (g \circ_A (h\circ_B f))\sigma \label{eq:polyassoc3}\\
    \tau_1 g \rho_1 \circ_A \tau_2 f \rho_2 &= \tau_3 (g\circ_A f) \rho_3 \label{eq:polyeqv}
  \end{align}
  where the associativity axioms~\eqref{eq:polyassoc1}, \eqref{eq:polyassoc2}, and~\eqref{eq:polyassoc3} apply whenever both sides make sense, with $\sigma$ in~\eqref{eq:polyassoc2} and~\eqref{eq:polyassoc3} chosen to make the equation well-typed, and in~\eqref{eq:polyeqv} the six permutations are related in a straightforward way that makes the equation well-typed.
  \end{itemize}
  %
  % A \textbf{functor} between polycategories acts on objects and polyarrows preserving all the operations.
  % A \textbf{transformation} has $(1,1)$-ary components $\alpha_A \in \cQ(FA;GA)$ satisfying a naturality condition with respect to all polyarrows in the domain.
\end{defn}

\begin{rmk}
  We can, and will, identify (symmetric) multicategories with
  (symmetric) polycategories that are \textbf{co-unary}, i.e.\
  $\P(\Gamma;\Delta)$ is empty unless $\abs\Delta=1$.
\end{rmk}

\begin{defn}\label{defn:tens}
  Let $A,B$ be objects of a (symmetric) polycategory $\P$.
  \begin{itemize}
  \item A \textbf{tensor product} of $A,B$ is an object $A\tens B$ with a morphism $(A,B) \to (A\tens B)$ such that the following precomposition maps\footnote{We take advantage of the symmetry of $\P$ to place the objects with universal properties last in the domain or first in the codomain.
    In the non-symmetric case, the tensor product isomorphism should be $\P(\Gamma_1,A\tens B, \Gamma_2;\Delta)\cong \P(\Gamma_1,A, B, \Gamma_2;\Delta)$, and so on.} are isomorphisms:
    \[ \P(\Gamma,A\tens B;\Delta)\toiso \P(\Gamma,A,B;\Delta). \]
  \item A \textbf{unit} is an object $\top$ with a morphism $() \to (\top)$ such that the following precomposition maps are isomorphisms:
    \[ \P(\Gamma,\top;\Delta) \toiso \P(\Gamma;\Delta). \]
  \item A \textbf{cotensor product}\footnote{I agree with~\cite{cs:wkdistrib} that the notations for tensor and cotensor products should be visually dual, but I find $\oplus$ has too strong a connotation of direct sums to use it for the cotensor product.
      Hence $\tens$ and $\cotens$.} of $A,B$ is an object $A\cotens B$ with a morphism $(A\cotens B)\to (A,B)$ such that the following postcomposition maps are isomorphisms:
    \[ \P(\Gamma;A\cotens B,\Delta)\toiso \P(\Gamma;A,B,\Delta). \]
  \item A \textbf{counit} is an object $\bot$ with a morphism $(\bot) \to ()$ such that the following postcomposition maps are isomorphisms:
    \[ \P(\Gamma;\bot,\Delta) \toiso \P(\Gamma;\Delta). \]
  \item A \textbf{dual} of $A$ is an object $A\d$ with morphisms $\eta: () \to (A,A\d)$ and $\ep: (A\d,A)\to ()$ such that
    $\ep \circ_A \eta = 1_{A\d}$ and $\ep\circ_{A\d} \eta = 1_A$.
  \item A \textbf{strong hom} of $A,B$ is an object $A\ihom B$ with a morphism $(A\ihom B,A) \to (B)$ such that the following precomposition maps are isomorphisms:
    \[ \P(\Gamma;\Delta,A\ihom B) \toiso \P(\Gamma,A;\Delta,B). \]
    It is a \textbf{weak hom} if this holds only when $\Delta=\emptyset$.
  \end{itemize}
\end{defn}

A polycategory is \textbf{representable} if it has all tensor
products, units, cotensor products, and counits; then it is
equivalently a \textbf{linearly distributive
  category}~\cite{cs:wkdistrib}, while if it also has duals then it is
a \textbf{$\ast$-autonomous category}~\cite{barr:staraut}.  Strong
homs can be defined in terms of duals and cotensors, if both exist, as
$(A\ihom B) = (A\d \cotens B)$; while duals can be defined in terms of
strong homs and a counit as $A\d = (A \ihom \bot)$.  Finally, if duals
exist, then tensors and cotensors are interdefinable by
$(A\cotens B) = (A\d \tens B\d)\d$ and dually.

Our basic structure will be a multicategory equipped with a presheaf
of polycategories.  It may not be immediately clear what should be
meant by a \emph{presheaf} on a multicategory; the following
definition is obtained by abstracting the structure of the
``representable presheaf'' $\cC(-;A)$ for an object $A$.

\begin{defn}
  Let $\cC$ be a symmetric multicategory and $\cA$ a category.
  An \textbf{\cA-valued presheaf on \cC} consists of:
  \begin{enumerate}
  \item An object $\cM(\Gamma)\in\cA$ for each finite list $\Gamma$ of objects of \cC.
  \item For all isomorphisms $\rho : \Gamma \toiso \Gamma'$, an action $\cM(\Gamma)\to \cM(\Gamma')$, written $x\mapsto x\rho$, that is functorial on composition of permutations.
  \item For each morphism $f\in\cC(\Gamma; A)$ and finite lists of objects $\Delta_1,\Delta_2$, an action morphism
    \[ f^*: \cM(\Delta_1,A,\Delta_2) \to \cM(\Delta_1,\Gamma,\Delta_2) \]
  \item Any $1_A^*$ is the identity, and the following diagrams commute for any morphisms $f\in \cC(\Gamma;A)$, $g\in \cC(\Lambda;B)$, and $h\in \cC(\Sigma_1,A,\Sigma_2;C)$.
    \[
      \begin{tikzcd}
        \cM(\Delta_1,A,\Delta_2,B,\Delta_3) \ar[r,"{f^*}"] \ar[d,"{g^*}"'] &
        \cM(\Delta_1,\Gamma,\Delta_2,B,\Delta_3) \ar[d,"{g^*}"]\\
        \cM(\Delta_1,A,\Delta_2,\Lambda,\Delta_3) \ar[r,"{f^*}"'] &
        \cM(\Delta_1,\Gamma,\Delta_2,\Lambda,\Delta_3)
      \end{tikzcd}
    \]
    \[
      \begin{tikzcd}
        \cM(\Delta_1,C,\Delta_2) \ar[r,"{h^*}"] \ar[dr,"{(h\circ_A f)^*}"'] &
        \cM(\Delta_1,\Sigma_1,A,\Sigma_2,\Delta_2) \ar[d,"{f^*}"]\\
        & \cM(\Delta_1,\Sigma_1,\Gamma,\Sigma_2,\Delta_2)
      \end{tikzcd}
    \]
  \item The following diagrams commute for any morphism $f\in \cC(\Gamma;A)$ and permutations $\rho:\Gamma\toiso\Gamma'$, and $\tau : (\Delta_1,A,\Delta_2) \toiso (\Delta_1',A,\Delta_2')$ where $\tau$ sends one of the notated copies of $A$ to the other one, and $\tau'$ treats $\Gamma$ as a block replacing $A$ in $\tau$:
    \[
      \begin{tikzcd}
        \cM(\Delta_1,A,\Delta_2) \ar[r,"{f^*}"] \ar[dr,"{(f\rho)^*}"'] &
        \cM(\Delta_1,\Gamma,\Delta_2) \ar[d,"\rho"] \\
        & \cM(\Delta_1,\Gamma',\Delta_2)
      \end{tikzcd}
      \quad
      \begin{tikzcd}
        \cM(\Delta_1,A,\Delta_2) \ar[r,"\tau"] \ar[d,"{f^*}"'] &
        \cM(\Delta_1',A,\Delta_2') \ar[d,"{f^*}"] \\
        \cM(\Delta_1,\Gamma,\Delta_2) \ar[r,"\tau'"'] &
        \cM(\Delta_1',\Gamma,\Delta_2')
      \end{tikzcd}
    \]
  \end{enumerate}
\end{defn}

\begin{egs}\ 
  \begin{enumerate}
  \item As suggested above, any $A\in \cC$ gives rise to a \textbf{representable} $\Set$-valued presheaf defined by $\yon_A(\Gamma) = \cC(\Gamma;A)$.
  \item For any presheaf \cM and any finite list of objects $\Delta$, there is a \textbf{shifted} presheaf $\cM[\Delta]$ defined by $\cM[\Delta](\Gamma) = \cM(\Gamma,\Delta)$.
  \end{enumerate}
\end{egs}

\begin{rmk}\label{thm:pshf-varieties}
  Presheaves on multicategories can be reformulated in several ways:
  \begin{enumerate}
  \item When a multicategory \cC is regarded as a co-unary polycategory, a $\Set$-valued presheaf on \cC is equivalently a \emph{module} over \cC in the sense of~\cite{hyland:pfthy-abs} whose nonempty values are all of the form $\cM(\Gamma;)$.
  \item A multicategory \cC equipped with a $\Set$-valued presheaf \cM can equivalently be considered as a polycategory that is \emph{co-subunary}, i.e.\ where morphisms have codomain arity 0 or 1: we define $\cC(\Gamma;) = \cM(\Gamma)$.
    The presheaf \cM is representable if and only if this polycategory has a ``counit in the co-subunary sense'' $\cC(\Gamma;\bot) \cong \cC(\Gamma;\,)$.\label{item:cosubunary}
  \item An \cA-valued presheaf on a multicategory \cC can equivalently be defined as an ordinary functor $(\bF_{\ten}\cC)\op\to \cA$, where $\bF_{\ten}\cC$ is the free symmetric strict monoidal category generated by \cC, whose objects are finite lists of objects of \cC.\label{item:pshf3}
  \end{enumerate}
  Formulation~\ref{item:pshf3} implies that the category $\psh(\cC)$
  of $\Set$-valued presheaves on \cC admits a Day-convolution monoidal
  structure~\cite{day:closed}.  By the monoidal Yoneda lemma, it
  follows that morphisms $(\yon_{A_1},\dots,\yon_{A_n})\to \cM$ in the
  underlying multicategory of $\psh(\cC)$ are in natural bijection
  with elements of $\cM(A_1,\dots A_n)$.  Accordingly, we will
  sometimes abuse notation by writing $x:(A_1,\dots ,A_n) \to \cM$
  instead of $x \in \cM(A_1,\dots,A_n)$, with the presheaf action
  similarly denoted by composition, $f^* x = x\circ f$.
\end{rmk}

Logically, a presheaf on a multicategory represents a ``logic over a
type theory'': we have terms\footnote{For an ordinary symmetric
  multicategory as in our case, the base type theory in question is an
  intuitionistic linear one; by instead using a \emph{cartesian}
  multicategory with an appropriate notion of presheaf we would model
  an intuitionistic nonlinear type theory.} $\Gamma\types t:A$ for the
morphisms of \cC, together with an additional judgment form
``$\Gamma\types \phi \isprop$'' for the elements of the presheaf, with
a substitution action by the terms.  The expected structure of
\emph{entailment} between such propositions can be modeled by choosing
an appropriate target category \cA other than $\Set$, depending on the
desired kind of logic.

In our case, we want the logic to be classical linear logic, so we
consider multicategories \cC equipped with a \emph{presheaf of
  polycategories}, which we generally denote $\Omega$.  Note that
$\Omega$ is equivalently an internal polycategory object in
$\psh(\cC)$.  Logically, the objects of \cC correspond to
\emph{types}, the morphisms correspond to \emph{terms}
\[ x_1:A_1, \dots, x_n:A_n \types t:B, \]
the elements of $\Omega(\Gamma)$ correspond to \emph{predicates}
\[ x_1:A_1, \dots, x_n:A_n \types \phi \isprop, \]
and the morphisms in $\Omega(\Gamma)$ correspond to \emph{sequents} or \emph{entailments} in context:
\[ x_1:A_1, \dots, x_n:A_n \mid \phi_1,\dots,\phi_m \types \psi_1,\dots,\psi_k. \]
Note that each $\phi_i$ and $\psi_j$ depends \emph{separately} linearly on the context: each variable $x_k$ is ``used exactly once'' in \emph{each} $\phi_i$ and $\psi_j$.

Following this intuition, we refer to such a pair $(\cC,\Omega)$ as a
\textbf{virtual linear hyperdoctrine}.  In general a
hyperdoctrine~\cite{lawvere:adjointness,lawvere:comprehension} is an
indexed category whose base category represents the types and terms in
a type theory and whose fibers represent the predicates and sequents
in a first-order logic over that type theory.  The word ``virtual'' is
used by analogy to~\cite{cs:multicats} and indicates that nothing
corresponding to the type constructors or logical connectives or
quantifiers is present; we have only the structural rules.  (The lack
of even finite products of types is what forces us to allow predicates
to depend on finite lists of types rather than single ones.)  Note
that we do not assume our polycategories $\Omega(\Gamma)$ to be
poly-posets; as for Lawvere, the fibers of our hyperdoctrine can
distinguish between different ``proofs'' with the same domain and
codomain.

\section{Dimension $(1,1)$: the 2-Chu-Dialectica construction}
\label{sec:2cd}

Let $\Omega$ be a presheaf of polycategories on a multicategory \cC;
we will describe another presheaf of polycategories $\adjcom$ on \cC.
For a finite list of objects $\Gamma$, an element of $\adjcom(\Gamma)$
is a triple $(\phi\p, \phi\m, \uphi)$, where $\phi\p,\phi\m$ are
objects of \cC and $\uphi \in \Omega(\Gamma,\phi\m,\phi\p)$.  The
presheaf action is induced in the obvious way from that of \Omega.

A morphism $(\phi_1,\dots,\phi_m)\to (\psi_1,\dots,\psi_n)$ in $\adjcom(\Gamma)$ consists of:
\begin{enumerate}
\item Morphisms in \cC:
  \begin{align*}
    f_j : (\phi_1\p, \dots, \phi_m\p, \psi_1\m, \dots, \widehat{\psi_j\m},\dots \psi_n\m) &\longrightarrow \psi_j\p\\
    g_i : (\phi_1\p , \dots, \widehat{\phi_i\p},\dots \phi_m\p , \psi_1\m, \dots, \psi_n\m) &\longrightarrow \phi_i\m
  \end{align*}
  for $1\le j\le n$ and $1\le i\le m$ (where $\widehat{\chi}$ means that $\chi$ is omitted from the list).
  We call these the \emph{primary components}.
\item A morphism in $\Omega(\Gamma,\phi_1\p,\dots,\phi_m\p,\psi_1\m,\dots,\psi_n\m)$:
%  \[ \al : ((\uphi_1\circ_{\phi_1\m} g_1),\dots,(\uphi_m\circ_{\phi_m\m} g_m)) \to ((\upsi_1 \circ_{\psi_1\p} f_1),\dots, (\upsi_n \circ_{\psi_n\p} f_n)) \]
  \[ \al : (g_1^* \uphi_1,\dots,g_m^* \uphi_m) \to (f_1^*\upsi_1,\dots, f_n^* \upsi_n) \]
  We call this the \emph{secondary component}.
\end{enumerate}
We have omitted to notate the action of symmetric groups needed to
make all the above composites live in the right place.  We will
continue to do the same below; in all cases there is only one possible
permutation that could be meant.  Note that $\Gamma$ appears in the
domain of $\al$, but \emph{not} in the domains of $f_j$ and $g_i$.
Also, if $n=m=0$, the only datum is the $(0,0)$-ary morphism $\al$ in
$\Omega(\Gamma)$.

The symmetric and presheaf actions on morphisms in $\adjcom(\Gamma)$
is obvious.  To define composition in the polycategory
$\adjcom(\Gamma)$, suppose we have another such morphism
$(\xi_1,\dots,\xi_p) \to (\ze_1,\dots,\ze_q)$, where
$\psi_{j_0} = \xi_{k_0}$, with primary components
\begin{align*}
  r_l : (\xi_1\p, \dots, \xi_p\p, \ze_1\m, \dots, \widehat{\ze_l\m},\dots \ze_q\m) &\longrightarrow \ze_l\p\\
  s_k : (\xi_1\p , \dots, \widehat{\xi_k\p},\dots \xi_p\p , \ze_1\m, \dots, \ze_q\m) &\longrightarrow \xi_k\m
\end{align*}
for $1\le k\le p$ and $1\le l\le q$, and secondary component
% \[ \be : ((\uxi_1\circ_{\xi_1\m} s_1),\dots,(\uxi_p\circ_{\xi_p\m} s_p)) \to ((\uze_1 \circ_{\ze_1\p} r_1),\dots, (\uze_q \circ_{\ze_q\p} r_q)), \]
\[ \be : (s_1^* \uxi_1,\dots,s_p^* \uxi_p) \to (r_1^* \uze_1,\dots, r_q ^* \uze_q). \]
For conciseness we write
% \begin{align*}
  \(\vec\phi\p = (\phi_1\p,\dots,\phi_m\p)\) and
  \(\vec\phi_{\neq j}\p = (\phi_1\p,\dots,\widehat{\phi_j\p},\dots\phi_m\p) \).
%\end{align*}
%and so on.

  The desired composite should be (up to symmetric action) a morphism
  $(\vec\phi,\vec\xi_{\neq k_0}) \to (\vec \psi_{\neq j_0},\vec\ze)$.
  We take its primary components to be (up to symmetric action)
\begin{align*}
  f_j \circ_{\psi_{j_0}\m = \xi_{k_0}\m} s_{k_0} : (\vec\phi\p, \vec\xi_{\neq k_0}\p, \vec\psi_{\neq j,j_0}\m ,\vec \ze\m) &\longrightarrow \psi_j\p \qquad (j\neq j_0)\\
  r_l \circ_{\xi_{k_0}\p = \psi_{j_0}\p} f_{j_0} : (\vec\phi\p, \vec\xi_{\neq k_0}\p, \vec\psi_{\neq j_0}\m ,\vec \ze_{\neq l}\m) &\longrightarrow \ze_l\p\\
  g_i \circ_{\psi_{j_0}\m = \xi_{k_0}\m} s_{k_0} : (\vec\phi_{\neq i}\p, \vec\xi_{\neq k_0}\p, \vec\psi_{\neq j_0}\m ,\vec \ze\m) &\longrightarrow \phi_i\m\\
  s_k \circ_{\xi_{k_0}\p = \psi_{j_0}\p} f_{j_0} : (\vec\phi\p, \vec\xi_{\neq k,k_0}\p, \vec\psi_{\neq j_0}\m ,\vec \ze\m) &\longrightarrow \xi_k\m \qquad (k\neq k_0).
\end{align*}
For the secondary component, first note that we have
\begin{alignat*}{2}
  s_{k_0}^* \al &: (s_{k_0}^* g_1^* \uphi_1,\dots, s_{k_0}^* g_m^* \uphi_m ) &&\to (s_{k_0}^* f_1^* \upsi_1,\dots, s_{k_0}^*  f_n^* \upsi_n)\\
  f_{j_0}^* \be &: (f_{j_0}^* s_1^* \uxi_1,\dots, f_{j_0}^* s_p^*\uxi_p) &&\to (f_{j_0}^* r_1^* \uze_1,\dots, f_{j_0}^*  r_q^* \uze_q )
  % \al \circ_{\psi_{j_0}\m = \xi_{k_0}\m} s_{k_0} &: (\uphi_1 g_1 s_{k_0},\dots, \uphi_m g_m s_{k_0}) \to (\upsi_1 f_1 s_{k_0},\dots,\upsi_n f_n s_{k_0})\\
  % \be \circ_{\xi_{k_0}\p = \psi_{j_0}\p} f_{j_0} &: (\uxi_1 s_1 f_{j_0},\dots,\uxi_p s_p f_{j_0}) \to (\uze_1 r_1 f_{j_0},\dots, \uze_q r_q f_{j_0})
\end{alignat*}
Now since $\upsi_{j_0} = \uxi_{k_0}$, by associativity of the presheaf
action we have
$s_{k_0}^* f_{j_0}^* \upsi_{j_0} = f_{j_0}^* s_{k_0}^* \uxi_{k_0} $.
Thus, we can compose these two morphisms along this common object to
get a morphism
\begin{multline*}
(s_{k_0}^* g_1^* \uphi_1,\dots, s_{k_0}^* g_m^* \uphi_m,  f_{j_0} ^* s_1 ^*  \uxi_1,\dots,\widehat{  f_{j_0}^* s_{k_0}^* \uxi_{k_0}},\dots,  f_{j_0}^* s_p^* \uxi_p)\\ \too
(  s_{k_0}^* f_1^* \upsi_1,\dots,\widehat{  s_{k_0}^* f_{j_0}^* \upsi_{j_0}},\dots,  s_{k_0}^* f_n^* \upsi_n,  f_{j_0}^* r_1^* \uze_1,\dots,   f_{j_0}^* r_q^* \uze_q)
\end{multline*}
This is the secondary component of our desired composite in
$\adjcom(\Gamma)$.  The associativity, equivariance, and so on of this
operation follow from the analogous properties in \Omega.

\section{Dimension $(0,1)$: the Dialectica construction}
\label{sec:rep}

Dialectica and Chu constructions generally yield a \emph{monoidal}
category (perhaps linearly distributive, closed, or $\ast$-autonomous)
or a fibration of such.  Our construction produces a fully virtual
(multi/poly-categorical) structure, so to compare it to the usual
constructions we need to consider its representability conditions,
which are induced from similar conditions on both \cC and \Omega.

\begin{defn}\label{defn:2-properties}
  A presheaf of polycategories $\Omega$ on \cC has \textbf{tensors}, a
  \textbf{unit}, \textbf{cotensors}, a \textbf{counit},
  \textbf{duals}, or \textbf{homs} if the polycategories
  $\Omega(\Gamma)$ have the relevant structure and it is preserved by
  the presheaf action (up to isomorphism).
\end{defn}

\begin{lem}
  Let \Omega be a presheaf of polycategories on a multicategory \cC.
  \begin{enumerate}
  \item If $\Omega$ is co-unary with tensors and a unit, it is equivalently a presheaf of symmetric monoidal categories.
  \item If $\Omega$ is co-unary with tensors, a unit, and homs, it is equivalently a presheaf of closed symmetric monoidal categories.
  \item If $\Omega$ has tensors, a unit, cotensors, and a counit, it is equivalently a presheaf of linearly distributive categories.
  \item If $\Omega$ has tensors, a unit, cotensors, a counit, and weak homs, it is equivalently a presheaf of ``full multiplicative categories''~\cite{cs:pfth-bill}: linearly distributive categories whose tensor (but not cotensor) monoidal structure is closed.
  \item If $\Omega$ has tensors, a unit, cotensors, a counit, and duals, it is equivalently a presheaf of $\ast$-autonomous categories.\qed
  \end{enumerate}
\end{lem}

We keep the notations $\tens,\top,\cotens,\bot,\ihom$ for such
structures on $\Omega$, but to avoid confusion we will instead denote
tensor products, units, and internal-homs in the multicategory \cC by
$A\otimes B$, $I$, and $[A,B]$ respectively.  We also require the
following assumption:

\begin{defn}
  A tensor product $(A,B) \to A\otimes B$ in a multicategory \cC is
  \textbf{preserved} by an \cA-valued presheaf \Omega if the induced
  maps \( \Omega(\Gamma,A\otimes B) \toiso \Omega(\Gamma,A,B) \) are
  all isomorphisms.  Similarly, $\Omega$ preserves a a unit
  $() \to (I)$ if it induces isomorphisms
  $\Omega(\Gamma,I) \toiso \Omega(\Gamma)$.
\end{defn}

\begin{lem}
  If \cC has all tensor products and a unit, then an \cA-valued
  presheaf $\Omega$ that preserves all such tensors and the unit is
  equivalently an ordinary presheaf on the underlying ordinary
  category of \cC.\qed
\end{lem}

\begin{eg}\label{eg:rep-pshf}
  Just as an internal category in \cC induces a representable presheaf
  of categories on \cC, an internal polycategory induces a presheaf of
  polycategories, and similarly for any other such structure.  Such
  presheaves always preserve all tensor products and units in \cC.
\end{eg}

\begin{thm}\label{thm:2tens}
  Let $\cC$ be a multicategory and \Omega a presheaf of polycategories on \cC.
  \begin{enumerate}
  \item $\adjcom$ always has duals.
  \item If $\cC$ has a unit and a terminal object, and $\Omega$ preserves the unit of \cC and has a unit (resp.\ a counit), then $\adjcom$ has a unit (resp.\ a counit).
  \item If \cC has tensor products, homs, and binary cartesian products, and $\Omega$ preserves the tensor products of \cC and has tensor products, cotensor products, or strong or weak homs, then $\adjcom$ also has tensor products, cotensor products, or strong or weak homs respectively.
  \end{enumerate}
\end{thm}
\begin{proof}
  The dual of $(\phi\p,\phi\m,\uphi)$ is
  \begin{equation*}
    (\dual\phi)\p = \phi\m\qquad
    (\dual\phi)\m = \phi\p\qquad
    (\Gamma,\phi\p,\phi\m) \toiso (\Gamma,\phi\m,\phi\p) \xto{\uphi} \Omega.
  \end{equation*}
  The unit is defined by $\top\p = I$ and $\top\m = 1$ (the terminal
  object), with $\ul{\top} = \top$ in $\Omega(\Gamma,I,1)$, while the
  counit similarly has $\bot\p=1$ and $\bot\m=I$.

  The tensor product of $(\phi\p,\phi\m,\uphi)$ and $(\psi\p,\psi\m,\upsi)$ is
  \begin{align*}
    (\phi\tens\psi)\p &= (\phi\p \otimes \psi\p)\\
    (\phi\tens\psi)\m &= [\phi\p, \psi\m] \times [\psi\p, \phi\m]
  \end{align*}
  with
  $\ul{\phi\tens\psi} \in \Omega(\Gamma,[\phi\p,\psi\m] \times
  [\psi\p,\phi\m], \phi\p\otimes \psi\p)$ induced by the universal
  property of $\phi\p\otimes \psi\p$ from the following tensor product
  in \Omega
  \begin{equation}
  \begin{array}[c]{c}\scriptstyle
    \big( (\Gamma,[\phi\p,\psi\m] \times [\psi\p,\phi\m], \phi\p,\psi\p)
    \to (\Gamma,[\phi\p,\psi\m], \phi\p,\psi\p)
    \to (\Gamma,\psi\m,\psi\p)
    \xto{\upsi} \Omega\big)\\
    \raisebox{-2pt}{$\scriptstyle\tens$}\\
    \scriptstyle\big( (\Gamma,[\phi\p,\psi\m] \times [\psi\p,\phi\m], \phi\p,\psi\p)
    \to (\Gamma,[\psi\p,\phi\m], \phi\p,\psi\p)
    \to (\Gamma,\phi\m,\phi\p)
    \xto{\uphi} \Omega\big)
  \end{array}\label{eq:2tens}
  \end{equation}
  The universal morphism $(\phi,\psi) \to(\phi\tens\psi)$ has primary components
  \begin{align*}
    (\phi\p, \psi\p) &\to (\phi\p \otimes \psi\p)\\
    ([\phi\p, \psi\m] \times [\psi\p, \phi\m], \phi\p) &\to ([\phi\p, \psi\m],\phi\p) \to \psi\m\\
    ([\phi\p, \psi\m] \times [\psi\p, \phi\m], \psi\p) &\to ([\psi\p, \phi\m],\psi\p) \to \phi\m
  \end{align*}
  and its secondary component exhibits the universal property of the tensor product~\eqref{eq:2tens}.

  We check the universal property of $\phi\tens\psi$ in the case of a
  morphism $(\phi\tens\psi, \xi) \to (\ze)$ in $\adjcom(\Gamma)$; the
  general case is the same but the notation is more tedious.  Such a
  morphism has primary components
  \begin{align*}
    f &: (\phi\p\otimes\psi\p, \xi\p) \to \ze\p\\
    g &: (\phi\p\otimes\psi\p,\ze\m) \to\xi\m\\
    h &: (\xi\p,\ze\m) \to [\phi\p, \psi\m] \times [\psi\p, \phi\m]
  \end{align*}
  and a secondary component
  \begin{equation}
    (\ul{\phi\tens\psi} \circ h, \uxi \circ g) \to (\uze \circ f).\label{eq:2tens-2mor}
  \end{equation}
  Composing $f,g,h$ with the components of
  $(\phi,\psi) \to(\phi\tens\psi)$ exactly implements the universal
  properties of $\phi\p \otimes \psi\p$ and
  $[\phi\p, \psi\m] \times [\psi\p, \phi\m]$, yielding a bijective
  correspondence to quadruples of morphisms
  \begin{alignat*}{2}
    f' &: (\phi\p,\psi\p, \xi\p) \to \ze\p &\qquad
    h' &: (\phi\p,\xi\p,\ze\m) \to \psi\m\\
    g' &: (\phi\p,\psi\p,\ze\m) \to\xi\m &
    h'' &: (\psi\p,\xi\p,\ze\m) \to \phi\m
  \end{alignat*}
  which are exactly as required for a morphism
  $(\phi,\psi,\xi) \to (\ze)$.  Similarly,
  composing~\eqref{eq:2tens-2mor} with the secondary component of
  $\phi\tens\psi$ simply composes $\ul{\smash{\phi\tens\psi}}$ with
  the universal map $(\phi\p, \psi\p) \to (\phi\p \otimes \psi\p)$,
  exposing the tensor product~\eqref{eq:2tens}, and then composes with
  the morphism exhibiting the universal property of the latter.

  Dually, the cotensor product of $\phi$ and $\psi$ is
  \begin{align*}
    (\phi\cotens\psi)\p &= [\phi\m, \psi\p] \times [\psi\m, \phi\p]\\
    (\phi\cotens\psi)\m &= (\phi\m \otimes \psi\m)
  \end{align*}
  with $\ul{\phi\cotens\psi}$ defined similarly using $\cotens$ in
  \Omega instead of $\tens$, while the counit has $\bot\p=1$ and
  $\bot\m=I$, with $\ul{\bot}=\bot$.  And the hom $\phi\multimap\psi$
  (strong or weak according to that of $\Omega$) is
  \begin{align*}
    (\phi\multimap \psi)\p &= [\phi\p,\psi\p] \times [\psi\m,\phi\m]\\
    (\phi\multimap\psi)\m &= (\phi\p \otimes \psi\m)
  \end{align*}
  with $\ul{\phi\multimap\psi}$ defined using $\multimap$ in \Omega.
\end{proof}

\begin{eg}
  The original Dialectica construction focused on what in our notation
  is the ``empty context'' component $\adjcom\zz$.  For instance,
  applying \cref{thm:2tens} to \cref{eg:rep-pshf} we see that if
  $\Omega$ is an internal closed monoidal poset in a closed symmetric
  monoidal category \cC, then $\adjcom\zz$ is a closed symmetric
  monoidal category.  This reproduces the general Dialectica
  construction from~\cite{depaiva:multirel,paiva:dialectica-chu}.
\end{eg}

\begin{egs}
  The original construction from~\cite{depaiva:dialectica-like}
  (called $\mathbf{GC}$ in~\cite{depaiva:dialectica}) is the case when
  we have a \emph{cartesian} closed category $\cC$, with
  $\Omega =\sub(\cC)$ its subobject fibration where $\sub(\cC)(A)$ is
  the poset of subobjects of $A$, with additional structure induced
  from that of \cC:
  \begin{enumerate}
  \item As long as \cC has finite limits, $\sub(\cC)$ is a presheaf of meet-semilattices, hence in particular symmetric monoidal posets, so we can regard it as a presheaf of multicategories (i.e.\ co-unary polycategories) with tensors and units.
    Thus, $\adjc(\sub(\cC))\zz$ is a symmetric monoidal category.
  \item If $\cC$ is a Heyting category, then $\sub(\cC)$ is a presheaf of Heyting algebras, i.e.\ cartesian closed posets, so we can regard it as a presheaf of multicategories with tensors, units, and homs.
    Thus, in this case $\adjc(\sub(\cC))\zz$ is a closed symmetric monoidal category.
  \item If $\cC$ is a coherent category, then $\sub(\cC)$ is a presheaf of distributive lattices.
    Since a distributive lattice can be regarded as a linearly distributive category, we can also regard it as a polycategory that is \emph{not} co-unary, and has tensors, a unit, cotensors, and a counit.
    Thus, in this case $\adjc(\sub(\cC))\zz$ is a linearly distributive category.\label{item:dia-ldist}
  \item If in~\ref{item:dia-ldist} \cC is furthermore a Heyting category, then the polycategories $\sub(\cC)(\Gamma)$ also have weak homs, so that $\ucd{\sub(\cC)}\zz$ is a full multiplicative category, as in~\cite{depaiva:dialectica-like}.
    (In this paper we will not consider the additive fragment, i.e.\ the cartesian products and coproducts in $\cdt$, or the exponential modalities $\oc$ and $\wn$.)
  \item If $\cC$ is furthermore a \emph{Boolean} category, then $\sub(\cC)$ is a presheaf of Boolean algebras, which as linearly distributive categories are $\ast$-autonomous; thus in this case $\adjc(\sub(\cC))\zz$ is also $\ast$-autonomous.
    More generally, we can restrict to the sub-Boolean-algebras of $\neg\neg$-closed subobjects in $\sub(\cC)$; this produces the $\ast$-autonomous category $\mathsf{Dec}\,\mathbf{GC}$ from~\cite{depaiva:dialectica-like}.
  \end{enumerate}
\end{egs}

% \begin{rmk}
%   There are other categorical Dialectica constructions, such as the one called $\mathbf{DC}$ in~\cite{depaiva:dialectica} which allows the backwards arrows (but not the forwards ones) to take an extra parameter.
%   Our construction from \cref{sec:2cd} can be modified to become a generalization of $\mathbf{DC}$ instead of $\mathbf{GC}$, by requiring a cartesian action on the domains (duplication and deletion of objects, i.e.\ contraction and weakening in logic), adding the extra parameters on the backwards arrows, and restricting all 2-morphisms to be co-unary (since $\mathbf{DC}$ is only monoidal, not linearly distributive or $\ast$-autonomous).
%   Like the construction we have described in detail, this one also produces a whole virtual linear hyperdoctrine as the output; thus it is even a generalization of the fibered Dialectica constructions of~\cite{biering:dialectica,hofstra:dialectica}.
% \end{rmk}

\section{Dimension $(0,0)$: the Chu construction}
\label{sec:rep2-chu}

The Chu construction is generally defined as an operation on closed
symmetric monoidal categories equipped with an arbitrary object
$\Omega$;
see~\cite{chu:construction,chu:constr-app,barr:chu-history,pavlovic:chu-i}.
We fit this into our context with the following construction.

\begin{defn}\label{eg:frobdisc}
  Any set $X$ is the set of objects of a
  \textbf{Frobenius-discrete}\footnote{A \emph{discrete} polycategory,
    in my preferred terminology, would be one in the image of the left
    adjoint to the forgetful functor from polycategories to sets,
    i.e.\ one containing only identity arrows.} polycategory $X\sfd$,
  for which a polyarrow $(x_1,\dots,x_m) \to (y_1,\dots,y_n)$ consists
  of an element $z\in X$ such that $x_i = z$ and $y_j = z$ for all
  $i,j$.
\end{defn}

The Frobenius-discrete polycategories are equivalently the coproducts
of copies of the terminal polycategory; this motivates the name, since
the terminal (symmetric) polycategory is freely generated by a
(commutative) Frobenius algebra.  Note that a $(0,0)$-ary arrow in a
Frobenius-discrete polycategory is still determined by a single
object, even though there is no domain or codomain for that object to
appear in.

The construction $X\mapsto X\sfd$ is functorial, so any $\Set$-valued
presheaf $\Omega$ on a multicategory \cC induces a presheaf of
polycategories $\Omega\sfd$.  Applying the construction of
\cref{sec:2cd}, we obtain another presheaf of polycategories
$\adjc(\Omega\sfd)$, whose objects are triples
$(\phi\p, \phi\m, \uphi)$, where $\phi\p,\phi\m$ are objects of \bC
and $\uphi: (\Gamma,\phi\m,\phi\p) \to \Omega$.  A morphism
$(\phi_1,\dots,\phi_m)\to (\psi_1,\dots,\psi_n)$ in
$\adjc(\Omega\sfd)(\Gamma)$ consists of
\begin{align*}
  f_j : (\phi_1\p, \dots, \phi_m\p, \psi_1\m, \dots, \widehat{\psi_j\m},\dots \psi_n\m) &\longrightarrow \psi_j\p\\
  g_i : (\phi_1\p , \dots, \widehat{\phi_i\p},\dots \phi_m\p , \psi_1\m, \dots, \psi_n\m) &\longrightarrow \phi_i\m\\
  \alpha : (\phi_1\p, \dots, \phi_m\p, \psi_1\m, \dots, \psi_n\m) &\longrightarrow \Omega
\end{align*}
such that
\[ (\uphi_1\circ_{\phi_1\m} g_1) = \dots = (\uphi_m\circ_{\phi_m\m}
  g_m) = (\upsi_1 \circ_{\psi_1\p} f_1) = \dots= (\upsi_n
  \circ_{\psi_n\p} f_n) = \alpha. \] (Of course, if $m+n>0$ then
$\alpha$ is uniquely determined by the $f$'s and $g$'s, but if $m=n=0$
then $\alpha$ is the only datum.)

A Frobenius-discrete polycategory always has duals; in fact each
object is its own dual.  Thus, $\adjc(\Omega\sfd)$ also has duals, and
in particular $\adjc(\Omega\sfd)\zz$ is a polycategory with duals (in
fact it is a $\ast$-polycategory; see \cref{sec:cyclic}).  However, a
Frobenius-discrete polycategory almost never has tensors or cotensors
(see \cref{rmk:chu1}).  So we cannot obtain tensors and cotensors in
$\adjc(\Omega\sfd)$ from \cref{thm:2tens}, but we can construct them
in a different way (coinciding with the usual Chu construction).

\begin{thm}\label{thm:chu}
  Suppose \cC is a closed symmetric monoidal category with pullbacks,
  and $\Omega$ is an object of \cC (identified with its representable
  presheaf $\yon_\Omega$).  Then $\adjc(\Omega\sfd)$ has tensors, a
  unit, cotensors, and a counit (and hence is a presheaf of
  $\ast$-autonomous categories).
\end{thm}
\begin{proof}
  The tensor product of $(\phi\p,\phi\m,\uphi)$ and $(\psi\p,\psi\m,\upsi)$ is now
  \begin{align*}
    (\phi\tens\psi)\p &= (\phi\p \otimes \psi\p)\\
    (\phi\tens\psi)\m &= [\phi\p, \psi\m] \times_{[\Gamma\otimes \phi\p\otimes\psi\p, \Omega]} [\psi\p, \phi\m]
  \end{align*}
  (where $\Gamma$ denotes abusively the tensor product of all the objects in $\Gamma$),
  with
  \[ \ul{\phi\tens\psi}: (\Gamma,[\phi\p,\psi\m] \times_{[\Gamma\otimes \phi\p\otimes\psi\p, \Omega]} [\psi\p,\phi\m], \phi\p\otimes \psi\p) \to \Omega\]
  induced by the universal property of $\phi\p\otimes \psi\p$ from the \emph{common value} of the following two morphisms
  \begin{equation}
  \begin{array}[c]{c}\scriptstyle
    (\Gamma,[\phi\p,\psi\m] \times_{[\Gamma\otimes\phi\p\otimes\psi\p, \Omega]} [\psi\p,\phi\m], \phi\p,\psi\p)
    \to (\Gamma,[\phi\p,\psi\m], \phi\p,\psi\p)
    \to (\Gamma,\psi\m,\psi\p)
    \xto{\upsi} \Omega\\
    \scriptstyle (\Gamma,[\phi\p,\psi\m] \times_{[\Gamma\otimes\phi\p\otimes\psi\p, \Omega]} [\psi\p,\phi\m], \phi\p,\psi\p)
    \to (\Gamma,[\psi\p,\phi\m], \phi\p,\psi\p)
    \to (\Gamma,\phi\m,\phi\p)
    \xto{\uphi} \Omega.
  \end{array}\label{eq:2tenschu}
  \end{equation}
  Its universal morphism is defined similarly:
  \begin{align*}
    (\phi\p, \psi\p) &\to (\phi\p \otimes \psi\p)\\
    ([\phi\p, \psi\m] \times_{[\Gamma\otimes\phi\p\otimes\psi\p, \Omega]} [\psi\p, \phi\m], \phi\p) &\to ([\phi\p, \psi\m],\phi\p) \to \psi\m\\
    ([\phi\p, \psi\m] \times_{[\Gamma\otimes\phi\p\otimes\psi\p, \Omega]} [\psi\p, \phi\m], \psi\p) &\to ([\psi\p, \phi\m],\psi\p) \to \phi\m
  \end{align*}
  plus the fact that the latter two of these, when composed with
  $\upsi$ and $\uphi$ respectively, yield~\eqref{eq:2tenschu}.  For
  the universal property, a morphism $(\phi\tens\psi, \xi) \to (\ze)$
  in $\adjc(\Omega\sfd)$ now consists of morphisms in $\cC$:
  \begin{align*}
    f &: (\phi\p\otimes\psi\p, \xi\p) \to \ze\p\\
    g &: (\phi\p\otimes\psi\p,\ze\m) \to\xi\m\\
    h &: (\xi\p,\ze\m) \to [\phi\p, \psi\m] \times_{[\Gamma\otimes\phi\p\otimes\psi\p, \Omega]} [\psi\p, \phi\m]
  \end{align*}
  such that
  \begin{equation}
    \ul{\phi\tens\psi} \circ (h,1) = \uxi \circ (1,g) = \uze \circ (f,1).\label{eq:2tens-eq}
  \end{equation}
  Composing with the universal morphism again implements the universal
  property of $\phi\p\otimes\psi\p$ and $[\phi\p, \psi\m]$ and
  $[\psi\p, \phi\m]$ to get
  \begin{alignat*}{2}
    f' &: (\phi\p,\psi\p, \xi\p) \to \ze\p &\qquad
    h' &: (\phi\p,\xi\p,\ze\m) \to \psi\m\\
    g' &: (\phi\p,\psi\p,\ze\m) \to\xi\m &
    h'' &: (\psi\p,\xi\p,\ze\m) \to \phi\m
  \end{alignat*}
  as required for a morphism $(\phi,\psi,\xi) \to (\ze)$; but now $h$
  is only determined by $h'$ and $h''$ subject to a compatibility
  condition of agreement in
  ${[\Gamma\otimes\phi\p\otimes\psi\p, \Omega]}$, which means
  equivalently that $\uphi\circ h'' = \upsi\circ h'$.  This is ensured
  by the equality condition for a morphism
  $(\phi,\psi,\xi) \to (\ze)$:
  \[ \uphi \circ h'' = \upsi \circ h' = \uxi \circ g' = \uze \circ f'.
  \]
  For the rest of the equality conditions, composing the morphisms
  in~\eqref{eq:2tens-eq} with the universal morphism
  $u : (\phi\p,\psi\p) \to (\phi\p\otimes\psi\p)$, which preserves and
  reflects equalities since it is a bijection, yields
  \begin{equation}
    \ul{\phi\tens\psi} \circ h \circ u = \uxi \circ g' = \uze \circ f'.
  \end{equation}
  and $\ul{\phi\tens\psi} \circ h \circ u$ is exactly the common value $\uphi\circ h'' = \upsi\circ h'$.

  As before, the general case is analogous.
  The unit is
  \begin{mathpar}
    \top\p = I\and
    \top\m = [\Gamma,\Omega]
  \end{mathpar}
  with $\ul{\top} : (\Gamma,[\Gamma,\Omega],I) \to \Omega$ induced by
  the universal property of $I$ from the evaluation map
  $(\Gamma,[\Gamma,\Omega]) \to \Omega$.  The cotensors and the counit
  are dual.
\end{proof}

Thus, the reason the Dialectica and Chu constructions look different
is that while they are both instances of a single abstract
construction at the virtual level, they are representable for
different reasons.

\begin{rmk}\label{rmk:chu1}
  It is natural to ask what the intersection of the Dialectica and Chu
  constructions is, i.e.\ when do both \cref{thm:2tens} and
  \cref{thm:chu} apply?  The reader can check that a
  Frobenius-discrete polycategory can only have tensors and a unit, or
  cotensors and a counit, when it has exactly one object.  Thus, this
  happens if and only if $\Omega=1$ is a terminal object, in which
  case the underlying ordinary category of $\adjc(1)$ is
  $\cC\times \cC\op$.
\end{rmk}

\begin{rmk}
  For an arbitrary multicategory \cC with \Set-valued presheaf \Omega,
  even if the hypotheses of \cref{thm:chu} fail, it still makes sense
  to refer to the polycategory $\adjc(\Omega\sfd)\zz$ as a \textbf{Chu
    construction} $\chu(\cC,\Omega)$.  A similar generalized Chu
  construction taking multi-bicategories to ``cyclic''
  poly-bicategories appears in~\cite[Example 1.8(2)]{cks:polybicats},
  but the symmetric case does not appear to be in the literature.
  (The symmetric Chu construction is not simply obtained by applying
  the non-symmetric one to a symmetric input.)

  The universal property of the Chu construction described
  in~\cite{pavlovic:chu-i} also generalizes cleanly to the
  polycategorical version: $\chu$ is a right adjoint to the forgetful
  functor from $\ast$-polycategories to co-subunary polycategories
  (i.e.\ multicategories equipped with a $\Set$-valued presheaf).  The
  special case of this for $\chu(-,1)$, namely that it is a right
  adjoint to the forgetful functor from $\ast$-polycategories to
  multicategories, appears in~\cite{dh:dk-cyc-opd-v1}.
\end{rmk}

\section{Dimension $(1,0)$: the 2-Chu and double Chu constructions}
\label{sec:2-chu}

Having recovered the classical Dialectica and Chu constructions, we now categorify the latter.

\begin{defn}
  Let $\vec\phi = (\phi_1,\dots,\phi_n)$ be a list of objects of a
  category \cX.  A \textbf{clique} on $\vec\phi$ is a functor from the
  chaotic category on $n$ objects to \cX that picks out the objects of
  $\vec\phi$, i.e.\ a family of isomorphisms
  $\theta_{ij}:\phi_i \toiso \phi_j$ such that $\theta_{ii} =1$ and
  $\theta_{jk}\theta_{ij} = \theta_{ik}$.
\end{defn}

Note there is a unique clique on the empty list (this is the cop-out from \cref{rmk:00}).

\begin{lem}
  Given cliques on $(\phi_1,\dots,\phi_m)$ and $(\xi_1,\dots,\xi_n)$
  with $\phi_{j_0} = \xi_{k_0}$, there is an induced clique on
  $(\vec\phi_{\neq j_0},\vec \xi_{\neq k_0})$, and this operation is
  associative.
\end{lem}
\begin{proof}
  We take the isomorphisms among the $\phi$'s and the $\xi$'s to be
  the given ones, and the isomorphism $\phi_j \toiso \xi_k$ to be the
  composite $\phi_j \toiso \phi_{j_0} = \xi_{k_0}\toiso \xi_k$.
\end{proof}

\begin{defn}
  For any category \cX, the \textbf{Frobenius pseudo-discrete}
  polycategory $\cX\sfpd$ has the same objects as \cX, with
  $\cX\sfpd(\Gamma;\Delta)$ the set of cliques on $(\Gamma,\Delta)$.
\end{defn}

This defines a functor $(-)\sfpd$ from the 1-category of categories to
the 1-category of polycategories.  Thus, a presheaf of categories
\Omega on a multicategory \cC gives rise to a presheaf of
polycategories $\Omega\sfpd$.  We write
\[ \chuz(\cC,\Omega) = \adjc(\Omega\sfpd)\zz. \] This polycategory
will be the underlying 1-dimensional structure of our 2-Chu
construction: its objects are triples $(A\p,A\m,\uA)$ with
$\uA : (A\p,A\m)\to \Omega$, and a polyarrow
$(A_1,\dots,A_m) \to (B_1,\dots,B_n)$ in $\chuz(\cC,\Omega)$ consists
of morphisms in $\cC$:
\begin{align*}
  f\p_j : (A_1\p, \dots, A_m\p, B_1\m, \dots, \widehat{B_j\m},\dots B_n\m) &\longrightarrow B_j\p\\
  f\m_i : (A_1\p , \dots, \widehat{A_i\p},\dots A_m\p , B_1\m, \dots, B_n\m) &\longrightarrow A_i\m
\end{align*}
together with a clique on
\[ \big((\uA_1\circ_{A_1\m} f\m_1),\dots,(\uA_m\circ_{A_m\m}
  f\m_m),(\uB_1 \circ_{B_1\p} f\p_1),\dots, (\uB_n \circ_{B_n\p}
  f\p_n)\big).\] That is, $\chuz(\cC,\Omega)$ is the polycategory of
$\Omega$-polarized objects and polarized multivariable adjunctions
described in \cref{sec:intro-mvar} (of the strict sort having exactly
one $(0,0)$-ary morphism, as in \cref{rmk:00}).

In particular, if $\Omega = \Set\in \cat$ and each $A,B$ is of the
form $\rep \cA = (\cA,\cA\op,\hom_\cA)$ (recall \cref{rmk:duality}),
then we can write the functors involved in a morphism
$(A_1,\dots,A_m) \to (B_1,\dots,B_n)$
\begin{align*}
  f\p_j : (\cA_1, \dots, \cA_m, \cB_1\op, \dots, \widehat{\cB_j\op},\dots \cB_n\op) &\longrightarrow \cB_j\\
  (f\m_i)\op : (\cA_1\op , \dots, \widehat{\cA_i\op},\dots \cA_m\op , \cB_1, \dots, \cB_n) &\longrightarrow \cA_i.
\end{align*}
and the clique becomes the family of adjunction isomorphisms
\[ \cA_1(f\m_1(a_2,\dots,a_m,b_1,\dots,b_n),a_1) \cong \cdots \cong
  \cB_n(b_n,f\p_n(a_1,\dots,a_m,b_1,\dots,b_{n-1})). \] Thus the
sub-polycategory of $\chuz(\cat,\Set)$ determined by objects of this
form is the polycategory of multivariable adjunctions.\footnote{We can
  also exclude $\Set$ from $\cat$ for size reasons, allowing the
  latter to consist of only \emph{small} categories, and still have
  $\Omega$ be a non-representable presheaf of categories.}

Now we want to incorporate the 2-categorical structure of \sC into the
output as well, obtaining a 2-polycategory and even a poly double
category.  By a \textbf{2-polycategory} we mean a polycategory
(strictly) enriched over $\cat$, so that the hom-objects
$\sC(\Gamma;\Delta)$ are categories and all operations are functorial.
Similarly, a \textbf{2-multicategory} can be defined as a co-unary
2-polycategory, while a \textbf{2-presheaf} on a 2-multicategory is a
$\cat$-valued presheaf whose actions are 2-functorial.  For instance,
any object $\Omega$ of a 2-multicategory represents a 2-presheaf.  As
in \cref{thm:pshf-varieties}\ref{item:cosubunary}, a 2-multicategory
equipped with a 2-presheaf can equivalently be regarded as a
co-subunary 2-polycategory.

Of course, if \cC is an ordinary multicategory regarded as a locally
discrete 2-multicategory, a 2-presheaf on \cC is just a $\cat$-valued
presheaf.  In particular, $\chuz$ can be regarded as a functor defined
on the category of subunary 2-polycategories with only identity
2-cells between co-unary arrows.

Next, recall that any category $\cA$ is the object-of-objects of a
canonical internal category in $\cat$ whose object-of-morphisms is
$\cA^\dtwo$, the category of arrows in $\cA$.  Put differently, this
is a double category $\dQ(\cA)$ whose vertical and horizontal arrows
are both those of $\cA$, and whose 2-cells are commutative squares.
Similarly, any 2-category $\sC$ can be enhanced to an internal
category $\dQ(\sC)$ in $2\text-\cat$ (a ``cylindrical'' 3-dimensional
structure) whose object-of-morphisms is $\mathit{Lax}(\dtwo,\sC)$ (the
2-category whose objects are arrows of \sC, whose morphisms are
squares in \sC inhabited by a 2-cell, and whose 2-cells are commuting
cylinders in $\sC$).  The underlying double category of this structure
consists of squares or ``quintets'' in $\sC$.

The same idea works for polycategories: any 2-polycategory \sP can be
enhanced to an internal category $\dQ(\sP)$ in the category
$2\text-\poly$ of 2-polycategories.  This gives a 3-dimensional
structure containing:
\begin{itemize}
\item Objects: those of \sP.
\item Horizontal poly-arrows: those of \sP.
\item Horizontal 2-cells between parallel poly-arrows: those of \sP.
\item Vertical arrows: the unary co-unary arrows of \sP.
\item 2-cells of the following shape:
  \[
  \begin{tikzcd}
    (A_1,\dots,A_m) \ar[r,"f"] \ar[d,shift left=7,"h_m"] \ar[d,shift right=7,"h_1"'] \ar[d,phantom,"\cdots"] \ar[dr,phantom,"\Downarrow"] &
    (B_1,\dots,B_n) \ar[d,shift left=7,"k_n"] \ar[d,shift right=7,"k_1"'] \ar[d,phantom,"\cdots"] \\
    (C_1,\dots,C_m) \ar[r,"g"'] &
    (D_1,\dots,D_n)
  \end{tikzcd}
  \]
  coming from 2-cells $k_1 \circ_{B_1} \cdots \circ k_n \circ_{B_n} f \Longrightarrow g \circ_{C_1} h_1 \circ \cdots \circ_{C_m} h_m$ in \sP.
\item ``Poly-cylinders'': commutativity relations in \sP.
  \[
  \begin{tikzcd}
    (A_1,\dots,A_m) \ar[r,shift left=3] \ar[r,shift right=2] \ar[r,phantom,"\Downarrow",shift left=.5] \ar[d,shift left=7] \ar[d,shift right=7] \ar[d,phantom,"\cdots"] \ar[dr,phantom,"\Downarrow"] &
    (B_1,\dots,B_n) \ar[d,shift left=7] \ar[d,shift right=7] \ar[d,phantom,"\cdots"] \\
    (C_1,\dots,C_m) \ar[r] &
    (D_1,\dots,D_n)
  \end{tikzcd}
  \quad=\quad
  \begin{tikzcd}
    (A_1,\dots,A_m)\ar[r] \ar[d,shift left=7] \ar[d,shift right=7] \ar[d,phantom,"\cdots"] \ar[dr,phantom,"\Downarrow"] &
    (B_1,\dots,B_n) \ar[d,shift left=7] \ar[d,shift right=7] \ar[d,phantom,"\cdots"] \\
    (C_1,\dots,C_m)  \ar[r,shift left=3] \ar[r,shift right=2] \ar[r,phantom,"\Downarrow",shift left=.5] &
    (D_1,\dots,D_n)
  \end{tikzcd}
  \]
\end{itemize}
In particular, when \sP is co-subunary, the 2-cells of $\dQ(\sP)$ are all horizontally co-unary or co-nullary:
\[ \begin{tikzcd}
    (A_1,\dots,A_m) \ar[r] \ar[d,shift left=7] \ar[d,shift right=7] \ar[d,phantom,"\cdots"] \ar[dr,phantom,"\Downarrow"] &
    B \ar[d] \\
    (C_1,\dots,C_m) \ar[r] &
    D
  \end{tikzcd}
  \quad\text{or}\quad
  \begin{tikzcd}
    (A_1,\dots,A_m) \ar[r] \ar[d,shift left=7] \ar[d,shift right=7] \ar[d,phantom,"\cdots"] \ar[dr,phantom,"\Downarrow"] &
    () \ar[d,equals] \\
    (C_1,\dots,C_m) \ar[r] &
    ()
  \end{tikzcd}
\]

Now let \sC be a 2-multicategory equipped with a 2-presheaf \Omega, as
before.  Regarding $(\sC,\Omega)$ as a co-subunary 2-polycategory, we
form the internal category $\dQ(\sC,\Omega)$ in $2\text-\poly$, which
is also co-subunary.  Now we forget the nonidentity horizontal 2-cells
between co-unary arrows, obtaining an internal category
$\dQ'(\sC,\Omega)$ in the category of ordinary multicategories
equipped with $\cat$-valued presheaves.  Finally, this latter category
is the domain of the above functor $\chuz$, which preserves pullbacks
and hence internal categories.  Thus, we can define:

\begin{defn}
  The \textbf{double Chu construction} of $(\sC, \Omega)$ is
  \[\dchu(\sC,\Omega) = \chuz(\dQ'(\sC,\Omega)).\]
  It is an internal category in polycategories, which we call a \textbf{poly double category}.
\end{defn}

Tracing through the constructions, we see that $\dchu(\sC,\Omega)$ can
be described more explicitly as follows.

\begin{itemize}
\item Its objects are triples $(A\p,A\m,\uA)$, with $\uA : (A\p,A\m)\to \Omega$.
\item Its horizontal poly-arrows are families of morphisms
  \begin{align*}
    f\p_j : (A_1\p, \dots, A_m\p, B_1\m, \dots, \widehat{B_j\m},\dots B_n\m) &\longrightarrow B_j\p\\
    f\m_i : (A_1\p , \dots, \widehat{A_i\p},\dots A_m\p , B_1\m, \dots, B_n\m) &\longrightarrow A_i\m
  \end{align*}
  equipped with a clique (the ``adjunction isomorphisms'') on
  \[ \big((\uA_1\circ_{A_1\m} f\m_1),\dots,(\uA_m\circ_{A_m\m} f\m_m),(\uB_1 \circ_{B_1\p} f\p_1),\dots, (\uB_n \circ_{B_n\p} f\p_n)\big).\]
\item A vertical arrow $u:A\to B$ is a triple $(u\p,u\m,\uu)$, where $u\p : A\p\to B\p$ and $u\m: A\m\to B\m$ are morphisms in \sC (note that both go in the forwards direction) and
  \[\uu : \uA \Longrightarrow \uB \circ (u\p,u\m)\]
  is a morphism in the hom-category $\sC(A\p,A\m;)$, i.e.\ a 2-cell in \sC.
  This comes from a co-nullary 2-cell in $\dQ'(\sC,\Omega)$.
\item A 2-cell 
  \[
  \begin{tikzcd}
    (A_1,\dots,A_m) \ar[r,"f"]
    \ar[d,shift left=7,"u_m"] \ar[d,shift right=7,"u_1"'] \ar[d,phantom,"\cdots"] \ar[dr,phantom,"\Downarrow\scriptstyle\mu"] &
    (B_1,\dots,B_n) \ar[d,shift left=7,"v_n"] \ar[d,shift right=7,"v_1"'] \ar[d,phantom,"\cdots"] \\
    (C_1,\dots,C_m) \ar[r,"g"'] &
    (D_1,\dots,D_n)
  \end{tikzcd}
  \]
  consists of a family of 2-cells in \sK:
  \[
  \begin{tikzcd}
    (A_1\p, \dots, A_m\p, B_1\m, \dots, \widehat{B_j\m},\dots B_n\m) \ar[r,"f\p_j"{name=f}]
    \ar[d,shift left=28] \ar[d,shift right=28]
    \ar[d,phantom,"{\scriptstyle\cdots (u_1\p,\dots,u_m\p,v_1\m,\dots,\widehat{v_j\m},\dots,v_n\m)\cdots}"]
    % \ar[dr,phantom,"\Downarrow\scriptstyle\mu"]
    &
    B_j\p \ar[d,"v_j\p"] \\
    (C_1\p, \dots, C_m\p, D_1\m, \dots, \widehat{D_j\m},\dots D_n\m) \ar[r,"g\p_j"'{name=g}] &
    D_j\p
    \ar[from=f,to=g,phantom,"\Downarrow\scriptstyle\mu_j\p"]
  \end{tikzcd}
  \]
  and
  \[
  \begin{tikzcd}
    (A_1\p , \dots, \widehat{A_i\p},\dots A_m\p , B_1\m, \dots, B_n\m) \ar[r,"f\m_i"{name=f}]
    \ar[d,shift left=28] \ar[d,shift right=28]
    \ar[d,phantom,"{\scriptstyle\cdots(u_1\p , \dots, \widehat{u_i\p},\dots u_m\p , v_1\m, \dots, v_n\m) \cdots}"]
    % \ar[dr,phantom,"\Downarrow\scriptstyle\mu"]
    &
    A_i\m \ar[d,"u_i\m"] \\
    (C_1\p , \dots, \widehat{C_i\p},\dots C_m\p , D_1\m, \dots, D_n\m) \ar[r,"g\m_i"'{name=g}] &
    C_i\m
    \ar[from=f,to=g,phantom,"\Downarrow\scriptstyle\mu_i\m"]
  \end{tikzcd}
  \]
  such that \emph{any two} of these 2-cells satisfy a commutativity condition relating them to the adjunction isomorphisms of $f,g$ and the structure 2-cells $\uu,\uv$.
  For instance, the condition for $\mu_1\p$ and $\mu_1\m$  is shown in \cref{fig:2-cell-compat}.
  \begin{figure}
    \centering
    \begin{equation*}%\label{eq:2-cell-compat}
  \hspace{-2cm}
  \begin{tikzcd}[column sep=huge,row sep=large]
    (A_1\p , \dots, A_m\p , B_1\m, \dots, B_n\m)
    \ar[r,"{\uA_1\circ (1,f\m_1)}"{name=f2},shift left=4]
    \ar[r,"{\uB_1\circ (1,f\p_1)}"'{name=f},shift right]
    \ar[d,shift left=20] \ar[d,shift right=20]
    \ar[d,phantom,"{\scriptstyle\cdots(u_1\p , \dots, u_m\p , v_1\m, \dots, v_n\m) \cdots}"]
    % \ar[dr,phantom,"\Downarrow\scriptstyle\mu"]
    &
    \Omega \ar[d,equals] \\
    (C_1\p , \dots, C_m\p , D_1\m, \dots, D_n\m)
    \ar[r,"{\uD_1\circ (1,g\p_1)}"'{name=g}]
    &
    \Omega
    \ar[from=f,to=g,phantom,"{\Downarrow\scriptstyle\uv_1\circ (1,\mu_1\p)}"]
    \ar[from=f2,to=f,phantom,"\Downarrow\scriptstyle\cong"]
  \end{tikzcd}
  \quad=\quad
  \begin{tikzcd}[column sep=huge,row sep=large]
    (A_1\p , \dots, A_m\p , B_1\m, \dots, B_n\m) \ar[r,"{\uA_1\circ (1,f\m_1)}"{name=f}]
    \ar[d,shift left=20] \ar[d,shift right=20]
    \ar[d,phantom,"{\scriptstyle\cdots(u_1\p , \dots, u_m\p , v_1\m, \dots, v_n\m) \cdots}"]
    % \ar[dr,phantom,"\Downarrow\scriptstyle\mu"]
    &
    \Omega \ar[d,equals] \\
    (C_1\p , \dots, C_m\p , D_1\m, \dots, D_n\m)
    \ar[r,"{\uC_1\circ (1,g\m_1)}"{name=g},shift left]
    \ar[r,"{\uD_1\circ (1,g\p_1)}"'{name=g2},shift right=4]
    &
    \Omega
    \ar[from=f,to=g,phantom,"{\Downarrow\scriptstyle\uu_1\circ (1,\mu_1\m)}"]
    \ar[from=g,to=g2,phantom,"\Downarrow\scriptstyle\cong"]
  \end{tikzcd}
  \hspace{-2cm}
  \end{equation*}
  \caption{A 2-cell compatibility condition for $\dchu(\sK)$}
  \label{fig:2-cell-compat}
\end{figure}
  % \begin{equation}\label{eq:2-cell-compat}
  % % \hspace{-2cm}
  % \begin{tikzcd}[column sep=huge,row sep=large]
  %   (A_1\p , \dots, A_m\p , B_1\m, \dots, B_n\m) \ar[r,"{\uA_1\circ (1,f\m_1)}"{name=f}]
  %   \ar[d,shift left=20] \ar[d,shift right=20]
  %   \ar[d,phantom,"{\scriptstyle\cdots(u_1\p , \dots, u_m\p , v_1\m, \dots, v_n\m) \cdots}"]
  %   % \ar[dr,phantom,"\Downarrow\scriptstyle\mu"]
  %   &
  %   \Omega \ar[d,equals] \\
  %   (C_1\p , \dots, C_m\p , D_1\m, \dots, D_n\m)
  %   \ar[r,"{\uC_1\circ (1,g\m_1)}"{name=g},shift left]
  %   \ar[r,"{\uC_2\circ (1,g\m_2)}"'{name=g2},shift right=4]
  %   &
  %   \Omega
  %   \ar[from=f,to=g,phantom,"{\Downarrow\scriptstyle\uu_1\circ (1,\mu_1\m)}"]
  %   \ar[from=g,to=g2,phantom,"\Downarrow\scriptstyle\cong"]
  % \end{tikzcd}
  % \quad=\quad
  % \begin{tikzcd}[column sep=huge,row sep=large]
  %   (A_1\p , \dots, A_m\p , B_1\m, \dots, B_n\m)
  %   \ar[r,"{\uA_1\circ (1,f\m_1)}"{name=f2},shift left=4]
  %   \ar[r,"{\uA_2\circ (1,f\m_2)}"'{name=f},shift right]
  %   \ar[d,shift left=20] \ar[d,shift right=20]
  %   \ar[d,phantom,"{\scriptstyle\cdots(u_1\p , \dots, u_m\p , v_1\m, \dots, v_n\m) \cdots}"]
  %   % \ar[dr,phantom,"\Downarrow\scriptstyle\mu"]
  %   &
  %   \Omega \ar[d,equals] \\
  %   (C_1\p , \dots, C_m\p , D_1\m, \dots, D_n\m)
  %   \ar[r,"{\uC_2\circ (1,g\m_2)}"'{name=g}]
  %   &
  %   \Omega
  %   \ar[from=f,to=g,phantom,"{\Downarrow\scriptstyle\uu_2\circ (1,\mu_2\m)}"]
  %   \ar[from=f2,to=f,phantom,"\Downarrow\scriptstyle\cong"]
  % \end{tikzcd}
  % % \hspace{-2cm}
  % \end{equation}
The 2-cells $\mu_j\p,\mu_i\m$ come from co-unary 2-cells in $\dQ'(\sC,\Omega)$, while the ${n+m \choose 2}$ commutativity conditions are a ``clique of commutative cylinders'' therein.
\end{itemize}

We can now quite easily define:

\begin{defn}
  The \textbf{2-Chu construction} of $(\sC, \Omega)$ is the 2-polycategory $\tchu(\sC,\Omega)$ obtained by discarding all the non-identity vertical arrows in $\dchu(\sC,\Omega)$.
\end{defn}

The 1-categorical Chu construction is usually described as a
$\ast$-autonomous category, under suitable conditions on the input
category (closed monoidal with pullbacks).  Our 2-categorical version
has no such conditions on the input, so it produces only a
2-polycategory.  In the presence of suitable structure we expect it to
be a ``$\ast$-autonomous 2-category'', but in order to prove this we
need to define the latter term.  Defining it as a particular kind of
monoidal 2-category would result in numerous tedious coherence axioms,
so instead we take the expected polycategorical characterization as a
definition.

\begin{defn}
  We say that a 2-polycategory $\P$ has \textbf{bicategorical} tensor
  products, units, cotensor products, and counits if they induce
  \emph{equivalences} of hom-categories:
  \begin{align*}
    \P(\Gamma,A\tens B;\Delta) &\simeq \P(\Gamma,A,B;\Delta)\\
    \P(\Gamma,\top;\Delta) &\simeq \P(\Gamma;\Delta) \\
    \P(\Gamma;A\cotens B,\Delta) &\simeq \P(\Gamma;A,B,\Delta) \\
    \P(\Gamma;\bot,\Delta) &\simeq \P(\Gamma;\Delta).
  \end{align*}
  If a $\top$ or $\bot$ only satisfies this property when
  $\abs{\Gamma}+\abs{\Delta}>0$, we call it a \textbf{positive}
  bicategorical unit or counit.  We say that $\P$ has
  \textbf{bicategorical duals} if for any $A$ there are morphisms
  $\eta: () \to (A,A\d)$ and $\ep: (A\d,A)\to ()$ with isomorphisms
  $\ep \circ_A \eta \cong 1_{A\d}$ and
  $\ep\circ_{A\d} \eta \cong 1_A$.\footnote{For a coherent notion of
    duality, these isomorphisms should also satisfy axioms; but we
    will not worry about that, since in a 2-Chu construction these
    isomorphisms are in fact equalities.}
\end{defn}

The positivity condition on units and counits is because our
definition of the $(0,0)$-ary morphisms is ``wrong'', as noted in
\cref{rmk:00}.

\begin{thm}
  If \sC is a 2-multicategory with bicategorical tensor products,
  unit, and homs, and also has bipullbacks\footnote{I.e.\
    bicategorical pullbacks, whose universal property is an
    equivalence of hom-categories.}, and $\Omega$ is an object of \sC,
  then $\tchu(\sC,\Omega)$ has bicategorical tensor products, cotensor
  products, positive unit and counit, and duals.
\end{thm}
\begin{proof}
  As in \cref{thm:chu}, the tensor product of $(A \p,A \m,\uA)$ and $( B\p, B\m,\uB)$ is
  \begin{align*}
    (A \tens B)\p &= (A \p \otimes  B\p)\\
    (A \tens B)\m &= [A \p,  B\m] \times^b_{[A \p\otimes B\p, \Omega]} [ B\p, A \m]
  \end{align*}
  where $\times^b$ denotes the bipullback.  To define
  $\ul{A \tens B}$, we note that now the following two morphisms are
  \emph{isomorphic}
  \begin{equation}
  \begin{array}[c]{c}\scriptstyle
    ([A \p, B\m] \times^b_{[A \p\otimes B\p, \Omega]} [ B\p,A \m], A \p, B\p)
    \to ([A \p, B\m], A \p, B\p)
    \to ( B\m, B\p)
    \xto{\uB} \Omega\\
    \scriptstyle ([A \p, B\m] \times^b_{[A \p\otimes B\p, \Omega]} [ B\p,A \m], A \p, B\p)
    \to ([ B\p,A \m], A \p, B\p)
    \to (A \m,A \p)
    \xto{\uA} \Omega
  \end{array}\label{eq:2tenschu-2}
  \end{equation}
  and determine
  $\ul{A \tens B} : ([A \p, B\m] \times [ B\p,A \m], A \p\otimes B\p)
  \to \Omega$, up to isomorphism, by the universal property of
  $A \p\otimes B\p$.  Its universal morphism
  $(A , B) \to (A \otimes B)$
  \begin{align*}
    (A \p,  B\p) &\to (A \p \otimes  B\p)\\
    ([A \p,  B\m] \times_{[A \p\otimes B\p, \Omega]} [ B\p, A \m], A \p) &\to ([A \p,  B\m],A \p) \to  B\m\\
    ([A \p,  B\m] \times_{[A \p\otimes B\p, \Omega]} [ B\p, A \m],  B\p) &\to ([ B\p, A \m], B\p) \to A \m
  \end{align*}
  plus the isomorphism between the two maps in~\eqref{eq:2tenschu-2}
  and the defining isomorphism of $\ul{A\tens B}$.  For the universal
  property, a morphism $(A \tens B, C) \to D$ in $\tchu(\sC,\Omega)$
  now consists of morphisms in $\sK$:
  \begin{align*}
    f &: (A \p\otimes B\p, C\p) \to D\p\\
    g &: (A \p\otimes B\p,D\m) \to C\m\\
    h &: (C\p,D\m) \to [A \p,  B\m] \times^b_{[A \p\otimes B\p, \Omega]} [ B\p, A \m]
  \end{align*}
  together with a clique on
  \begin{equation}
    \ul{A \tens B} \circ (1,h) \qquad \uC \circ (1,g) \qquad \uD \circ (f,1).\label{eq:2tens-eq-2}
  \end{equation}
  Composing with the universal morphism implements the universal
  properties of $A \p\otimes B\p$ and $[A \p, B\m]$ and
  $[ B\p, A \m]$, yielding an equivalence to the category of
  quadruples
  \begin{alignat*}{2}
    f' &: (A \p, B\p, C\p) \to D\p &\qquad
    h' &: (A \p,C\p,D\m) \to  B\m\\
    g' &: (A \p, B\p,D\m) \to C\m &
    h'' &: ( B\p,C\p,D\m) \to A \m
  \end{alignat*}
  together with a clique corresponding to~\eqref{eq:2tenschu-2}, plus
  an additional \emph{isomorphism} between $\uA\circ (1,h'')$ and
  $\uB\circ (1,h')$ coming from the bipullback.  This yields the
  desired clique on
  \[ \big(\uA\circ (1,h''),\, \uB\circ (1,h'),\, \uC\circ (1,g'),\, \uD\circ (f',1)\big),
  \]
  and hence the desired morphism $(A,B,C)\to D$.  The general case is
  analogous, as is the cotensor product.

  As before, we define the unit by $\top\p = I$ and $\top\m = \Omega$,
  with $\ul{\top} : (I,\Omega) \to \Omega$ induced by the universal
  property of $I$.  Its universal property is straightforward to
  check; the case of $(0,0)$-ary morphisms fails because morphisms
  $\top \to ()$ in $\tchu(\sC,\Omega)$ are equivalent to morphisms
  $I\to \Omega$ in \sC, whereas there is only one morphism $()\to ()$
  in $\tchu(\sC,\Omega)$.
\end{proof}

\begin{rmk}
  If (as in $\cat$) the tensor products, units, and homs in $\sK$
  satisfy a strict universal property, and the bipullbacks are strict
  \emph{iso-comma objects} (not strict pullbacks!), then the tensor
  and cotensor products in $\tchu(\sC,\Omega)$ are again strict.  But
  the unit and counit of $\tchu(\sC,\Omega)$ are not strict even in
  this case.
\end{rmk}

\begin{rmk}
  When we construct a monoidal 2-category from a 2-polycategory, the
  positivity condition should be irrelevant.  That is, once given a
  definition of ``$\ast$-autonomous 2-category'' as a monoidal
  2-category with extra structure, any 2-polycategory with
  bicategorical tensors, cotensors, and duals and \emph{positive}
  bicategorical unit and counit should still have an underlying
  $\ast$-autonomous 2-category.  Moreover, this should give the
  correct ``monoidal'' version of $\tchu(\sC,\Omega)$, despite our
  incorrect definition of the (0,0)-ary morphisms in the
  polycategorical version.
\end{rmk}

Our primary interest is in the case $\sK=\cat$ and $\Omega=\Set$.  In
\cref{sec:cyclic} we will show that $\dchu(\cat,\Set)$ contains the
cyclic multi double category $\madj$ of multivariable adjunctions, by
restricting to the ``representable''\footnote{Another name might be
  ``discrete'', since these are analogous to sets regarded as
  ``discrete Chu spaces'' in $\chu(\Set,2)$.} objects
$\rep{\cA} = (\cA,\cA\op,\hom_\cA)$.  Here we instead mention a few
applications of the full structure $\dchu(\cat,\Set)$.

\begin{eg}\label{eg:polcat}
  Any (poly) double category has an underlying \emph{vertical
    2-category} consisting of the objects, vertical arrows, and
  2-cells whose vertical source and target are identity horizontal
  arrows.  The vertical 2-category of $\dchu(\cat,\Set)$ is isomorphic
  to the 2-category \textsf{PolCat} of \textbf{polarized categories}
  from~\cite{cs:polarized}.  (Since an object of \textsf{PolCat} is by
  definition two categories with a profunctor between them, i.e.\ a
  functor $\mathbf{X}_o\op \times \mathbf{X}_p \to \Set$, this
  isomorphism has to dualize one of the categories.)  The term
  ``polarized'' comes from a logical perspective, with $\cA\p$ and
  $\cA\m$ as the ``positive'' and ``negative'' types that can occur on
  the left or right sides of a sequent\footnote{As noted in
    \cref{rmk:duality}, in most of the paper we consider
    (multivariable) adjunctions to point in the direction of their
    \emph{right} adjoints.  But in \cref{eg:polcat}, \cref{eg:poladj},
    and \cref{eg:conjoints} it is more natural to orient them in the
    other direction.}, and the elements of $\underline{\cA}(A,B)$ as
  the set of sequents $A \vdash B$.
\end{eg}

\begin{eg}\label{eg:poladj}
  The horizontal morphisms of $\dchu (\cat,\Set)$ are not the same as
  the ``inner/outer adjoints'' of~\cite{cs:polarized}, but they are a
  different sensible notion of ``(multivariable) adjunction'' for
  polarized categories.  For instance, just as a horizontal
  pseudomonoid in $\madj$ is a closed monoidal category, a horizontal
  pseudomonoid in $\dchu(\cat,\Set)$ is a natural notion of ``closed
  monoidal polarized category'': it has a tensor product
  $\otimes :\cA\p\times \cA\p \to \cA\p$ and internal-homs
  $\multimap : \cA\p \times \cA\m \to \cA\m$ and
  $\multimapinv : \cA\m \times \cA\p \to \cA\p$ with natural
  bijections between sequents
  \[
    \begin{array}{rcl}
      A_1\otimes A_2 & \vdash &B\\\hline
      A_1 & \vdash &A_2\multimap B\\\hline
      A_2 & \vdash &B\multimapinv A_1.
    \end{array}
  \]
  This allows us to take any of the above sets as a \emph{definition}
  of a set of sequents $A_1,A_2 \vdash B$.  We also have coherent
  associativity isomorphisms of all sorts --- not just
  $(A_1\otimes A_2)\otimes A_3 \cong A_1 \otimes (A_2\otimes A_3)$ but
  also
  $A_1 \multimap (A_2 \multimap B) \cong (A_1\otimes A_2)\multimap B$
  etc.\ (in the polarized case none of these is determined by the
  others) --- giving a consistent definition of a set of sequents
  $A_1,A_2,A_3\vdash B$, and so on for higher arity as well.  (The
  fact that closed monoidal categories are particular pseudomonoids in
  $\dchu(\cat,\Set)$ was observed by~\cite{garner:clmon-2chu}.)
  Similarly, just as it can be shown that a Frobenius pseudomonoid in
  $\madj$ is a $\ast$-autonomous
  category~\cite{ds:quantum,street:frob-psmon,egger:frob-lindist,shulman:frobmvar},
  a Frobenius pseudomonoid in $\dchu(\cat,\Set)$ is a
  ``$\ast$-autonomous polarized category'', with an additional
  ``co-closed monoidal structure'' $\parr$ allowing a consistent
  definition of $A_1,\dots,A_m \vdash B_1,\dots, B_n$ in terms of
  $A_1\otimes \dots\otimes A_m \vdash B_1\parr \dots\parr B_n$.
\end{eg}

\begin{eg}\label{eg:conjoints}
  Intuitively, a polarized category should have ``binary products'' if
  its diagonal functor $A\to A\times A$ has a ``right adjoint''.
  However, as noted in~\cite{cs:polarized}, right adjoints in the
  vertical 2-category \textsf{PolCat} are not the correct notion.  The
  inner/outer adjoints of~\cite{cs:polarized} give one possible
  solution, but the double category $\dchu(\cat,\Set)$ gives another.
  The diagonal $A\to A\times A$ only exists as a \emph{vertical} arrow
  in this double category, but~\cite{gp:double-adjoints} have defined
  a notion of ``adjunction'' between a vertical arrow and a horizontal
  arrow in a double category, called a \textbf{conjunction}.

  In our case, for $A,B\in \dchu(\cat,\Set)$, a ``right conjoint'' of
  a vertical arrow $u : A\to B$ with components $u\p : A\p\to B\p$ and
  $u\m : A\m \to B\m$ consists essentially of an ordinary right
  adjoint $f\m$ to $u\m$ together with a compatible bijection between
  sequents $u\p(\Gamma) \vdash \Delta$ and
  $\Gamma \vdash f\m(\Delta)$.
%   A right conjoint of $u$ is a horizontal arrow $f:B\to A$, with components $f\p:B\p \to A\p$ and $f\m: A\m\to B\m$ and $\uA(a\m,f\p(b\p)) \cong \uB(f\m(a\m),b\p)$, together with
%   \begin{enumerate}
%   \item an isomorphism $f\m \cong u\m$, and
%   \item an adjunction $u\p \dashv f\p$, such that
%   \item the following squares commute:
%     \[
%       \begin{tikzcd}
%         \uA(a\m,a\p) \ar[r,"\eta"] \ar[d,"u"'] &
%         \uA(a\m,f\p(u\p(a\p))) \ar[d,"\cong"]  \\
%         \uB(u\m(a\m),u\p(a\p)) \ar[r,"\cong"'] &
%         \uA(f\m(a\m),u\p(a\p))
%       \end{tikzcd}
% \]\[
%       \begin{tikzcd}
%         \uA(a\m,f\p(b\p)) \ar[r,"\cong"] \ar[d,"u"'] &
%         \uB(f\m(a\m),b\p) \ar[d,"\cong"] \\
%         \uB(u\m(a\m),u\p(f\p(b\p))) \ar[r,"\ep"] & \uB(u\m(a\m),b\p)
%       \end{tikzcd}
%     \]
%   \end{enumerate}
%   If we identify $f\m$ with $u\m$ along the given isomorphism, then such a conjoint consists of an ordinary right adjoint $f\p$ of $u\p$ equipped with an additional ``polarized adjunction isomorphism''$\uA(a\m,f\p(b\p)) \cong \uB(u\m(a\m),b\p)$ that is compatible in a suitable sense.
  In the case when $u:A\to A\times A$ is the diagonal, this means that $A\m$ has binary products in the ordinary sense, and we also have a compatible natural bijection between sequents $\Gamma \vdash \Delta_1\times \Delta_2$ and pairs of sequents $\Gamma\vdash\Delta_1$ and $\Gamma\vdash\Delta_2$.
\end{eg}

\begin{eg}
  Let $k:\cA\to\cB$ be a functor, and write $\rep{\cB_k}$ for the
  object $(\cB,\cA\op,\ul{\cB_k})\in\dchu(\cat,\Set)$ where
  $\ul{\cB_k}(a,b) = \cB(k(a),b)$.  Then a horizontal morphism
  $\rep{\cC}\to \rep{\cB_k}$, with $\rep{\cC} = (\cC,\cC\op,\hom_\cC)$
  representable, is known as a \textbf{relative adjunction}: a pair of
  functors $f:\cA\to\cC$ and $g:\cC\to\cB$ with a natural isomorphism
  $\cC(f(a),b) \cong \cB(k(a),g(b))$.
\end{eg}

\begin{eg}
  For any category $\cA$, we have a ``maximal'' object
  $\floor{\cA} = (\cA,\Set^{\cA},\ev)$ of $\dchu(\cat,\Set)$.  A
  horizontal morphism $\floor{\cA} \to \floor{\cB}$ is just a functor
  $\cA\to\cB$, and similarly a two-variable morphism
  $(\floor{\cA},\floor{\cB})\to\floor{\cC}$ is just a two-variable
  functor $\cA\times\cB\to\cC$.

  If $\cA$ has finite limits, we also have
  $\floorlex{\cA} = (\cA,\Lex(\cA,\Set),\ev)$, where $\Lex(\cA,\Set)$
  denotes the category of finite-limit-preserving functors.  Then a
  horizontal morphism $\floorlex{\cA} \to\floorlex{\cB}$ is equivalent
  to a finite-limit-preserving functor $\cA\to\cB$, but also to a
  finitary right adjoint $\Lex(\cB,\Set) \to \Lex(\cA,\Set)$.  This is
  essentially Gabriel--Ulmer duality~\cite{gu:locpres} for locally
  finitely presentable categories, and generalizes to many other
  doctrines of limits (the maximal case $\floor{\cA}$ corresponds to
  the empty doctrine of no limits).  A two-variable morphism
  $(\floorlex{\cA},\floorlex{\cB})\to\floorlex{\cC}$ is a two-variable
  functor $\cA\times\cB\to\cC$ that preserves finite limits in each
  variable separately.

  Thus, just as the 1-Chu construction gives abstract homes for
  1-categorical concrete dualities like Stone duality and Pontryagin
  duality, the 2-Chu construction gives abstract homes for
  2-categorical concrete dualities like Gabriel--Ulmer
  duality~\cite{pbb:cat-duality}.
\end{eg}

\begin{eg}
  In~\cite{avery:thesis}, objects of the 1-Chu construction
  $\chu(\mathsf{Cat},\cC)$, for an arbitrary category $\cC$, are
  called \textbf{aritations}.  In \S4.7 thereof a structure-semantics
  adjunction is phrased in terms of a universal morphism in
  $\chu(\mathsf{Cat},\cC)$, and in Chapter 5 our
  $\rep\cB \in\chu(\mathsf{Cat},\Set)$ is called the \textbf{canonical
    aritation}.  The possibility of the weaker notion of morphism in
  the 2-Chu construction $\tchu(\cat,\cC)$ (reducing to adjunctions
  between canonical aritations with $\cC=\Set$) is considered in
  \S11.1 of~\cite{avery:thesis}.
\end{eg}

\begin{rmk}
  We can also ``iterate'' the Chu construction in various ways.  For
  instance, from the 2-polycategory $\tmadj$ we can construct
  $\dchu(\tcat,\tmadj)$.  Since the objects of $\tmadj$ are
  categories, every 2-category with its hom-functor yields a
  representable object of $\dchu(\tcat,\tmadj)$.  A horizontal
  morphism between such objects consists of functors $f\p:\cA\to\cB$
  and $f\m:\cB\to\cA$ with an \emph{adjunction}
  $\cA(f\m(b),a) \toot \cB(b,f\p(a))$.  Such \textbf{local
    adjunctions} were studied by~\cite{bp:local-adjointness} in the
  more general context of bicategories and (op)lax functors.
\end{rmk}

\begin{rmk}
  The method of categorifying a construction by applying a
  pullback-preserving functor to internal categories in its domain can
  also be applied to the general construction $\adjcom$ from
  \cref{sec:2cd}.  I do not know whether there are interesting
  examples of ``double Dialectica constructions''.
\end{rmk}

\section{Cyclic multicategories and parametrized mates}
\label{sec:cyclic}

Finally, as promised in \cref{sec:intro-mvar}, we can define the poly
double category of multivariable adjunctions as a subcategory of
$\dchu(\cat,\Set)$.

\begin{defn}\label{defn:madj}
  The poly double category $\madj$ is the sub-double-polycategory of $\dchu(\cat,\Set)$ determined by:
  \begin{itemize}
  \item The objects of the form $\rep\cA = (\cA,\cA\op,\hom_\cA)$ for a category $\cA$.
  \item The vertical arrows of the form $\rep f = (f,f\op,\hom_f)$ for a functor $f:\cA\to\cB$.
  \item All the horizontal arrows and 2-cells relating these.
  \end{itemize}
\end{defn}

We want to compare this with the cyclic multi double category of
multivariable adjunctions from~\cite{cgr:cyclic}.  This requires
making precise the relationship between polycategories and cyclic
multicategories; as suggested in \cref{sec:introduction}, we will show
that cyclic symmetric multicategories are almost equivalent to
polycategories with strict duals.  In fact, there are multiple ways of
defining each of these notions, which we compare with the following
omnibus definition.

\begin{defn}\label{defn:stpoly}
  Let $X\subseteq \dN\times \dN$.
  A \textbf{(symmetric) $X$-ary $\ast$-polycategory} \cP consists of the following:
  \begin{itemize}
  \item A set of objects equipped with a strict involution $(-)\d$, so that $(A\d)\d=A$ strictly.
    If $\Gamma$ is a list of objects, we write $\Gamma\d$ for applying $(-)\d$ to each object in $\Gamma$.
  \item For each pair $(\Gamma, \Delta)$ of finite lists of objects such that $(\abs\Gamma,\abs\Delta)\in X$, a set $\P(\Gamma; \Delta)$ of polyarrows.
  \item For any $\Gamma,\Lambda,\Delta,\Sigma$ and any isomorphism of lists (i.e.\ object-preserving permutation) $\si:\Gamma,\Delta\d \toiso \Lambda,\Sigma\d$, an action
    \( (-)^\sigma : \cP(\Gamma;\Delta) \toiso \cP(\Lambda;\Sigma) \),
    whenever both hom-sets exist, which is functorial on composition of permutations.
  \item Each object $A$ has identities $1^m_A \in \P(A; A)$, $1^l_A \in \P(A\d,A;)$, and/or $1^r_A \in \P(;A,A\d)$, each existing whenever the relevant hom-set does.
    Moreover, any two of these that exist simultaneously are each other's image under the relevant permutation.
  \item For finite lists of objects $\Gamma, \Delta, \Lambda, \Sigma$, and object $A$, composition maps
    \begin{alignat*}{2}
      \circ^m_A &: \P(\Lambda_1, A,\Lambda_2; \Sigma) \times \P(\Gamma; \Delta_1, A,\Delta_2) &&\to \P(\Lambda_1, \Gamma,\Lambda_2 ; \Delta_1, \Sigma,\Delta_2)\\
      \circ^l_A &: \P(\Lambda_1, A,\Lambda_2; \Sigma) \times \P(\Gamma_1,A\d,\Gamma_2; \Delta) &&\to \P(\Lambda_1,\Lambda_2, \Gamma_1,\Gamma_2 ; \Delta, \Sigma)\\
      \circ^r_A &: \P(\Lambda; \Sigma_1,A\d,\Sigma_2) \times \P(\Gamma; \Delta_1, A,\Delta_2) &&\to \P(\Lambda, \Gamma ; \Delta_1,\Delta_2, \Sigma_1,\Sigma_2)
    \end{alignat*}
    in each case presuming that all three hom-sets exist.
    Moreover, any two of these that exist simultaneously are each other's image under the relevant permutations, as are the corresponding composites along $A\d$; in other words any two of the composites
    \begin{mathpar}
      g \circ^m_A f \and g^\rho \circ^r_A f \and g \circ^l_A f^\sigma \and
      (f^\sigma \circ^m_{A\d} g^\rho)^\tau \and (f \circ^r_{A\d} g^\rho)^\tau \and (f^\sigma \circ^l_{A\d} g)^\tau
    \end{mathpar}
    that exist are equal.
  \item Axioms of associativity and equivariance for all choices of $i,j \in \{m,l,r\}$ and whenever both sides make sense and the permutations make everything well-typed:
    \begin{align*}
      1^i_A \circ^i_A f &= f \\
      f \circ^i_A 1^i_A &= f \\
      (h \circ^i_B g) \circ^j_A f &= h \circ^i_B (g \circ^j_A f) \\
      (h \circ^i_B g) \circ^j_A f &= ((h \circ^j_A f) \circ^i_B g)^\sigma \\
      h \circ^i_B (g \circ^j_A f) &= (g \circ^i_A (h\circ^i_B f))^\sigma \\
      g^\rho \circ^i_A f^\sigma &= (g\circ^i_A f)^\tau 
    \end{align*}
  \end{itemize}
  We write $\stpoly X$ for the category of $X$-ary $\ast$-polycategories.
\end{defn}

\begin{eg}
  A $\{(1,1)\}$-ary $\ast$-polycategory is just an ordinary category
  equipped with a strict contravariant involution, since none of the
  $l$ or $r$ data exists.  Even more trivially, a $\{(0,0)\}$-ary
  $\ast$-polycategory is just a set $\cP\zzz$, with no operations.
\end{eg}

\begin{defn}\label{defn:csm}
  We define a \textbf{cyclic symmetric multicategory} to be a
  ``co-unary $\ast$-polycategory'', i.e.\ an $(\dN \times \{1\})$-ary
  one.  To see that this is sensible, note firstly that it ensures
  that none of the $l$ and $r$ data exist.  Thus a
  $(\dN \times \{1\})$-ary $\ast$-polycategory is just a symmetric
  multicategory with a strict involution on its objects and an
  extended action on the homsets $\cP(A_1,\dots,A_n;B)$ indexed by the
  symmetric group $S_{n+1}$.  But $S_{n+1}$ is generated by its two
  subgroups $S_n$ (permuting the first $n$ objects $A_1,\dots,A_n$)
  and $C_{n+1}$ (the cyclic group of order $n+1$, permuting the
  objects cyclically).  The resulting action of $S_n$ is just that of
  a symmetric multicategory, while the action of $C_{n+1}$ says that
  the underlying non-symmetric multicategory of $\cP$ is a
  \emph{cyclic} multicategory in the sense of~\cite{cgr:cyclic}, and
  the relations in $S_{n+1}$ between these subgroups say that the
  symmetric and cyclic structure are compatible in a natural way.
\end{defn}

\begin{eg}
  If $X=\dN\times \dN$, then all the composites and identities exist,
  and each pair of operations $1^i_A$ and $\circ^i_A$ uniquely
  determine the others.  In particular, if we look at $1^m_A$ and
  $\circ^m_A$, we see that an $(\dN\times \dN)$-ary
  $\ast$-polycategory reduces to an ordinary
  \textbf{$\ast$-polycategory} as defined
  in~\cite[\S5.3]{hyland:pfthy-abs}; for emphasis we may call it a
  \textbf{bi-infinitary} $\ast$-polycategory.
\end{eg}

\begin{rmk}
  Note that in a bi-infinitary $\ast$-polycategory, $A\d$ is indeed a
  dual of $A$: the identities $1^l_A$ and $1^r_A$ supply the unit and
  counit of the duality.  Conversely, any polycategory equipped with
  ``strictly involutive duals'' can be made into a
  $\ast$-polycategory.
\end{rmk}

\begin{eg}
  If $X=\dN\times \{0\}$, then \emph{only} $1^l_A$ and $\circ^l_A$
  exist.  Thus an $(\dN\times \{0\})$-ary $\ast$-polycategory may be
  called an \textbf{entries-only $\ast$-polycategory}, by analogy with
  ``entries-only'' cyclic operads (which are the positive-ary
  one-object case) --- since there is no codomain, the objects in the
  domain are simply called ``entries''.
  % This may require updating when the final version of~\cite{dh:dk-cyc-opd} is published.
  In~\cite{dh:dk-cyc-opd-v1}, entries-only $\ast$-polycategories are called ``colored cyclic operads'', but I prefer the terminology of~\cite{gk:cyclic-operads,cgr:cyclic,hry:higher-cyc-opd} whereby ``cyclic multicategories'' and ``cyclic operads'' can be regarded as ordinary multicategories or operads equipped only with the structure of an involution and a compatible cyclic action, rather than additionally with the stuff of an extra hom-set $\cP\zzz$.
\end{eg}

\begin{eg}
  With $X=\{0,1,\dots,n\}\times \{1\}$ (or
  $X=\{0,1,\dots,n+1\}\times \{0\}$ for the entries-only version) we
  obtain \textbf{$n$-truncated} cyclic symmetric multicategories,
  which include $n$-truncated cyclic operads (for truncated operads
  see e.g.~\cite{snpr:modsp-fopd}).
\end{eg}

Of course, if $Y\subseteq X$ we have a functor
$U^X_Y : \stpoly X \to \stpoly Y$ that forgets the morphisms with
undesired arities and the operations relating to them.  As we will now
see, these functors often do not forget very much.

Given a fixed set \cO of objects, let $\iSeq_{X}(\cO)$ be the groupoid
whose objects are pairs $(\Gamma;\Delta)$ of finite lists of elements
of \cO with $(\abs{\Gamma},\abs{\Delta})\in X$, and whose morphisms
are isomorphisms $\Gamma\d,\Delta \toiso \Lambda\d,\Sigma$.  An
inclusion $Y\subseteq X$ yields a fully faithful inclusion
$\iSeq_{Y}(\cO) \into \iSeq_{X}(\cO)$.  By an \textbf{$X$-ary
  collection over \cO} we mean a functor
$\cP : \iSeq_{X}(\cO) \to \Set$; thus an $X$-ary $\ast$-polycategory
consists of an $X$-ary collection over a set of objects together with
identities and composition operations.

\begin{thm}\label{thm:stpoly-eqv}
  If $Y\subseteq X \subseteq \dN\times \dN$ and for any $(m,n)\in X$
  there exists $(k,\ell)\in Y$ such that $k+\ell = m+n$, then the
  forgetful functor $U^X_Y : \stpoly X \to \stpoly Y$ is an
  equivalence.
\end{thm}
\begin{proof}
  The assumption ensures that the corresponding inclusion
  $\iSeq_{Y}(\cO) \into \iSeq_{X}(\cO)$, for any set \cO, is
  essentially surjective and hence an equivalence.  For if
  $(\Gamma;\Delta) \in \iSeq_{X}(\cO)$, with say $\abs{\Gamma}=m$ and
  $\abs{\Delta}=n$ where $(m,n)\in X$, we can choose $(k,\ell)\in Y$
  as in the assumption and find an isomorphism
  $(\Gamma;\Delta)\cong (\Lambda;\Sigma)$ such that $\abs{\Lambda}=k$
  and $\abs{\Sigma}=\ell$, hence $(\Lambda;\Sigma) \in \iSeq_Y(\cO)$.
  Moreover, any hom-set in an $X$-ary $\ast$-polycategory \cP is
  isomorphic to one in $U^X_Y \cP$, and any composition operation in
  \cP is related by the corresponding permutation actions to one that
  exists in $U^X_Y \cP$.  Thus the structure of \cP is uniquely
  determined by that of $U^X_Y \cP$.

  Finally, given a $Y$-ary $\ast$-polycategory \cQ, its underlying
  $Y$-ary collection extends to an $X$-ary one, uniquely up to unique
  isomorphism, and we can use these same permutation actions to define
  the necessary identities and compositions for the latter to be an
  $X$-ary $\ast$-polycategory.  It is straightforward to check that
  the axioms are then satisfied.
\end{proof}

Let $\dN_{>0} = \{m\in\dN \mid m>0 \}$ and
$(\dN\times \dN)_{>0} = \{ (m,n) \in \dN\times\dN \mid m+n>0 \}$.

\begin{cor}
  In the following diagram of forgetful functors:
  \begin{equation}
    \begin{tikzcd}
      \stpoly{\dN \times \{0\}} \ar[d] &
      \stpoly{\dN\times \dN} \ar[l,"\sim"'] \ar[r,"\sim"] \ar[d] &
      \stpoly{\dN\times \{0,1\}} \ar[d] \\
      \stpoly{\dN_{>0}\times \{0\}} &
      \stpoly{(\dN\times \dN)_{>0}} \ar[l,"\sim"'] \ar[r,"\sim"] &
      \stpoly{\dN\times \{1\}}
    \end{tikzcd}\label{eq:stpoly-sqs}
  \end{equation}
  all the horizontal functors are equivalences.
\end{cor}
\begin{proof}
  Each of these inclusions satisfies the hypothesis of \cref{thm:stpoly-eqv}.
\end{proof}

Thus, bi-infinitary $\ast$-polycategories are equivalent to
entries-only $\ast$-polycategories and also to co-subunary ones (i.e.\
$(\dN\times \{0,1\})$-ary ones).  The former equivalence is familiar
from the syntax of classical linear logic, which can be presented
either with two-sided sequents or one-sided ones (although
\emph{right}-sided sequents are more common than left-sided ones,
corresponding to $(\{0\}\times\dN)$-ary $\ast$-polycategories instead
of $({\dN \times \{0\}})$-ary ones).  Co-subunary syntax is less
common, but can be found for instance in~\cite{reddy:acceptors}.

Similarly, the bottom row shows that cyclic symmetric multicategories
(\cref{defn:csm}) are equivalent to positive-ary entries-only ones.
This suggests that arbitrary $\ast$-polycategories could also be
called something like ``augmented cyclic symmetric multicategories''.

\begin{thm}\label{thm:starpoly-01-1}
  Each of the vertical functors in~\eqref{eq:stpoly-sqs} has both a
  left adjoint $L$ and a right adjoint $R$, each of which is fully
  faithful (equivalently, the unit $\Id \to U L$ and counit
  $U R \to \Id$ are isomorphisms).  Moreover, the counit $L U \to \Id$
  and unit $\Id \to R U$ are bijective on objects, and fully faithful
  except on $(0,0)$-ary morphisms.
\end{thm}

This lemma makes \cref{rmk:00} precise: the underlying cyclic
symmetric multicategory of a $\ast$-polycategory remembers everything
but the $(0,0)$-ary morphisms.  The fully faithful right adjoint of
the left-hand vertical functor appears in~\cite{dh:dk-cyc-opd-v1}.

\begin{proof}
  It suffices to consider the right-hand one
  $U:\stpoly{\dN\times \{0,1\}}\to \stpoly{\dN\times \{1\}}$.  To
  start with, since all the structures in question are essentially
  algebraic and $U$ simply forgets some of the data, it preserves
  limits.  Thus, by the adjoint functor theorem for locally
  presentable categories, it has a left adjoint.

  For its right adjoint, we define the homsets of $R\cP$ by right Kan
  extending those of $\cP$ along the inclusion
  $\iSeq_{\dN\times\{1\}}(\cO) \into \iSeq_{\dN\times \{0,1\}}(\cO)$.
  This automatically gives the symmetric actions, with $R\cP\zzz = 1$.
  The only new composition operations we need to define are those
  involving co-nullary morphisms:
  \[ R\P(\Lambda, A; ) \times \P(\Gamma; A) \xto{\circ_A}
    R\P(\Lambda, \Gamma;)
  \]
  Suppose $g\in R\P(\Lambda, A; )\cong \P(\Lambda;A\d)$ and
  $f\in \P(\Gamma; A)$.  If $\abs{\Lambda}>0$, say
  $\Lambda = \Lambda',B$, we can permute $B$ into the codomain of $g$
  and $A$ into its domain, and compose in \cP along $A$ If instead
  $\abs{\Gamma}>0$, say $\Gamma = \Gamma',C$, we can permute $A$ into
  the domain of $f$ and $C$ into the codomain, and compose in \cP
  along $A\d$.
  The unit, equivariance, and associativity axioms follow directly. % as in \cref{thm:starpoly-om-01}.
  The remaining composites to define have the form
  \[\cP(;A\d) \times \cP(;A) \to R\cP\zzz = 1,\]
  so they exist uniquely and all axioms about them are true.

  Evidently $U R \cP \cong \cP$ naturally.  On the other hand, note
  that all of the above definitions were forced except for $R\cP\zzz$
  and the compositions having it as codomain.  Thus, if \cP is given
  as a co-subunary $\ast$-polycategory, it must be isomorphic to
  $R U \cP$ except possibly at $\zzz$.  Since $R U \cP\zzz = 1$ is
  terminal, this ``isomorphism away from $\zzz$'' extends to a unique
  functor $\cP \to R U \cP$ that is, as claimed, bijective on objects
  and fully faithful except on $(0,0)$-ary morphisms.  The triangle
  identities for an adjunction are straightforward.

  Finally, full-faithfulness of $L$ follows from that of $R$ by a
  standard abstract argument, and the fact that $U$ remembers the
  objects and non-$(0,0)$-ary morphisms implies that $L U \to \Id$ is
  also bijective on objects and fully faithful except on $(0,0)$-ary
  morphisms.
\end{proof}

\begin{rmk}
  The $(0,0)$-ary morphisms of $L \cP$ are, as befits a left adjoint,
  ``freely generated'' by all composites $g\circ_A f$ for
  $f\in \cP(;A)$ and $g\in \cP(;A\d) \cong L\cP(A;)$, subject to
  relations imposed to force the necessary associativity axiom.
\end{rmk}

Now I claim that our poly double category $\madj$ is in fact a
\emph{$\ast$-poly double category}, i.e.\ an internal category in
$\ast$-polycategories.  More generally, we have:

\begin{thm}
  If \Omega is a presheaf of $\ast$-polycategories on a multicategory \cC, so is $\adjcom$.
\end{thm}
\begin{proof}
  We take the dual of $(\phi\p,\phi\m,\uphi)$ to be
  $(\phi\m,\phi\p,\uphi\d)$, where $\uphi\d$ is the dual of $\uphi$ in
  the $\ast$-polycategory $\Omega(\Gamma,\phi\m,\phi\p;)$, acted on by
  a symmetry to land in $\Omega(\Gamma,\phi\p,\phi\m;)$.  The
  symmetric action on 2-morphisms in $\adjcom(\Gamma)$ permutes the
  morphisms $f_j$ and $g_i$ and uses the symmetric action on morphisms
  in $\Omega$.
\end{proof}

\begin{cor}
  Any double Chu construction $\dchu(\sC,\Omega)$ is a $\ast$-poly double category.
\end{cor}
\begin{proof}
  Frobenius (pseudo-)discrete polycategories are always naturally $\ast$-polycategories.
\end{proof}

Recall from \cref{defn:madj} that $\madj$ consists of the objects
$\rep\cA = (\cA,\cA\op,\hom_\cA)$ and similar vertical arrows in
$\dchu(\cat,\Set)$.  It is therefore closed under the duality of
$\dchu(\cat,\Set)$, so it is also a $\ast$-poly double category.
Hence it has an underlying cyclic symmetric multi double category,
which we can compare to the cyclic multi double category
of~\cite{cgr:cyclic}.  In~\cite{cgr:cyclic} no symmetric structure was
considered, but we can of course forget the existence of that
symmetric structure and remember only the cyclic one.  This enables us
to finally state the following theorem.

\begin{rmk}
  In fact,~\cite{cgr:cyclic} work directly with $n$-variable
  \emph{mutual left} adjunctions.  Thus, in the language of
  \cref{defn:stpoly}, what their construction yields most directly is
  a positive-ary entries-only (i.e.\ $(\dN_{>0}\times \{0\})$-ary)
  $\ast$-polycategory.  However, to facilitate comparison
  with~\cite{cgr:cyclic} we will likewise use the notation of the
  equivalent $(\dN\times \{1\})$-ary version.
\end{rmk}

\begin{thm}\label{thm:cgr}
  The underlying cyclic multi double category of the $\ast$-poly
  double category $\madj$ is isomorphic to the cyclic multi double
  category constructed in~\cite{cgr:cyclic}.
\end{thm}
\begin{proof}
  For now, let $\madjs$ denote our version and $\madjcgr$ denote
  theirs.  By inspection, the two coincide on objects (categories),
  vertical arrows (functors), and horizontal arrows (co-unary
  multivariable adjunctions).  (Recall in particular that for
  $f:\rep\cA \to \rep\cB$ in $\madjs$, the functor $f\p : \cA\to\cB$
  is the \emph{right} adjoint and $f\m: \cB\op\to \cA\op$ is the
  \emph{left} adjoint.)

  However, the 2-cells of $\madjs$ are, like those in
  $\dchu(\cat,\Set)$, \emph{families} of natural transformations
  $\mu\p_j,\mu\m_i$ related by the axioms such as
  \cref{fig:2-cell-compat}.  Specifically, a 2-cell
    \[
  \begin{tikzcd}
    (A_1,\dots,A_m) \ar[r,"f"]
    \ar[d,shift left=7,"u_m"] \ar[d,shift right=7,"u_1"'] \ar[d,phantom,"\cdots"] \ar[dr,phantom,"\quad\Downarrow\scriptstyle\mu"] &
    B \ar[d,"v"] \\
    (C_1,\dots,C_m) \ar[r,"g"'] &
    D
  \end{tikzcd}
  \]
  consists of a family of natural transformations
  \[
  \begin{tikzcd}
    (A_1, \dots, A_m) \ar[r,"f\p"{name=f}]
    \ar[d,shift left=10] \ar[d,shift right=10]
    \ar[d,phantom,"{\scriptscriptstyle\cdots (u_1,\dots,u_m)\cdots}"]
    % \ar[dr,phantom,"\Downarrow\scriptstyle\mu"]
    &
    B \ar[d,"v"] \\
    (C_1, \dots, C_m) \ar[r,"g\p"'{name=g}] &
    D
    \ar[from=f,to=g,phantom,"\Downarrow\scriptstyle\mu\p"]
  \end{tikzcd}
  \]
  and
  \[
  \begin{tikzcd}
    (A_1 , \dots, \widehat{A_i},\dots A_m , B\op) \ar[r,"f\m_i"{name=f}]
    \ar[d,shift left=20] \ar[d,shift right=20]
    \ar[d,phantom,"{\scriptscriptstyle\cdots(u_1 , \dots, \widehat{u_i},\dots u_m , v\op) \cdots}"]
    % \ar[dr,phantom,"\Downarrow\scriptstyle\mu"]
    &
    A_i\op \ar[d,"u_i\op"] \\
    (C_1 , \dots, \widehat{C_i},\dots C_m , D\op) \ar[r,"g\m_i"'{name=g}] &
    C_i\op
    \ar[from=f,to=g,phantom,"\Downarrow\scriptstyle\mu_i\m"]
  \end{tikzcd}\mathrlap{\qquad (1\le i\le m)}
  \]
  any two of which satisfy a commutativity condition;
  % For instance, the condition for $\mu_1\p$ and $\mu_1\m$  is %$\mu_2\m$
  % \begin{equation}\label{eq:2-cell-compat}
  % % \hspace{-2cm}
  % \begin{tikzcd}[column sep=huge,row sep=large]
  %   (A_1\p , \dots, A_m\p , B_1\m, \dots, B_n\m)
  %   \ar[r,"{\uA_1\circ (1,f\m_1)}"{name=f2},shift left=4]
  %   \ar[r,"{\uB_1\circ (1,f\p_1)}"'{name=f},shift right]
  %   \ar[d,shift left=20] \ar[d,shift right=20]
  %   \ar[d,phantom,"{\scriptstyle\cdots(u_1\p , \dots, u_m\p , v_1\m, \dots, v_n\m) \cdots}"]
  %   % \ar[dr,phantom,"\Downarrow\scriptstyle\mu"]
  %   &
  %   () \ar[d,equals] \\
  %   (C_1\p , \dots, C_m\p , D_1\m, \dots, D_n\m)
  %   \ar[r,"{\uD_1\circ (1,g\p_1)}"'{name=g}]
  %   &
  %   ()
  %   \ar[from=f,to=g,phantom,"{\Downarrow\scriptstyle\uv_1\circ (1,\mu_1\p)}"]
  %   \ar[from=f2,to=f,phantom,"\Downarrow\scriptstyle\cong"]
  % \end{tikzcd}
  % \quad=\quad
  % \begin{tikzcd}[column sep=huge,row sep=large]
  %   (A_1\p , \dots, A_m\p , B_1\m, \dots, B_n\m) \ar[r,"{\uA_1\circ (1,f\m_1)}"{name=f}]
  %   \ar[d,shift left=20] \ar[d,shift right=20]
  %   \ar[d,phantom,"{\scriptstyle\cdots(u_1\p , \dots, u_m\p , v_1\m, \dots, v_n\m) \cdots}"]
  %   % \ar[dr,phantom,"\Downarrow\scriptstyle\mu"]
  %   &
  %   () \ar[d,equals] \\
  %   (C_1\p , \dots, C_m\p , D_1\m, \dots, D_n\m)
  %   \ar[r,"{\uC_1\circ (1,g\m_1)}"{name=g},shift left]
  %   \ar[r,"{\uD_1\circ (1,g\p_1)}"'{name=g2},shift right=4]
  %   &
  %   ()
  %   \ar[from=f,to=g,phantom,"{\Downarrow\scriptstyle\uu_1\circ (1,\mu_1\m)}"]
  %   \ar[from=g,to=g2,phantom,"\Downarrow\scriptstyle\cong"]
  % \end{tikzcd}
  % % \hspace{-2cm}
  % \end{equation}
  whereas an analogous 2-cell in $\madjcgr$ consists \emph{only} of the transformation $\mu\p$.
  Thus, we have a multicategory functor $\madjs \to \madjcgr$ that simply forgets the transformations $\mu\m_i$.

  We now show that this functor preserves the cyclic action.  As
  before, this is obvious except on the 2-cells.  In $\madjs$, the
  cyclic action on 2-cells simply rotates the $\mu\p$ and $\mu\m_i$;
  whereas in $\madjcgr$ the cyclic action is defined by constructing
  mates.  The point is that the compatibility axioms on the 2-cells
  $\mu\p$ and $\mu\m_i$ in $\madjs$ are precisely a way of saying that
  they are each other's mates.  For instance, the condition from
  \cref{fig:2-cell-compat} for $\mu\p$ and $\mu_1\m$ becomes
  % \[\hspace{-3cm}
  % \begin{tikzcd}
  %   A_1(a_1,f_1\m(a_2,\dots,a_m,b_1,\dots,b_n)) \ar[d,"u_1"'] \ar[r,"\cong"] &
  %   A_2(a_2,f_2\m(a_1,a_3,\dots,a_m,b_1,\dots,b_n)) \ar[d,"u_2"] \\
  %   C_1(u_1(a_1),u_1(f_1\m(a_2,\dots,a_m,b_1,\dots,b_n))) \ar[d,"{\mu_1\m}"'] &
  %   C_2(u_2(a_2), u_2(f_2\m(a_1,a_3,\dots,a_m,b_1,\dots,b_n))) \ar[d,"{\mu_2\m}"] \\
  %   C_1(u_1(a_1),g_1\m(u_2(a_2),\dots,u_m(a_m),v_1(b_1),\dots,v_n(b_n)) \ar[r,"\cong"'] &
  %   C_2(u_2(a_2), g_2\m(u_1(a_1),u_3(a_3),\dots,u_m(a_m),v_1(b_1),\dots,v_n(b_n))))
  % \end{tikzcd}
  % \hspace{-3cm}
  % \]
  \[\hspace{-3cm}
    \begin{tikzcd}
    \cA_1(f_1\m(\vec a,b),a_1) \ar[d,"u_1"'] \ar[r,"\cong"] &
    \cB(b,f\p(a_1,\vec a)) \ar[d,"v"]\\
    \cC_1(u_1(f_1\m(\vec a,b)),u_1(a_1)) \ar[d,"{\mu_1\m}"'] &
    \cD(v(b),v(f\p(a_1,\vec a))) \ar[d,"{\mu\p}"] \\
    \cC_1(g_1\m(\vec{ua},v(b)),u_1(a_1)) \ar[r,"\cong"'] &
    \cD(v(b),g\p(u_1(a_1),\vec{ua}))
  \end{tikzcd}
  \hspace{-3cm}
  \]
  % \[\hspace{-3cm}
  % \begin{tikzcd}
  %   A_1(a_1,f_1\m(a_2,\vec a,\vec b)) \ar[d,"u_1"'] \ar[r,"\cong"] &
  %   A_2(a_2,f_2\m(a_1,\vec a,\vec b)) \ar[d,"u_2"] \\
  %   C_1(u_1(a_1),u_1(f_1\m(a_2,\vec a,\vec b))) \ar[d,"{\mu_1\m}"'] &
  %   C_2(u_2(a_2), u_2(f_2\m(a_1,\vec a,\vec b))) \ar[d,"{\mu_2\m}"] \\
  %   C_1(u_1(a_1),g_1\m(u_2(a_2),\vec{ua},\vec{vb}) \ar[r,"\cong"'] &
  %   C_2(u_2(a_2), g_2\m(u_1(a_1),\vec{ua},\vec{vb})))
  % \end{tikzcd}
  % \hspace{-3cm}
  % \]
  where $\vec a = (a_2,\dots,a_m)$ and $\vec{ua} = (u_2(a_2),\dots,u_m(a_m))$.
  The Yoneda lemma implies that this is equivalent to
  \[
  \begin{tikzcd}
    v(b) \ar[r] \ar[d] &
    g\p(g_1\m(\vec{ua},v(b)),\vec{ua}) \ar[d,"\mu\m_1"] \\
    v(f\p(f_1\m(\vec a,b),\vec a))\ar[r,"\mu\p"'] &
    g\p(u_1(f_1\m(\vec a,b)),\vec{ua})
  \end{tikzcd}
  \]
  % \[
  % \begin{tikzcd}
  %   u_2(a_2) \ar[r] \ar[d] &
  %   u_2(f_2\m(f_1\m(a_2,\vec a,\vec b),\vec a,\vec b)) \ar[d]\\
  %   g_2\m(g_1\m(u_2(a_2),\vec{ua},\vec{vb}),\vec{ua},\vec{vb}) \ar[r] &
  %   g_2\m(u_1(f_1\m(a_2,\vec a,\vec b)),\vec{ua},\vec{vb})
  % \end{tikzcd}
  % \]
  If we fix $\vec a$ and write
  \begin{alignat*}{3}
    F\p(a) &= f\p(a,\vec a) &\qquad
    G\p(c) &= g\p(c,\vec{ua}) &\qquad
    U(a) &= u_1(a)\\
    F\m(b) &= f_1\m(\vec a,b) &\qquad
    G\m(d) &= g_1\m(\vec{ua},d) &\qquad
    V(b) &= v(b)
  \end{alignat*}
  then this becomes
  \begin{equation}
  \begin{tikzcd}
    V b  \ar[r] \ar[d] &
    G\p G\m V b \ar[d,"G\p \mu\m"] \\
    V F\p F\m b \ar[r,"\mu\p F\m"'] &
    G\p U F\m b 
  \end{tikzcd}\label{eq:mate-sq}
  \end{equation}
  which is a standard condition characterizing $\mu\p$ and $\mu\m$ as
  mates under the one-variable adjunctions $F\m\dashv F\p$ and
  $G\m\dashv G\p$.  Explicitly, if we apply $G\m$ on the outside and
  postcompose with the counit of $G\m\dashv G\p$, we get
  \[
  \begin{tikzcd}[column sep=large]
    G\m V b  \ar[r] \ar[d] &
    G\m G\p G\m V b \ar[d,"G\m G\p \mu\m"] \ar[r] &
    G\m V b \ar[d,"\mu\m"] \\
    G\m V F\p F\m b \ar[r,"G\m \mu\p F\m"'] &
    G\m G\p U F\m b \ar[r] &
    U F\m b
  \end{tikzcd}
  \]
  where the right-hand square is naturality and the top composite is
  $1_{G\m V b}$ by a triangle identity.  Thus, $\mu\m$ is
  characterized as the left-bottom composite, i.e.\ as a mate of
  $\mu\p$.  We can dually characterize $\mu\p$ as a mate of $\mu\m$;
  while conversely if either is defined as a mate of the other in such
  a way then~\eqref{eq:mate-sq} commutes.

  One does have to check that such a definition is natural in the
  \emph{other} variables, but this was done
  in~\cite[Prop.~2.11]{cgr:cyclic}.  Thus, the functor
  $\madjs\to\madjcgr$ preserves the cyclic action.  Moreover, this
  also shows that it is \emph{faithful} on 2-cells, since all the
  $\mu\m_i$'s are determined as mates of $\mu\p$.

  To show that it is also full on 2-cells, we need to know that if
  $\mu\p$ is given and we define all the $\mu\m_i$'s as its mates, the
  resulting $\mu\m_i$'s satisfy their \emph{own} pairwise conditions
  (\cref{fig:2-cell-compat}), and therefore define a 2-cell in
  $\madjs$.  But this is the content of~\cite[Prop.~2.13 and
  Theorem~2.16]{cgr:cyclic}.  Thus, the functor $\madjs\to\madjcgr$ is
  an isomorphism.
\end{proof}

\begin{cor}\label{thm:madj}
  A 2-cell in $\madj$ is uniquely determined by any one of the transformations $\mu_j\p$ or $\mu_i\m$.\qed
\end{cor}

\cref{thm:madj} is not true of more general 2-cells in
$\dchu(\cat,\Set)$: a transformation between ``polarized adjunctions''
must be ``equipped with specified mates''.

Recall also (\cref{rmk:duality}) that our conventions were chosen to
agree with those of~\cite{cgr:cyclic}, so that a 2-cell $f\to g$ in
$\madj$ is determined by transformations in the \emph{same} direction
between the \emph{right} adjoints $f_i\p \to g_i\p$ and in the
\emph{opposite} direction between the \emph{left} adjoints
$g_j\m \to f_j\m$.  But this is a fairly arbitrary choice.

\begin{rmk}
  Since we chose to ``incorrectly'' give our $\ast$-polycategory
  $\madj$ exactly one $(0,0)$-ary morphism (recall \cref{rmk:00}), it
  happens to be in the image of the right adjoint $R$ from
  \cref{thm:starpoly-01-1}.  Thus, it is $R$ of its underlying cyclic
  symmetric multicategory, which by \cref{thm:cgr} is that
  of~\cite{cgr:cyclic}.  Thus, we could equivalently have constructed
  it by (adding a symmetric action and) applying \cref{thm:cgr} to the
  construction in~\cite{cgr:cyclic}; but the relationship to the Chu
  and Dialectica constructions would then be obscured.
\end{rmk}

\begin{rmk}
  We have focused on multivariable adjunctions between ordinary
  categories and $\dchu(\cat,\Set)$, mainly for simplicity and to
  match~\cite{cgr:cyclic}.  However, multivariable adjunctions exist
  much more generally, e.g.\ for enriched, internal, and indexed
  categories, as well as the ``enriched indexed categories''
  of~\cite{shulman:eicats}; the only requirement is that in the
  enriched cases the enriching category must apparently be symmetric.
  Each of these contexts gives rise to a similar poly double category
  of multivariable adjunctions that embeds into an appropriate double
  Chu construction.

  There ought to be a general theorem encompassing all these cases,
  applying to any 2-category $\sK$ containing an object $\Omega$
  satisfying some sort of ``Yoneda lemma'', but it is not clear
  exactly what this should mean.  Existing contexts for formal Yoneda
  lemmas such
  as~\cite{street-walters:yoneda,street:fib-yoneda-2cat,weber:2toposes,wood:proarrows-i}
  are either too closely tied to the one-variable case, lack a notion
  of ``opposite'', or consider only ``cartesian'' situations at the
  expense of enriched ones.
\end{rmk}


\begin{thebibliography}{SNPR05}

\bibitem[Ave17]{avery:thesis}
Tom Avery.
\newblock {\em Structure and Semantics}.
\newblock PhD thesis, University of Edinburgh, 2017.
\newblock arXiv:1708.01050.

\bibitem[Bar79]{barr:staraut}
Michael Barr.
\newblock {\em $\ast$-autonomous categories}, volume 752 of {\em Lecture Notes
  in Mathematics}.
\newblock Springer, 1979.

\bibitem[Bar91]{barr:staraut-ll}
Michael Barr.
\newblock *-autonomous categories and linear logic.
\newblock {\em Mathematical Structures in Computer Science}, 1(2):159–178,
  1991.

\bibitem[Bar06]{barr:chu-history}
Michael Barr.
\newblock The {Chu} construction: history of an idea.
\newblock {\em Theory and Applications of Categories}, 17(1):10--16, 2006.

\bibitem[Bie08]{biering:dialectica}
Bodil Biering.
\newblock {\em Dialectica interpretations: a categorical analysis}.
\newblock PhD thesis, IT University of Copenhagen, 2008.

\bibitem[BP88]{bp:local-adjointness}
Renato Betti and A.~John Power.
\newblock On local adjointness of distributive bicategories.
\newblock {\em Boll. Un. Mat. Ital. B (7)}, 2(4):931--947, 1988.

\bibitem[CGR14]{cgr:cyclic}
Eugenia Cheng, Nick Gurski, and Emily Riehl.
\newblock Cyclic multicategories, multivariable adjunctions and mates.
\newblock {\em Journal of K-Theory}, 13(2):337–396, 2014.

\bibitem[Chu78]{chu:construction}
Po-Hsiang Chu.
\newblock Constructing $*$-autonomous categories.
\newblock {M.~Sc.}\ thesis, McGill University, 1978.

\bibitem[Chu79]{chu:constr-app}
Po-Hsaing Chu.
\newblock Constructing $\ast$-autonomous categories.
\newblock In {\em $\ast$-autonomous categories}, volume 752 of {\em Lecture
  Notes in Mathematics}, chapter Appendix. Springer-Verlag, 1979.

\bibitem[CKS00]{cks:linbicat}
J.R.B. Cockett, J.~Koslowski, and R.A.G. Seely.
\newblock Introduction to linear bicategories.
\newblock {\em Mathematical Structures in Computer Science}, 2:165--203, 2000.

\bibitem[CKS03]{cks:polybicats}
J.R.B. Cockett, J.~Koslowski, and R.A.G. Seely.
\newblock Morphisms and modules for poly-bicategories.
\newblock {\em Theory and Applications of Categories}, 11(2):15--74, 2003.

\bibitem[CS97a]{cs:pfth-bill}
J.R.B. Cockett and R.A.G. Seely.
\newblock Proof theory for full intuitionistic linear logic, bilinear logic,
  and mix categories.
\newblock {\em Theory and Applications of Categories}, 3(5):85--131, 1997.

\bibitem[CS97b]{cs:wkdistrib}
Robin Cockett and Robert Seely.
\newblock Weakly distributive categories.
\newblock {\em Journal of Pure and Applied Algebra}, 114(2):133--173, 1997.

\bibitem[CS07]{cs:polarized}
J.R.B. Cockett and R.A.G. Seely.
\newblock Polarized category theory, modules, and game semantics.
\newblock {\em Theory and Applications of Categories}, 18(2):4--101, 2007.

\bibitem[CS10]{cs:multicats}
G.S.H. Cruttwell and Michael Shulman.
\newblock A unified framework for generalized multicategories.
\newblock {\em Theory Appl. Categ.}, 24:580--655, 2010.
\newblock arXiv:0907.2460.

\bibitem[Day70]{day:closed}
Brian Day.
\newblock On closed categories of functors.
\newblock In {\em Reports of the Midwest Category Seminar, IV}, Lecture Notes
  in Mathematics, Vol. 137, pages 1--38. Springer, Berlin, 1970.

\bibitem[DCH18]{dh:dk-cyc-opd-v1}
Gabriel~C. Drummond-Cole and Philip Hackney.
\newblock {Dwyer--Kan} homotopy theory for cyclic operads.
\newblock arXiv:1809.06322v1, 2018.

\bibitem[dP89a]{depaiva:dialectica}
Valeria de~Paiva.
\newblock The {Dialectica} categories.
\newblock In J.~Gray and A.~Scedrov, editors, {\em Proc. of Categories in
  Computer Science and Logic, Boulder, CO, 1987}, volume~92 of {\em
  Contemporary Mathematics}. American Mathematical Society, 1989.

\bibitem[dP89b]{depaiva:dialectica-like}
Valeria C.~V. de~Paiva.
\newblock A {Dialectica}-like model of linear logic.
\newblock In David~H. Pitt, David~E. Rydeheard, Peter Dybjer, Andrew~M. Pitts,
  and Axel Poign{\'e}, editors, {\em Category Theory and Computer Science:
  Manchester, UK, September 5--8, 1989 Proceedings}, pages 341--356. Springer
  Berlin Heidelberg, Berlin, Heidelberg, 1989.

\bibitem[dP91]{depaiva:multirel}
Valeria de~Paiva.
\newblock Categorical multirelations, linear logic and {Petri} nets.
\newblock Technical Report 225, University of Cambridge, 1991.

\bibitem[dP06]{paiva:dialectica-chu}
Valeria de~Paiva.
\newblock {Dialectica} and {Chu} constructions: cousins?
\newblock {\em Theory and Applications of Categories}, 17(7):127--152, 2006.

\bibitem[DS97]{ds:monbi-hopfagbd}
Brian Day and Ross Street.
\newblock Monoidal bicategories and {H}opf algebroids.
\newblock {\em Adv. Math.}, 129(1):99--157, 1997.

\bibitem[DS04]{ds:quantum}
Brian Day and Ross Street.
\newblock Quantum categories, star autonomy, and quantum groupoids.
\newblock In {\em Galois Theory, Hopf Algebras, and Semiabelian Categories},
  volume~43 of {\em Fields Institute Communications}, pages 193--231. American
  Math. Soc., 2004.
\newblock arXiv:math/0301209.

\bibitem[Egg10]{egger:frob-lindist}
J.M. Egger.
\newblock The {Frobenius} relations meet linear distributivity.
\newblock {\em Theory and Applications of Categories}, 24(2):25--38, 2010.

\bibitem[EK66]{ek:closed-cats}
Samuel Eilenberg and G.~Max Kelly.
\newblock Closed categories.
\newblock In {\em Proc. Conf. Categorical Algebra (La Jolla, Calif., 1965)},
  pages 421--562. Springer, New York, 1966.

\bibitem[Gar09]{garner:clmon-2chu}
Richard Garner.
\newblock $n$-{Category} {Caf{\'e}} comment: ``{Re}: Monoidal closed categories
  and their deviant relatives''.
\newblock
  \url{https://golem.ph.utexas.edu/category/2009/02/monoidal_closed_categories_and.html#c022449},
  2009.

\bibitem[Gir87]{girard:ll}
Jean-Yves Girard.
\newblock Linear logic.
\newblock {\em Theoretical Computer Science}, 50(1):1-- 101, 1987.

\bibitem[GK95]{gk:cyclic-operads}
E.~Getzler and M.~M. Kapranov.
\newblock Cyclic operads and cyclic homology.
\newblock In {\em Geometry, topology, \& physics}, Conf. Proc. Lecture Notes
  Geom. Topology, IV, pages 167--201. Int. Press, Cambridge, MA, 1995.

\bibitem[G{\"o}d58]{godel:dialectica}
Kurt G{\"o}del.
\newblock \"{U}ber eine bisher noch nicht ben\"{u}tzte {Erweiterung} des
  finiten {Standpunktes}.
\newblock {\em Dialectica}, pages 280--287, 1958.

\bibitem[GP04]{gp:double-adjoints}
Marco Grandis and Robert Pare.
\newblock Adjoints for double categories.
\newblock {\em Cah. Topol. G\'eom. Diff\'er. Cat\'eg.}, 45(3):193--240, 2004.

\bibitem[Gra80]{gray:thc-cc-laxlim}
John~W. Gray.
\newblock Closed categories, lax limits and homotopy limits.
\newblock {\em J. Pure Appl. Algebra}, 19:127--158, 1980.

\bibitem[GU71]{gu:locpres}
P.~Gabriel and F.~Ulmer.
\newblock {\em Lokal pr\"{a}sentierbare Kategorien}, volume 221 of {\em Lecture
  Notes in Mathematics}.
\newblock Springer, 1971.

\bibitem[Gui13]{guitart:trijunctions}
Ren\'{e} Guitart.
\newblock Trijunctions and triadic {Galois} connections.
\newblock {\em Cah. Topol. G\'{e}om. Diff\'{e}r. Cat\'{e}g.}, 54(1):13--27,
  2013.

\bibitem[Hof11]{hofstra:dialectica}
Pieter Hofstra.
\newblock The dialectica monad and its cousins.
\newblock In {\em Models, Logics, and Higher-dimensional Categories: A Tribute
  to the Work of Mih\'{a}ly Makkai}, pages 107--137. American Mathematical
  Society, 09 2011.

\bibitem[Hov99]{hovey:modelcats}
Mark Hovey.
\newblock {\em Model Categories}, volume~63 of {\em Mathematical Surveys and
  Monographs}.
\newblock American Mathematical Society, 1999.

\bibitem[HRY19]{hry:higher-cyc-opd}
Philip Hackney, Marcy Robertson, and Donald Yau.
\newblock Higher cyclic operads.
\newblock {\em Algebraic \& Geometric Topology}, 19:863--940, 2019.

\bibitem[Hyl02]{hyland:pfthy-abs}
J.M.E. Hyland.
\newblock Proof theory in the abstract.
\newblock {\em Annals of Pure and Applied Logic}, 114:43--78, 2002.

\bibitem[KS74]{ks:r2cats}
G.~M. Kelly and Ross Street.
\newblock Review of the elements of $2$-categories.
\newblock In {\em Category Seminar (Proc. Sem., Sydney, 1972/1973)}, volume 420
  of {\em Lecture Notes in Math.}, pages 75--103. Springer, Berlin, 1974.

\bibitem[Law70]{lawvere:comprehension}
F.~William Lawvere.
\newblock Equality in hyperdoctrines and comprehension schema as an adjoint
  functor.
\newblock In {\em Applications of Categorical Algebra (Proc. Sympos. Pure
  Math., Vol. XVII, New York, 1968)}, pages 1--14. Amer. Math. Soc.,
  Providence, R.I., 1970.

\bibitem[Law06]{lawvere:adjointness}
F.~William Lawvere.
\newblock Adjointness in foundations.
\newblock {\em Repr. Theory Appl. Categ.}, 16:1--16 (electronic), 2006.
\newblock Reprinted from Dialectica {\bf 23} (1969).

\bibitem[Pav93]{pavlovic:chu-i}
Du{\v{s}}ko Pavlovi{\'c}.
\newblock Chu {I}: cofree equivalences, dualities, and $\ast$-autonomous
  categories.
\newblock {\em Math. Struct. in Comp. Science}, 11, 1993.

\bibitem[PBB06]{pbb:cat-duality}
Vaughan Pratt, John Baez, and Michael Barr.
\newblock Does duality categorify?
\newblock Discussion on the categories mailing list:
  \url{http://www.mta.ca/~cat-dist/archive/2006/06-4} and
  \url{http://www.mta.ca/~cat-dist/archive/2006/06-5}, 2006.

\bibitem[Red91]{reddy:acceptors}
Uday~S. Reddy.
\newblock Acceptors as values.
\newblock \url{http://www.cs.bham.ac.uk/~udr/}, 1991.

\bibitem[Rie13]{riehl:mon-ams}
Emily Riehl.
\newblock Monoidal algebraic model structures.
\newblock {\em Journal of Pure and Applied Algebra}, 217(6):1069--1104, 2013.

\bibitem[Shu13]{shulman:eicats}
Michael Shulman.
\newblock Enriched indexed categories.
\newblock {\em Theory and Applications of Categories}, 28(21):616--695, 2013.

\bibitem[Shu18]{shulman:contravariance}
Michael Shulman.
\newblock Contravariance through enrichment.
\newblock {\em Theory and Applications of Categories}, 33(5):95--130, 2018.
\newblock arXiv:1606.05058.

\bibitem[Shu19]{shulman:frobmvar}
Michael Shulman.
\newblock $\ast$-autonomous categories are {Frobenius} pseudomonoids.
\newblock In preparation, 2019.

\bibitem[SNPR05]{snpr:modsp-fopd}
F.~Guillén Santos, V.~Navarro, P.~Pascual, and A.~Roig.
\newblock Moduli spaces and formal operads.
\newblock {\em Duke Math. J.}, 129(2):291--335, 08 2005.

\bibitem[Str74]{street:fib-yoneda-2cat}
Ross Street.
\newblock Fibrations and {Y}oneda's lemma in a {$2$}-category.
\newblock In {\em Category Seminar (Proc. Sem., Sydney, 1972/1973)}, pages
  104--133. Lecture Notes in Math., Vol. 420. Springer, Berlin, 1974.

\bibitem[Str04]{street:frob-psmon}
Ross Street.
\newblock Frobenius monads and pseudomonoids.
\newblock {\em J. Math. Phys.}, 45(10):3930--3948, 2004.

\bibitem[SW78]{street-walters:yoneda}
Ross Street and Robert Walters.
\newblock Yoneda structures on 2-categories.
\newblock {\em J. Algebra}, 50(2):350--379, 1978.

\bibitem[Sza75]{szabo:polycats}
M.E. Szabo.
\newblock Polycategories.
\newblock {\em Communications in Algebra}, 3(8):663--689, 1975.

\bibitem[Web07]{weber:2toposes}
Mark Weber.
\newblock Yoneda structures from 2-toposes.
\newblock {\em Appl. Categ. Structures}, 15(3):259--323, 2007.

\bibitem[Woo82]{wood:proarrows-i}
R.~J. Wood.
\newblock Abstract proarrows. {I}.
\newblock {\em Cahiers Topologie G\'eom. Diff\'erentielle}, 23(3):279--290,
  1982.

\end{thebibliography}
\end{document}